\documentclass[12pt,a4paper]{amsart}

\newif\ifpersonal
\usepackage{amsmath,amsthm,amssymb,mathrsfs,mathtools,bm}
\usepackage{microtype,lmodern,textcomp}
\usepackage{enumerate,comment,braket,xspace,tikz-cd}
\usepackage[utf8]{inputenc}
\usepackage[T1]{fontenc}
\usepackage{longtable}
\usepackage[all]{xy}

\usepackage[centering,vscale=0.7,hscale=0.7]{geometry}
\usepackage{graphicx}
\usepackage[hidelinks]{hyperref}
\usepackage[capitalize]{cleveref}

\allowdisplaybreaks
\numberwithin{equation}{section}
\theoremstyle{plain}
\newtheorem{thm}[equation]{Theorem}
\newtheorem{lem}[equation]{Lemma}
\newtheorem{prop}[equation]{Proposition}
\newtheorem{cor}[equation]{Corollary}
\newtheorem{claim}[equation]{Claim}
\newtheorem{assum}[equation]{Assumption}
\theoremstyle{definition}
\newtheorem{defin}[equation]{Definition}
\theoremstyle{remark}

\newtheorem{rem}[equation]{Remark}

\ifpersonal
\newcommand{\personal}[1]{\textcolor[rgb]{0,0,1}{(Personal: #1)}}
\else
\newcommand{\personal}[1]{\ignorespaces}
\fi

\providecommand{\abs}[1]{\lvert#1\rvert}

\newcommand{\bbC}{\mathbb C}

\newcommand{\bbP}{\mathbb P}

\newcommand{\bbR}{\mathbb R}

\newcommand{\bbZ}{\mathbb Z}

\makeatletter
\let\save@mathaccent\mathaccent
\newcommand*\if@single[3]{%
	\setbox0\hbox{${\mathaccent"0362{#1}}^H$}%
	\setbox2\hbox{${\mathaccent"0362{\kern0pt#1}}^H$}%
	\ifdim\ht0=\ht2 #3\else #2\fi
}
\newcommand*\rel@kern[1]{\kern#1\dimexpr\macc@kerna}
\newcommand*\widebar[1]{\@ifnextchar^{{\wide@bar{#1}{0}}}{\wide@bar{#1}{1}}}
\newcommand*\wide@bar[2]{\if@single{#1}{\wide@bar@{#1}{#2}{1}}{\wide@bar@{#1}{#2}{2}}}
\newcommand*\wide@bar@[3]{%
	\begingroup
	\def\mathaccent##1##2{%
		\let\mathaccent\save@mathaccent
		\if#32 \let\macc@nucleus\first@char \fi
		\setbox\z@\hbox{$\macc@style{\macc@nucleus}_{}$}%
		\setbox\tw@\hbox{$\macc@style{\macc@nucleus}{}_{}$}%
		\dimen@\wd\tw@
		\advance\dimen@-\wd\z@
		\divide\dimen@ 3
		\@tempdima\wd\tw@
		\advance\@tempdima-\scriptspace
		\divide\@tempdima 10
		\advance\dimen@-\@tempdima
		\ifdim\dimen@>\z@ \dimen@0pt\fi
		\rel@kern{0.6}\kern-\dimen@
		\if#31
		\overline{\rel@kern{-0.6}\kern\dimen@\macc@nucleus\rel@kern{0.4}\kern\dimen@}%
		\advance\dimen@0.4\dimexpr\macc@kerna
		\let\final@kern#2%
		\ifdim\dimen@<\z@ \let\final@kern1\fi
		\if\final@kern1 \kern-\dimen@\fi
		\else
		\overline{\rel@kern{-0.6}\kern\dimen@#1}%
		\fi
	}%
	\macc@depth\@ne
	\let\math@bgroup\@empty \let\math@egroup\macc@set@skewchar
	\mathsurround\z@ \frozen@everymath{\mathgroup\macc@group\relax}%
	\macc@set@skewchar\relax
	\let\mathaccentV\macc@nested@a
	\if#31
	\macc@nested@a\relax111{#1}%
	\else
	\def\gobble@till@marker##1\endmarker{}%
	\futurelet\first@char\gobble@till@marker#1\endmarker
	\ifcat\noexpand\first@char A\else
	\def\first@char{}%
	\fi
	\macc@nested@a\relax111{\first@char}%
	\fi
	\endgroup
}
\makeatother

\newcommand{\hF}{\widehat F}
\newcommand{\hG}{\widehat G}

\newcommand{\hGamma}{\widehat\Gamma}

\newcommand{\tC}{\widetilde C}

\newcommand{\tX}{\widetilde X}

\DeclareFontFamily{U}{BOONDOX-calo}{\skewchar\font=45 }
\DeclareFontShape{U}{BOONDOX-calo}{m}{n}{<-> s*[1.05] BOONDOX-r-calo}{}
\DeclareFontShape{U}{BOONDOX-calo}{b}{n}{<-> s*[1.05] BOONDOX-b-calo}{}
\DeclareMathAlphabet{\mathcalboondox}{U}{BOONDOX-calo}{m}{n}
\newcommand{\be}{\mathbf e}

\newcommand{\PreLog}{\mathrm{PreLog}}

\newcommand{\Zaffine}{$\bbZ$-affine\xspace}

\newcommand{\inv}{^{-1}}

\DeclareMathOperator{\Coker}{Coker}
\DeclareMathOperator{\Ker}{Ker}
\DeclareMathOperator{\Map}{Map}
\DeclareMathOperator{\Spec}{Spec}

\DeclareMathOperator{\res}{res}
\DeclareMathOperator{\rk}{rk}

\begin{document}
\title[Realization of tropical curves in abelian surfaces]{Realization of tropical curves\\ in abelian surfaces}
\author{Takeo Nishinou}
\address{Takeo Nishinou, Department of Mathematics, Rikkyo University, Nishi-Ikebukuro, Toshima, Tokyo, Japan}
\email{nishinou@rikkyo.ac.jp}
%\author{Tony Yue YU}
%\address{Tony Yue YU, Laboratoire de Math\UTF{00E9}matiques d'Orsay, Universit\UTF{00E9} Paris-Sud, 91405 Orsay, France}
%\email{yuyuetony@gmail.com}
\date{}

\begin{abstract}
We construct algebraic curves in abelian surfaces by utilizing tropical curves in real tori.
We give a necessary and sufficient condition for a tropical curve in a real torus to be realizable by an algebraic curve in an abelian surface.
When this condition is satisfied, the number of algebraic curves can be computed by a combinatorial formula.
This provides the algebraic-tropical correspondence theorem for abelian surfaces analogous to Mikhalkin's correspondence theorem for 
 toric surfaces.
Thus, the number of algebraic curves passing through generic points in an abelian surface can be computed purely combinatorially via tropical curves.
\end{abstract}

\maketitle

%\todo{Change today}
\personal{Personal comments shown!}

\tableofcontents

\section{Introduction}\label{sec:intro_(realization)}

\subsection{Realization problem}
Let $B^{\times}$ be a punctured disc on the complex plane. 
An analytic family of holomorphic curves in $\bbP^n_{\Bbb C}\times B^{\times}$ 
 gives rise to tropical curves in $\bbR^n$, which are balanced piecewise linear weighted graphs, 
 see Section \ref{sec:tropicalization}.
One idea of tropical geometry is to use such combinatorial gadgets to study holomorphic curves (cf.\ \cite{Mikhalkin_Tropical_ICM_2006,Gathmann_Tropical_2006,Itenberg_Tropical_2009}).
However, not all tropical curves arise from holomorphic curves.
The problem of determining whether a tropical curve arises from holomorphic curves is called the \emph{tropical realization problem}, also known as the \emph{tropical lifting problem}.

For toric surfaces, 
 it is a consequence of Mikhalkin's correspondence theorem 
 \cite{Mikhalkin_Enumerative_2005} that all 3-valent tropical curves in $\bbR^2$ are realizable.
The papers \cite{Nishinou_Toric_2006,Tyomkin_Tropical_2012}
 extended the correspondence theorem to general toric varieties in the case of rational curves.
In \cite{Cheung_Faithful_realizability_2014}, these were further generalized and the correspondence theorem was proved
 for regular (but not necessarily 3-valent) tropical curves satisfying certain technical conditions. 
Finally, in \cite{Nishinou_regular_tropical_curves_2025}, it was proved that 
 all regular tropical curves in $\bbR^n$ are realizable.

In all these studies, the crucial assumption is that the tropical curves are regular.
On the side of holomorphic curves, this corresponds to the vanishing of the cohomology group 
 in which the obstructions to deforming the curves lie, so that the obstructions automatically vanish.
As a result, the deformation theory of tropical curves is directly comparable to that of holomorphic curves.

In this paper, we aim to study the realization problem for the case of abelian surfaces, or more generally
 two-dimensional complex tori.
Although these varieties are not toric, they are closely related to toric varieties through 
 Mumford's construction of degenerating families of complex tori \cite[Section 6]{Mumford_Analytic_1972}.
Holomorphic curves in a Mumford family of complex tori give rise to tropical curves in a real torus, 
 see Section \ref{sec:tropicalization}.
It is natural to ask whether a given tropical curve in a real torus arise from a holomorphic curve.
We provide a necessary and sufficient condition for this realization problem in the case of 3-valent
 tropical curves.
Moreover, when the tropical curve is realizable, 
 the number of associated holomorphic curves can be computed using a simple combinatorial formula.

\begin{rem}
The regularity is a transversality condition for tropical curves.
Namely, in general, a tropical curve can be expressed as a solution to a system of linear equations
 (and inequalities to avoid edges of negative length)
 constructed from the data of the places of vertices and the direction of the edges.
The regularity means the space of solutions to this system has the expected dimension.
A formal definition for the regularity was given in \cite[Definition 4.1]{Cheung_Faithful_realizability_2014},
 see also \cite[Lemma 33]{Nishinou_regular_tropical_curves_2025}.
It can be shown that the regularity is also equivalent to the condition that the map $F$
 introduced in Section \ref{sec:multiplicity} has finite cokernel.
In particular, Corollary \ref{cor:rank}, which states that the cokernel of the map $F$ has rank one 
 for a 3-valent tropical curve on the two dimensional torus, implies that such a tropical curve
 is never regular.
\end{rem}

The most serious difficulty in this study is the fact that a tropical curve in a real torus is never regular.
This issue impacts nearly every aspect of the problem,
 leading to a markedly different perspective compared to previous studies cited above.
There are two primary issues to address.
\begin{enumerate}
\item Given a tropical curve on a real torus, there may not be a corresponding degenerate 
 holomorphic curve.
\item Even if there is a degenerate curve, the existence of its first order deformations is not obvious.
If a first-order deformation exists, it is necessary to verify that the obstructions to higher-order deformations vanish.
\end{enumerate}

The realization problem typically begins with the construction of a singular holomorphic curve
 in the central fiber of a degenerating family of a given ambient space. 
This singular curve is designed to reflect the properties of a given tropical curve.
The central fiber is usually a union of toric varieties glued along toric divisors, and 
 the singular curve is constructed by gluing simple local pieces 
 whose geometry is determined by the combinatorics of the vertices of the tropical curve.
The regularity of the tropical curve ensures the transversality in this gluing process.
Therefore,
 the point (1) above did not pose a major problem in the previous studies, 
 at least when the tropical curve is 3-valent.
In particular, for a regular tropical curve, there is always a corresponding singular curve.
In contrast, since tropical curves in this study are not regular, greater care is required during the gluing process.
In fact, it turns out that there are tropical curves for which no corresponding singular curve exists.
Nevertheless, a careful analysis of the gluing conditions leads to
 a simple necessary and sufficient condition for the existence of a corresponding singular curve.
This will be done in Theorem \ref{thm:pre-log}.

After constructing a singular curve on the central fiber of a degenerating family, 
 the next step is to deform it to a general fiber.
If this is possible, since the singular curve reflects properties of the tropical curve, one can relate 
 tropical curves to holomorphic curves on the original ambient space.
Although the central fiber and the curve on it are singular, working with the log smooth deformation theory
 \cite{Kato_Logarithmic_1989,Kato_Log_smooth_deformation_1996},
 the regularity of the tropical curve ensures that the obstructions to deforming the singular curve
 vanish.
Thus, the point (2) was not a concern in the previous studies, either.

In cases where the cohomology group relevant to the obstruction 
 does not vanish, as in our study, addressing this issue is often challenging.
Namely, for each non-negative integer $k$, 
 given a $k$-th order deformation of a map, 
 the obstruction to deforming it one step further depends non-linearly on the lower order terms.
Fortunately, in the present situation we can make use of results in \cite{Nishinou_Obstruction_2018}
 related to the classical notion of semiregularity \cite{Kodaira_Spencer_Semiregularity_1959}.
Theorems in \cite{Nishinou_Obstruction_2018} claim that if an immersion from a reduced curve
 into a smooth complex surface
 satisfies the so-called semiregularity condition, 
 then it is unobstructed in the sense that any first order deformation
 can be extended to arbitrary high orders.
When the surface has a trivial canonical sheaf, then the semiregularity condition is automatically satisfied. 
The problem is that we are dealing with relative situations, 
 the varieties are singular, and furthermore,  
 the map from a curve is not necessarily an immersion.
Nonetheless, the results in \cite{Nishinou_Obstruction_2018} can be applied to our situation with additional calculations.
Namely, we combine the argument in \cite{Nishinou_Obstruction_2018}
 with a trick which reduces the situation to the `standard' one, where 
 the curves are represented by embeddings, see Subsection \ref{subsec:std}.
Here, the assumption (\ref{eq:assump}) below plays an important role.
In this way, we will be able to show the unobstructedness of the deformation.
 The details of this calculation are presented in Subsection \ref{subsec:maincal}.\\

\subsection{Main results}
We now present the precise formulations of our main results 
 (Theorems \ref{thm:intro_realization} and \ref{thm:intro_multiplicity}).
We work in the analytic category, since we will rely on the results in \cite{Nishinou_Obstruction_2018}
 whose proof partially depends on transcendental method.
Let 
\[
\overline{\Lambda}=\begin{pmatrix}
n_{11} & n_{21}\\ n_{12} & n_{22}
\end{pmatrix}
\]
 be a $2\times 2$ integer matrix with non-zero determinant.
Let 
\[
\alpha_{11},\; \alpha_{12},\; \alpha_{21}, \;\alpha_{22}
\]
 be non-zero complex numbers.
We denote by 
\[
k = \Bbb C\{t\}[t^{-1}]
\]
 the quotient field of the ring of convergent series 
 on a disc around the origin 
 of the complex plane. 

Let $B^{\times}$ be a small punctured disc with the fixed parameter $t$
 and consider the product
\[
(\Bbb C^{\times})^2\times B^{\times}.
\]
When regarded as a bundle over $B^{\times}$, 
 the set of its sections naturally forms a group under fiberwise multiplication.
Let $\Lambda$ be a subgroup of this group of sections.
Assume that $\Lambda$ is isomorphic to $\Bbb Z^2$ and forms a discrete subgroup on each fiber.
The quotient of $(\Bbb C^{\times})^2\times B^{\times}/\Lambda$ defines a family $\mathcal X^{\times}$
 of compact surfaces over $B^{\times}$.
In this paper, we
 assume that $\Lambda$ is of the following special form.
 
\begin{equation}\label{eq:assump}
\begin{split}
\Lambda \text{ is generated by }\hspace{2.5in} \\
\lambda_1\coloneqq(\alpha_{11} t^{n_{11}}, \alpha_{12} t^{n_{12}})
\text{ and }
\lambda_2\coloneqq(\alpha_{21} t^{n_{21}}, \alpha_{22} t^{n_{22}}).
\end{split}
\end{equation}

This family can be compactified to a family $\mathcal X$ over the disc $B$, see \cite{Mumford_Analytic_1972}.
Such a compactified family is referred to as a \emph{Mumford family}.
We can also obtain a family of compact surfaces if $\alpha_{ij}t^{n_{ij}}$ is replaced by a more general series,
 such as $\alpha_{ij}t^{n_{ij}}+ \alpha_{ij, 1}t^{n_{ij}+1} + \cdots$.
However, in general, the associated surface may not contain any
 curve.

%Let $X$ be the quotient of $\Gmk^2$ by $\Lambda$.
%\personal{We restrict to the case where $\Lambda$ only consist of monomials, because we would like to extend the action of $\Lambda$ on $\Gmk^2$ to the toric degeneration $\mathcal X$ of $\Gmk^2$ given by any $\overline{\Lambda}$-periodic polyhedral decomposition of $\bbR^2$.}
%It is an algebraic variety over $k$ by Riemann's bilinear relations (cf.\ \cite[Theorem 2.4]{Bosch_Degenerating_1991}).
%Let $X^\an$ be the Berkovich analytification of $X$ (cf.\ \cite{Berkovich_Non-archimedean_first_steps_2008}).

Let $N=\bbZ^2$ and 
\[
S=N_\bbR/\overline{\Lambda}.
\]
Here, by an abuse of notation, $\overline{\Lambda}$ denotes the subgroup of $N$ generated by the column vectors
\begin{equation}\label{eq:oLambda}
\overline{\lambda}_1= \begin{pmatrix}
 n_{11} \\ n_{12}
 \end{pmatrix},\;\;
\overline{\lambda}_2 = \begin{pmatrix}
 n_{21}\\  n_{22}
 \end{pmatrix}.
\end{equation}
We fix an orientation of $N_{\Bbb R}$,
It induces an orientation on $S$.

\begin{rem}\label{rem:troptorus}
$S$ is called a tropical torus in \cite{Mikhalkin_2008} emphasizing the integral affine structure on it
 induced from $N\subset N_{\Bbb R}$.
\end{rem}
%The valuation map $\val^2\colon (\Gmk^2)\to\bbR^2$ induces a map $\tau\colon X\to S$.
%We regard $S$ as the tropicalization of $X$.
%

A \emph{parametrized holomorphic curve}
 in $\mathcal X^{\times}$ consists of a finite field extension $k\hookrightarrow k'$, 
 a smooth projective curve $\mathcal C^{\times}$ over $k'$ and a map 
\[
f\colon \mathcal C^{\times}\to \mathcal X_{k'}^{\times},
\]
 where $\mathcal X_{k'}^{\times}$
 is the base change of $\mathcal X^{\times}$ corresponding to the field extension.
A \emph{parametrized tropical curve} in $S$ consists of a metric graph $\Gamma$ and a \Zaffine immersion  
\[
h\colon\Gamma\to S
\]
 that is balanced at every vertex of $\Gamma$.
See Definition \ref{def:parametrized_tropical_curve_in_R^2} and \ref{def:parametrized_tropical_curve_in_S} for the details.
We assume moreover that the image of every vertex of $\Gamma$ is an integer point in $S$.
Given a parametrized holomorphic curve $f\colon \mathcal C^{\times}
 \to \mathcal X_{k'}^{\times}$, we can associate with it  
 a parametrized tropical curve $h\colon \Gamma\to S$ in $S$, see \cref{sec:tropicalization}.

\begin{defin}\label{def:realizable} 
A parametrized tropical curve $h\colon\Gamma\to S$ is called \emph{realizable}
 if it comes from a parametrized holomorphic curve in $\mathcal X$ via the above tropicalization process.
\end{defin}
Our main goal is to find necessary and sufficient conditions for realizability.\\

Given a parametrized tropical curve $h\colon\Gamma\to S$, we can lift it to a $\overline{\Lambda}$-periodic tropical curve
\[
\widetilde h\colon\widetilde{\Gamma}\to N_\bbR.
\]
Let 
\[
\Delta\subset N_\bbR
\]
 be a parallelogram fundamental domain of the $\overline{\Lambda}$-action with four sides 
\[
B_1,B_2, B_3, B_4
\]
 such that 
\[
\overline{\lambda}_1(B_3)=B_1,\;\;\overline{\lambda}_2(B_4)=B_2.
\]
We assume that $\widetilde h(\widetilde{\Gamma})$ intersects the boundary $\partial\Delta$ transversally, i.e.\ the intersections occur away from the vertices of $\widetilde{\Gamma}$ and the corners of $\partial\Delta$.% (cf.\ \cref{fig:intro}).
This assumption is justified by the translation invariance of the problem, allowing us to make it without loss of generality.

%\begin{figure}[!ht]
%	\centering
%	\setlength{\unitlength}{0.3\textwidth}
%	\begin{picture} (1,0.7)
%	\put(0,0){\includegraphics[width=\unitlength]{figure_intro}}
%	\put(0.78,0.3){\color[rgb]{0,0,0}\makebox(0,0)[lb]{\smash{$B_1$}}}
%	\put(0.7,0.63){\color[rgb]{0,0,0}\makebox(0,0)[lb]{\smash{$B_2$}}}
%	\put(0.22,0.43){\color[rgb]{0,0,0}\makebox(0,0)[lb]{\smash{$B_3$}}}
%	\put(0.2,0.00){\color[rgb]{0,0,0}\makebox(0,0)[lb]{\smash{$B_4$}}}
%	\end{picture}
%	\caption{The fundamental domain $\Delta\subset N_\bbR$.}
%	\label{fig:intro}
%\end{figure}

Let 
\[
\widetilde e_1,\dots,\widetilde e_l\;\; (\text{resp.}\ \widetilde f_1,\dots,\widetilde f_m)
\]
 be the edges of $\widetilde{\Gamma}$ whose images intersect $B_1$ (resp.\ $B_2$).
Let 
\[
(a_i,b_i)\;\; (\text{resp.}\ (c_j,d_j))
\]
 be the weight vector (i.e.\ derivative) of $\widetilde h$ at 
 $\widetilde e_i$ (resp.\ $\widetilde f_j$) pointing from the inside to the outside of $\Delta$.

The \emph{weight} $W_e$ of an edge $e$ of $\Gamma$ is the multiplicity of the weight vector.
Let $\delta$ be the greatest common divisor of the weights of all edges of $\Gamma$.
The \emph{weight} $W_v$ of a 3-valent vertex $v$ of $\Gamma$ is the norm of the cross product of the weight vectors of any two of the three edges connected to $v$.
It is independent of the choice by the balancing condition.

We introduce the following quantity
\begin{align*}
\sigma&\coloneqq\prod_{i=1}^l (\alpha_{12}^{a_i/\delta}\alpha_{11}^{-b_i/\delta})\cdot\prod_{j=1}^m(\alpha_{22}^{c_j/\delta}\alpha_{21}^{-d_j/\delta}).
%\eta&\coloneqq\sum_{i=1}^l (n_{12}a_i-n_{11}b_i) + \sum_{j=1}^m (n_{22}c_j-n_{21}d_j).
\end{align*}
It follows from the balancing condition that $\sigma$ does
 not depend on the choice of the fundamental domain $\Delta\subset N_\bbR$.

\begin{thm}\label{thm:intro_realization}
	Let $h\colon\Gamma\to S$ be a 3-valent parametrized tropical curve in $S$.
	It is realizable by a parametrized holomorphic curve in $\mathcal X^{\times}$ if and only if the 
		quantity $\sigma$ is equal to $(-1)^{\sum W_v/\delta}$.
\end{thm}

After this paper has been posted, 
 Blomme \cite[Lemma 4.6]{Blomme_Enumeration_of_curves_1} has shown that the exponent 
 $\sum W_v/\delta$ is always even, implying that $(-1)^{\sum W_v/\delta}$ is in fact equal to one,
 see Remark \ref{rem:exponent}.
Given Theorem \ref{thm:intro_realization}, 
 it is natural to ask how many parametrized holomorphic curves in $\mathcal X^{\times}$ correspond to a given realizable parametrized tropical curve in $S$.
To obtain a finite count, it is necessary to impose constraints that reduce the dimension of the associated moduli spaces.

Let $h\colon\Gamma\to S$ be a 3-valent realizable parametrized tropical curve of genus $g$ 
 passing through $g$ integer points 
\[
p_1,\dots,p_g\in S.
\]
\begin{rem}\label{rem:expdim}
Note that given an immersion $\phi\colon C\to X$ from an irreducible
 nonsingular curve to a two-dimensional complex torus,
 the map $\phi$ deforms in a $g$-dimensional family, where $g$ is the genus of $C$.
The sketch of proof is as follows.
First, from the fact that the canonical sheaf of $X$ is trivial, 
 adjunction shows that the normal sheaf $\mathcal N_{\phi}$ of the map $\phi$ is isomorphic to the canonical sheaf of $C$. 
In particular, $\dim H^0(C, \mathcal N_{\phi}) = g$.
Suppose we have constructed a $k$-th order deformation $\phi_k$ of $\phi$ for some positive integer $k$.
An infinitesimal parallel transportation provides a $(k+1)$-th order deformation of $\phi$ extending $\phi_k$.
In particular, $\phi_k$ is unobstructed and its extension to $(k+1)$-th order maps composes a $g$-dimensional family.
Repeating this, we obtain a $g$-dimensional family of formal deformations of $\phi$, and applying 
 implicit function theorem \cite[Proposition 1.5]{Kosarew_1991}, we obtain a $g$-dimensional analytic family.
On the other hand, a 3-valent parametrized tropical curve on $S$ also has $g$-dimensional deformations,
 where $g$ is the genus of the graph $\Gamma$, see \cite[Proposition 3.20]{Blomme_Enumeration_of_curves_1}.
Therefore, in both cases, the counting curves with $g$ point constraints make sense.
\end{rem}

Assume that the preimage $h\inv(p_i)$ 
 lies in the interiors of the edges of $\Gamma$.
Furthermore, assume the constraints make the tropical curve rigid.
Choose a point $v_i$ from $h^{-1}(p_i)$ for each $i = 1, \dot, g$.
Let $\Gamma'$ denote the graph $\Gamma$ subdivided by $v_i$.

\begin{rem}
As we will see in Remark \ref{rem:double_edge}, 
 even if the constraints $p_1, \dots, p_g$ are generic, there are tropical curves where $h^{-1}(p_i)$
 consists of multiple points.
In this case, there are multiple edge $e_1, \dots, e_l$ whose images by $h$ contain $p_i$.
In particular, the images of these edges overlap.
Since the constraints $p_1, \dots, p_g$ make the tropical curve rigid, this implies that this overlap
 cannot be resolved by deforming the map $h$.
The point $v_i$ can lie on any of these edges.
\end{rem}

Let 
\[
x_1,\dots,x_g
\]
 be 
 $k$-rational points in $\mathcal X^{\times}$ such that, 
 possibly after a finite base field extension, they can be extended to sections of $\mathcal X\to B$, 
 and each $x_i$ specializes to a general point on the irreducible component of the central fiber of $\mathcal X$
 corresponding to $p_i$.
See Section \ref{sec:multiplicity} for a construction of such sections.

We define a map
\begin{align*}
G\colon \Map(V(\Gamma'),N)
&\xrightarrow{\mathmakebox[3ex]{}} \bigoplus_{e \in E(\Gamma')} (N/N_{e})\ \oplus\ \bigoplus_{i=1}^g N\\
\phi &\xmapsto{\mathmakebox[3ex]{}}
\Big(\big(\phi(\partial^+e)-\phi(\partial^-e)\big)_e,\big(\phi(v_i)\big)_i\Big)
\end{align*}
where $V(-)$ denotes the set of vertices, $E(-)$ denotes the set of edges, $N_e$ denotes the sublattice generated by the weight vector of $e$, and $\partial^- e$, $\partial^+ e$ denote the two endpoints of an edge $e$ according to a fixed orientation.
Consider the tensor product of $G$ with $\bbC^*$ over $\bbZ$:
\[G_{\bbC^*}\colon \Map\big(V(\Gamma'),N_{\bbC^*}\big)
\longrightarrow \bigoplus_{e \in E(\Gamma')} (N/N_{e})_{\bbC^*} \oplus\ \bigoplus_{i=1}^g N_{\bbC^*}.\]

\begin{thm}\label{thm:intro_multiplicity}
	The number of parametrized holomorphic
    curves in $\mathcal X^{\times}$ (up to equivalence given by base field extension) passing through $x_1,\dots,x_g$ that give rise to the 3-valent realizable parametrized tropical curve $h\colon\Gamma\to S$ is equal to the product
	\[\abs{\Ker G_{\bbC^*}}\cdot\prod_{e\in E(\Gamma')} W_e.\]
\end{thm}

On the other hand, given a family of holomorphic curves in $\mathcal X^{\times}$
 passing through a set of general sections $x_1,\dots,x_g$ of
 $\mathcal X^{\times}\to B^{\times}$, it induces a tropical curve on $S$ by 
 the construction in \cref{sec:tropicalization}.
This tropical curve passes through a set of $g$ general integer points on $S$, and then it must be 3-valent, 
 see Proposition \ref{prop:rev}.
By combining this fact with Theorem \ref{thm:intro_realization},
 we obtain the following.
\begin{cor}
Given general constraints 
\[
p_1, \dots, p_g\in S
\]
 and 
\[
x_1, \dots, x_g\in \mathcal X^{\times},
\] 
 the number of 3-valent realizable parametrized tropical curves of genus $g$ passing through $p_1, \dots, p_g$, 
 counted with multiplicity 
\[
\abs{\Ker G_{\bbC^*}}\cdot\prod_{e\in E(\Gamma')} W_e,
\]
 is equal to
 the number of parametrized holomorphic curves passing through $x_1, \dots, x_g$.\qed
\end{cor}

In terms of the lattice index, the number $\abs{\Ker G_{\bbC^*}}$ is equal to the number of elements of the cokernel of the 
 map
\[
\Map(V(\Gamma'),N)
 \xrightarrow{\mathmakebox[3ex]{}} 
 {\rm Im}\, G_{\Bbb R}\bigcap\left(\bigoplus_{e \in E(\Gamma')} (N/N_{e})'\ \oplus\ \bigoplus_{i=1}^g N\right)
\]
 induced by $G$, where $G_{\Bbb R} = G\otimes\Bbb R$ and $(N/N_{e})'$ is the torsion free part of 
 $N/N_{e}$.
Here, ${\rm Im}\, G_{\Bbb R}$ has codimension one in 
\[
\left(\bigoplus_{e \in E(\Gamma')} (N/N_{e})'\ \oplus\ \bigoplus_{i=1}^g N\right)\otimes \Bbb R.
\]
On the other hand, as mentioned in the related works below, 
 Blomme \cite[Theorem 4.9]{Blomme_Enumeration_of_curves_1}
 has established a formula for the multiplicity
 in terms of weights of vertices.
%For readers more inclined to complex geometry, note that our abelian surface $X$ over $k$ can be extended to a family of complex tori over a punctured disc.
%Hence, by Artin's approximation theorem \cite{Artin_Algebraic_Approximation_1969}, the parametrized algebraic curves can also be extended over a possibly smaller punctured disc.

\medskip
\subsection{Related works}\label{subsec:relatedworks}

The realization problem in tropical geometry has been studied extensively in the literature.
Besides the works \cite{Mikhalkin_Enumerative_2005,Nishinou_Toric_2006,Tyomkin_Tropical_2012,Nishinou_Correspondence_2009} already mentioned in the beginning, various interesting results have been obtained in the works by Shustin \cite{Shustin_A_tropical_approach_2005}, Speyer \cite{Speyer_Uniformizing_2007},
Brugall\'{e}-Shaw \cite{Brugalle_Obstructions_to_approximating_2011},
Katz-Payne \cite{Katz_Payne_Realization_2011},
Katz \cite{Katz_Lifting_2012},
Bogart-Katz \cite{Bogart_Obstructions_to_lifting_2012},
Birkmeyer-Gathmann-Schmitz \cite{Gathmann_Realizability_2013}.
We would also like to mention the work of Halle-Rose \cite{Halle_Tropical_count_of_curves_on_abelian_varieties} regarding the enumeration of tropical curves in tropical abelian varieties.		
For the counting problem of Gromov-Witten type, there is a beautiful result by Bryan and Leung \cite{Bryan_Generating_1999}.
Contrary to Gromov-Witten type invariants, our counting is sensitive to the complex structure.
Thus, it gives a combinatorial counting of curves on the family of complex tori satisfying the condition (\ref{eq:assump}) above.
For another example of an 
 obstructed deformation problem of curves on surfaces, see, for example, \cite{Tannenbaum_1982}.

Blomme has developed the enumerative aspects of the results of this paper, 
 \cite{Blomme_Enumeration_of_curves_3,Blomme_Enumeration_of_curves_2,Blomme_Enumeration_of_curves_1}.
As we mentioned above, 
 he gave a formula for the multiplicity in Theorem \ref{thm:intro_multiplicity} as a product of suitable weights 
 of vertices, similar to the one in \cite{Mikhalkin_Enumerative_2005}.
Based on it, he obtained a tropical geometric proof of results of 
 \cite{Bryan_Generating_1999}, and their generalization.
Thus, he revealed the relation between tropical curves on real two dimensional torus and 
 the so-called reduced Gromov-Witten invariants, which was introduced in \cite{Bryan_2000} to define
 counting problems on Calabi-Yau surfaces (the usual Gromov-Witten invariant for these surfaces always vanishes).
These are results about so-called curves with primitive homology classes.
In \cite{Bryan_Counting_on_abelian_surface_2018, Oberdieck_GW_theory_2022}, 
 enumerative study of curves of general classes on Abelian surfaces was initiated.
In \cite{Blomme_Enumeration_of_curves_3,Blomme_Enumeration_of_curves_2}, Blomme also showed that
 tropical method can recover these results as well.

\begin{rem}
One can also study the realization problem for curves in higher-dimensional abelian varieties by tropical method.
There are several issues that need to be addressed in this study.
Obviously, we need to prove the analogue of Theorem \ref{thm:pre-log}, and to
 extend the results of \cite{Nishinou_Obstruction_2018} to higher codimensional cases.
However, the most significant issue will be the superabundance of tropical curves.
For tropical curves in tropical tori (see Remark \ref{rem:troptorus}), all of them are superabundant in the sense that 
 their parameter spaces have larger dimensions than expected, even in the case of 
 tropical curves in two-dimensional tori.
For example, when the target space is a three-dimensional tropical torus or
 three-dimensional abelian variety, the expected dimension of a curve is always zero, 
 whereas there are obvious three-dimensional freedom of translation.

On the other hand, plane tropical curves are regular, 
 that is, their parameter spaces have the expected dimension.
However, when the target space has dimension three or higher, then superabundant tropical curves
 generically appear even when the target space is an affine space.
Thus, curves in higher-dimensional ambient spaces will have large obstruction spaces, which makes the
 deformation theory more intricate.
\end{rem}

%We present the tropicalization procedure via Berkovich geometry, which we regard as the natural context for the problem.
%We refer to \cite{Baker_Nonarchimedean_2011,Gubler_Guide_2013,Yu_Balancing_2013,Yu_Tropicalization_2014} for more discussions on tropicalizations.

\subsection{Notation}\label{subsec:notation}
We will work in the complex analytic category.
In this paper, we will study non-constant maps $\varphi_0\colon C_0\to X_0$ from a curve $C_0$
 to a complex analytic surface $X_0$ and their deformations.
Here, the domain $C_0$ is a reduced, connected and nodal curve.
The surface $X_0$ is also reduced, and it is the central fiber of a family of surfaces $\mathcal X$ over a disc $B$
 (a Mumford family, see \cite[Corollary 6.6]{Mumford_Analytic_1972} and comments after it).
A deformation of $\varphi_0$ over $\Spec\Bbb C[t]/(t^{n+1})$ will be written as 
\[
\varphi_n\colon C_n\to X_n = \mathcal X\times_B\Spec\Bbb C[t]/(t^{n+1}).
\]
By the image of a map $\varphi_0$ or $\varphi_n$, we mean the analytic locally ringed space
 with the annihilator structure, see \cite[Chapter I, Definition 1.45]{Greuel_Lossen_Shustin_book2007}.
That is, if $U$ is an open subset of $C_n$ with the induced structure of an analytic locally ringed space, 
 and $V$ is an open subset of $X_n$ such that 
 $\varphi_n(U)$ is closed in $V$, we associate the structure sheaf 
\[
\mathcal O_V/\mathcal{A}nn_{\mathcal O_V}((\varphi_n)_*\mathcal O_U)
\]
 to the image $\varphi_n(U)$.

\subsubsection{Glossary}
Here we compile the notations which are frequently used in this paper.

\begin{longtable}{p{4cm}p{7cm}}
%\begin{tabular}{p{3cm}p{9cm}}
${\Gamma}$ & A metric graph (see Definition \ref{def:metric_graph}).\\
$V_{\infty}(\Gamma)$ & The set of infinite vertices of $\overline{\Gamma}$ (see Definition \ref{def:metric_graph}).\\
$V(\Gamma), E(\Gamma)$ & The set of vertices and edges of $\Gamma$ (see Definition \ref{def:metric_graph}).\\
$N$, $N_{\Bbb R}$ & A finitely generated free abelian group, and 
 the vector space associated with it by extension of scalars.
 In this paper, $N$ is almost always of rank two.
We fix an orientation of $N_{\Bbb R}$.\\
$\overline{\Lambda}$ & A full rank subgroup of $N$.\\
$S = N_{\Bbb R}/\overline{\Lambda}$ & A real torus.\\
$h\colon \Gamma\to S$ & A parametrized tropical curve in $S$, also denoted by $(\Gamma, h)$.
 Here, $\Gamma$ does not have infinite vertices (see Definition \ref{def:parametrized_tropical_curve_in_S}).\\
$w_{v, e}$ & The weight vector of the edge $e$ at $v$ of a parametrized tropical curve $(\Gamma, h)$
 (see Definition \ref{def:parametrized_tropical_curve_in_R^2}).\\
$W_e$ & The scalar weight of the edge $e$ defined as the positive integer $W_e$
 such that $w_{v, e}/W_e$ is a primitive vector of $N$ (see Definition \ref{def:weight}).\\
$W_v$ & The weight of the vertex $v$ (see Definition \ref{def:weight}).\\
$\mathcal D$ & A polyhedral subdivision of $S$ whose 1-skeleton contains the image $h(\Gamma)$.\\
$\mathcal D^0, \mathcal D^1$ & The set of 0-cells and 1-cells of $\mathcal D$, respectively.\\
$\mathcal X\to B$ & A Mumford family of degenerating complex tori of dimension two over a small disc $B$.
 It is determined by $\mathcal D$ (see Definition \ref{def:M1}).\\
$X_0$ & The central fiber of $\mathcal X$.\\
$X_{0, \rho}$ & The irreducible component of $X_0$ corresponding to $\rho\in \mathcal D^0$.\\
$X_{0, \sigma}$ & The codimension one stratum of $X_0$ corresponding to $\sigma\in\mathcal D^1$.\\
$\widetilde{\mathcal D}$ & The $\overline{\Lambda}$-periodic polyhedral subdivision of $N_{\Bbb R}$ induced by $\mathcal D$ (see Definition \ref{def:M2}).\\
$\widetilde{\mathcal X}$ & A toric variety over $\Spec\Bbb C[t]$ of infinite type associated with $\widetilde{\mathcal D}$.
 $\mathcal X$ is a quotient of the restriction of $\widetilde{\mathcal X}$ over $B\subset \Spec\Bbb C[t]$
 by a lattice $\Lambda$ of rank two (see Definition \ref{def:M2}).\\
$\varphi_0\colon C_0\to X_0$ & A pre-log curve in $X_0$ associated with a parametrized tropical curve
 $h\colon\Gamma\to S$ (see Definition \ref{def:pre-log}).\\
$C_{0, v}$ & The irreducible component of $C_0$ corresponding to $v\in V(\Gamma)$.\\

$\{\mathscr U_i\}_{i\in I}$ & An open covering of $C_0$ with certain properties (see Definition \ref{def:scrU}).\\ 
$\Psi_{\mathcal X}, \Psi_{X_0}$ & The normalized relative complex volume form of $\mathcal X$ and
 its restriction to $X_0$ (see Definition \ref{def:volumeform}).\\
$\Psi_{C_0}$ & The pullback of $\Psi_{X_0}$ to $C_0$ by $\varphi_0$.\\
$(\Gamma_s, h_s)$ & The standard two-vertex tropical curve (see Definition \ref{def:std}).\\
 $V_1$, $V_2$ & The two vertices of $\Gamma_s$.\\
 $E$ & The edge of $\Gamma_s$ connecting $V_1$ and $V_2$.\\
 $\Gamma_{v_1, v_2}$ & The graph associated with an open neighborhood of the edge $e_{v_1v_2}$ connecting 
adjacent vertices $v_1$ and $v_2$ of $\Gamma$.\\
 $(\Gamma_{v_1v_2}, h_{v_1v_2})$ &  The two-vertex tropical curve associated with the 
 restriction of $h$ to $\Gamma_{v_1, v_2}$.
 As graphs, $\Gamma_s$ and $\Gamma_{v_1, v_2}$ are identified.\\
$\Phi, \Phi_{\Bbb R}$ & The integral linear map and its extension of scalars
 satisfying $\Phi_{\Bbb R}\circ h_s = h_{v_1v_2}$.\\
$\overline{\mathcal C}_s$ & The closed cone in $\Bbb R^3$ associated with the standard two-vertex tropical 
 curve $(\Gamma_s, h_s)$.\\
$\overline{\mathcal C}_{v_1v_2}$ & The analogue of $\overline{\mathcal C}_s$
 for $(\Gamma_{v_1v_2}, h_{v_1v_2})$.\\
$\overline{\mathcal Y}_{V_1V_2}^s$ & The toric variety associated with the cone $\overline{\mathcal C}_s$.
It is a degenerating family of toric varieties over $\Bbb C$ (see Definition \ref{def:mathcalY}).\\
$\mathcal Y_{V_1V_2}^s$ & An open toric subvariety of 
 $\overline{\mathcal Y}_{V_1V_2}^s$ (see Definition \ref{def:mathcalY}).\\
$\overline{Y}_{0, V_1V_2}^s$ & The central fiber of $\overline{\mathcal Y}_{V_1V_2}^s$ (see Definition \ref{def:Y}).\\
$Y_{0, V_1V_2}^s$ &  The central fiber of $\mathcal Y_{V_1V_2}^s$ (see Definition \ref{def:Y}).\\
$\overline{\mathcal Y}_{v_1v_2}$ & The toric variety associated with the cone
$\overline{\mathcal C}_{v_1v_2}$ (see Definition \ref{def:mathcalY}).\\
$\mathcal Y_{v_1v_2}$ & The analogue of $\mathcal Y_{V_1V_2}^s$. 
 It is a subvariety of $\mathcal X$ (see Definition \ref{def:mathcalY}).\\
$\overline{Y}_{0, v_1v_2}$ & The central fiber of $\overline{\mathcal Y}_{v_1v_2}$ (see Definition \ref{def:Y}).\\
$Y_{0, v_1v_2}$ &The central fiber of $\mathcal Y_{v_1v_2}$.
The varieties $\overline{Y}_{0, v_1v_2}$ and $Y_{0, v_1v_2}$ are subvarieties of $\mathcal X$ (see Definition \ref{def:Y}).\\
$P_{\Phi}$ & The toric map $\overline{\mathcal Y}_{V_1V_2}^s\to 
 \overline{\mathcal Y}_{v_1v_2}$ associated with $\Phi$.
 Its restriction $\mathcal Y_{V_1V_2}^s\to \mathcal Y_{v_1v_2}$ is also denoted by $P_{\Phi}$
 (see Definition \ref{def:Pphi}).\\
$\phi_{0, v_1v_2}\colon D_0\to \overline{Y}_{0, v_1v_2}$ & A pre-log curve of type $(\Gamma_{v_1v_2}, h_{v_1v_2})$
 (see Lemma \ref{lem:v_1v_2}).
 $D_0$ is the chain of rational curves with length two (i.e., it has two irreducible components and one node).\\
$\psi_{0, V_1V_2}\colon D_0\to \overline{Y}_{0, V_1V_2}^s$ & A pre-log curve of type $(\Gamma_s, h_s)$
 which satisfies $\phi_{0, v_1v_2} = P_{\Phi}\circ \psi_{0, V_1V_2}$ (see Lemma \ref{lem:v_1v_2}).
 We denote the restriction $\psi_{0, V_1V_2}|_{\mathscr D_{v_1v_2}}$ also by $\psi_{0, V_1V_2}$.\\
$\mathscr D$, $\mathscr D_{v_1v_2}$
 & An affine open subset of $D_0$ isomorphic to $\Spec\Bbb C[s, u]/(su)$ or 
 $\Spec\Bbb C[s, \frac{1}{s-\alpha}, u]/(su)$, 
 where $s$ and $u$ are coordinates of the branches of $D_0$ that vanish at the node, 
 and $\alpha$ is a non-zero complex number
 (see Definitions \ref{def:scrD} and \ref{def:Dv1v2}).\\
$\{\mathscr W_k\}_{k\in I'}$ & An open covering of $X_0$, whose index set $I'$ contains the index set $I$
 of the covering $\{\mathscr U_i\}_{i\in I}$ of $C_0$, such that for each $i\in I$, 
 $\varphi_0(\mathscr U_i)$ is an
 analytic subset of $\mathscr W_i$ (see Definition \ref{def:scrW}).\\
$\mathscr U_{i, n}$ & Given a deformation $C_n$ of $C_0$ over $\Spec\Bbb C[t]/(t^{n+1})$, the restriction of the structure
 of an analytic space to the open subset $\mathscr U_i$ of $C_0$.
 See Definition \ref{def:u_inw_in}.\\
$\mathscr W_{i,n}$ & The restriction of the structure of an analytic space of 
 $X_n = \mathcal X\otimes_{\Spec\Bbb C[t]}\Spec\Bbb C[t]/(t^{n+1})$ to the open subset 
 $\mathscr W_i$ of $X_0$.
 See Definition \ref{def:u_inw_in}.\\
$\mathscr W_{k, V_1V_2}^s$ & Given $\mathscr W_k\subset X_0$ and 
 $P_{\Phi}|_{Y_{0, V_1V_2}^s}\colon Y_{0, V_1V_2}^s\to X_0$
 whose image contains $\mathscr W_k$, 
 the connected component of 
  $P_{\Phi}|_{Y_{0, V_1V_2}^s}^{-1}(\mathscr W_k)$ containing the image
  of $\psi_{0, V_1V_2}|_{\mathscr U_k}$. See \cref{eq:scrW}.\\
$\mathscr W_{k, n, V_1V_2}^s$ & The restriction of the structure of an analytic space of 
 $\mathcal Y_{V_1V_2}^s\otimes_{\Spec\Bbb C[t]}\Spec\Bbb C[t]/(t^{n+1})$ to
 $\mathscr W_{k, V_1V_2}^s$.
 In particular, $\mathscr W_{k, V_1V_2}^s = \mathscr W_{k, 0, V_1V_2}^s$.
See \cref{eq:scrW2}.\\
$\mathscr U_{i, X}, \mathscr U_{i, Y}$ & Irreducible components of $\mathscr U_i$ when it contains a node of $C_0$.
See \cref{eq:scrUXY}.\\
$Y_{X, V_1V_2}^s, Y_{Y, V_1V_2}^s$ & Irreducible components of $Y_{0, V_1V_2}^s$.
See \cref{eq:scrWYX} and \cref{eq:scrWYY}.\\
$\mathscr W_{i, X, V_1V_2}^s, \mathscr W_{i, Y, V_1V_2}^s$ & 
Irreducible components of $\mathscr W_{i, V_1V_2}^s$ when it is the union of product of discs.
See \cref{eq:scrWYX} and \cref{eq:scrWYY}.\\
$\psi_{n, i}$ & The restriction of $\psi_{n, V_1V_2}$ to $\mathscr U_i$.
See Subsection \ref{subsec:maincal}.
%\end{tabular}
\end{longtable}

\medskip
\paragraph{\bfseries Acknowledgements}
We are very grateful to Pierrick Bousseau and Tony Yue Yu for inspiring discussions.
This paper is originally conceived as a joint project with Tony Yue Yu.
There are many places where he made invaluable contributions, especially in the proof of Theorem \ref{thm:pre-log}
 and in the most part of Section \ref{sec:multiplicity}.
We would like to deeply thank him.
We also thank the reviewers for their detailed comments.
The author is supported by JSPS KAKENHI Grant Number 18K03313.

\section{Tropicalization of curves in complex tori}\label{sec:tropicalization}

In this section, we give precise definitions concerning tropical curves and describe the tropicalization procedure for curves 
 in complex tori.
We use the setup in the introduction.

\begin{defin} \label{def:metric_graph}
	A \emph{metric graph} $\Gamma$ is a metric space whose underlying topological space is a graph, and the metric on each edge is modeled on an interval $[0, l]$ for $l\in(0,+\infty]$.
	\begin{itemize}
	\item A vertex $v$ of $\Gamma$ is called an infinite vertex if any neighborhood of $v$ has infinite length.
	 Otherwise, it is called a finite vertex.
	\item We assume that all the infinite vertices are 1-valent and all vertices have finite valence.
	\end{itemize}
The graph $\Gamma$ can be an infinite graph, that is, $\Gamma$ may have infinite number of vertices and edges
 (e.g. the honeycomb graph in the Euclidean plane).
	We use the following notations:
	\begin{itemize}
	\item $V(\Gamma)$: The set of vertices of $\Gamma$.
	\item $V_\infty(\Gamma)$: The set of infinite vertices of $\Gamma$.
	\item $E(\Gamma)$: The set of edges of $\Gamma$.
	\end{itemize}
	For each vertex $v$ of $\Gamma$, we denote by $e_{v,1}, e_{v,2}, \dots$ the edges of $\Gamma$ connected to $v$.
	For each edge $e$ of $\Gamma$, we denote by $e^\circ$ its interior, i.e.\ the edge minus its two endpoints.
\end{defin}

\begin{defin} \label{def:parametrized_tropical_curve_in_R^2}
	A \emph{parametrized tropical curve in $N_\bbR$} consists of a metric graph $\Gamma$ and a continuous immersion $h\colon \Gamma\setminus V_\infty(\Gamma)\to N_\bbR$ satisfying the following conditions:
	\begin{enumerate}
	\item The image of every finite vertex in $\Gamma$ is an integer point.
	\item For each finite vertex $v$ of $\Gamma$ and each edge $e$ connected to $v$, the restriction $h|_{e^\circ}$ is linear with integer derivative $w_{v,e}\in N$, defined with respect to the unit tangent vector pointing from $v$ to $e$.
	We call $w_{v,e}$ the \emph{weight vector} of the edge $e$ at $v$.
	\item (Balancing condition) For each finite vertex $v$, we require that
	\[\sum_{e\ni v} w_{v,e}=0\]
	holds, 
	where the sum is over all edges connected to $v$.
	\end{enumerate}
	For	 ease of notation, we will denote the parametrized tropical curve above by $h\colon\Gamma\to N_\bbR$, omitting $V_\infty(\Gamma)$ unless explicitly required.
	The \emph{genus} of a parametrized tropical curve is the first Betti number of the graph $\Gamma$.
\end{defin}
\begin{rem}
In some definitions of tropical curves, the map $h$ may be allowed to contract some edges of $\Gamma$ to points.
In such cases, the balancing condition must be adjusted accordingly.
In this paper, since all tropical curves are immersions, we adopt the above definition.
\end{rem}

\begin{defin} \label{def:parametrized_tropical_curve_in_S}
	A \emph{parametrized tropical curve in $S=N_\bbR/\overline{\Lambda}$} consists of a metric graph $\Gamma$ without infinite vertices whose underlying graph is finite, and a continuous immersion $h\colon\Gamma\to S$ satisfying the same conditions as in \cref{def:parametrized_tropical_curve_in_R^2}.
	%The fact that $h\colon\Gamma\to S$ is an immersion implies that $\Gamma$ does not contain any infinite vertices.
\end{defin}
In particular, all edges of $\Gamma$ have finite length in this case.

\begin{defin} \label{def:weight}
	For an edge $e$ with an endpoint $v$ of a parametrized tropical curve, the \emph{scalar weight}, or simply
	 the \emph{weight}, of $e$, 
	denoted by $W_e$, is the multiplicity of the integral vector $w_{v,e}$.
	For any vertex $v$, we denote by $\gamma_v$ the greatest common divisor of the weights of all the edges connected to $v$.
	For a 3-valent vertex $v$, the \emph{weight} of $v$, denoted by $W_v$, is the norm of the cross product of the weight vectors of any two of the three edges connected to $v$.
	It is independent of the choice of the two edges by the balancing condition.
	For a 2-valent vertex $v$, we define its weight $W_v$ to be 1.
\end{defin}

A tropical curve in a real two-dimensional torus can be obtained from a suitable family of 
 holomorphic curves in a degenerating family of complex tori.  
This procedure is not the main focus of this paper, as the purpose here is 
 to construct a family of holomorphic curves from a given tropical curve.  
Therefore, only a brief outline is provided. 
%For further details, see \cite[Section 6]{Nishinou_Toric_2006}.  

Consider a Mumford family $\mathcal X \to B$ of degenerating complex tori of dimension two, 
 as described in \cite[Corollary 6.6]{Mumford_Analytic_1972} (see also the comments after it).  
Here, $B$ denotes a small disc in $\mathbb{C}$ centered at $0 \in \mathbb{C}$.  
The Mumford family is associated with a polyhedral decomposition of a real two-dimensional torus.  
Its central fiber, $X_0$, is a reduced union of toric surfaces, glued along their toric divisors.  

Suppose there is a family of irreducible holomorphic curves
\[
\psi^{\times}\colon\mathcal C^{\times}\to \mathcal X^{\times}
\]
 defined over the punctured disc $B^{\times} = B\setminus\{0\}$.
We assume that the map $\psi^{\times}$ is generically injective, so that the image is reduced, 
 which is the case relevant to us.
The image 
\[
\mathcal I^{\times} = \psi^{\times}(\mathcal C^{\times})
\]
 is an analytic subvariety of $\mathcal X$, and
 by taking its closure, we obtain a family $\mathcal D$ of curves in $\mathcal X$ over $B$.
Explicitly, for any $p\in X_0$, take a neighborhood $U_p$ of $p$ in $\mathcal X$ so that 
 $U_p\cap \mathcal I^{\times}$ is defined by an analytic function over $k = \Bbb C\{t\}[t^{-1}]$.
By multiplying by a power of $t$ if necessary, we can assume that the defining function contains
 no negative power of $t$, 
 and that its constant term (with respect to $t$) is non-zero.
Then, $\mathcal I\cap U_p$ is defined by taking $t = 0$ in the defining function.
Since different defining functions differ only by invertible functions, the analytic structure of 
\[
I_0 = \mathcal I\cap X_0
\]
 is well-defined.

%In particular, it gives an analytic subvariety $D_0$ on the central fiber $X_0$.
By \cite[Proposition 6.2]{Nishinou_Toric_2006}, after performing toric blow-ups of $\mathcal X$ if necessary, 
 $I_0$ can be made disjoint from the torus fixed points of $X_0$.
This process is well-defined because every point of $X_0$ has a neighborhood in $\mathcal X$ 
 that is canonically isomorphic to an analytic open subset of a toric variety, making toric blow-ups applicable.  
The toric blow-ups of $\mathcal X$ induce a refinement $\mathcal D$ of 
 the original polyhedral decomposition of the real two-dimensional torus.  
Then, each irreducible component of $I_0$ gives a torically transverse curve in an irreducible component of $X_0$
 which is a toric surface, and its intersection multiplicities with the toric divisors 
 determine a tropical curve with a single finite vertex whose underlying graph is a subgraph of 
 $\mathcal D$.
These pieces can be glued into a global tropical curve on the real torus.

\begin{rem}
In this construction, the curve $I_0$ can be highly singular, non-reduced, and may have 
 geometric genus lower than that of the curve in the generic fiber.
Also, it may have singular points on the toric divisors of the components of $X_0$.
Nevertheless, the above construction gives a tropical curve on the torus.
This is because the weights of edges and the balancing condition at vertices
 are determined by topological data as follows.

Let $p\in I_0$ be a point on a toric divisor $Z$ 
 of an irreducible component of $X_0$.
Let $X_{0, 1}$ and $X_{0, 2}$ be the irreducible components of $X_0$ such that $X_{0, 1}\cap X_{0, 2} = Z$.
Let 
\[
I_{0, i_1}, \dots, I_{0, i_k}
\]
 and
\[
I_{0, j_1}, \dots, I_{0, j_l}
\]
 be the irreducible components 
 of $I_0$ contained in $X_{0, 1}$ and $X_{0, 2}$, respectively.
We consider them with the analytic structure induced from that of $I_0$.
In particular, we put appropriate multiplicities to them if they are non-reduced.
Now, consider a neighborhood $U_p$ of $p$ in $\mathcal X$ and take its intersection with the fiber $X_t$ of 
 $\mathcal X$ over $t\in B$, where $|t|$ is small but not zero.
Taking $U_p$ appropriately, the set $U_p\cap X_t$ is homeomorphic to 
\[
\mathscr C\times \Delta,
\] 
 where
\[
\mathscr C = S^1\times (-1, 1)
\]
 is a cylinder and $\Delta$ is a disc.
On the other hand, the intersection $X_0\cap U_p$ is homeomorphic to 
\[
\mathscr E\times \Delta,
\] 
 where $\mathscr E$ is the union of two cones formed by collapsing the circle $S^1\times \{0\}$
 in $\mathscr C$ to a point, see Figure \ref{fig:cone}.
Moreover, $X_0\cap U_p$ is a deformation retract of $U_p$, and
 the set $\{e\}\times \Delta$, where $e$ is the apex of $\mathscr E$, is an open subset of the toric divisor 
 $Z$. 
Let $a$ be the homology class of the intersection of $\mathcal I^{\times}$ with $X_t\cap U_p$, where 
\[
H_1(X_t\cap U_p, \Bbb Z)\cong H_1(\mathscr C, \Bbb Z) \cong \Bbb Z.
\] 
This homology class $a$ does not depend on $t$ when $|t|$ is sufficiently small.
There may be several connected components in $\mathcal I^{\times}\cap (X_t\cap U_p)$, but it does not affect the
 argument.
By choosing the identification $H_1(X_t\cap U_p, \Bbb Z)\cong \Bbb Z$
 appropriately, $a$ is equal to the sum of the 
 intersection multiplicities of the components $I_{0, i_1}, \dots, I_{0, i_k}$ with $Z$ at $p$, 
 and it is also equal to the corresponding sum associated with the components 
 $I_{0, j_1}, \dots, I_{0, j_l}$.
%In particular, if $X_{0, a}$ and $X_{0, b}$ are the components of $X_0$ containing $D$, 
% then $C_0\cap X_{0, a}$ and $C_0\cap X_{0, b}$ are both nontrivial curves and have
% the same intersection multiplicity at $p$.
This shows that the weights of the edges of the associated tropical curve are well-defined.
Moreover, the balancing condition is the requirement needed for the claim that each component of $I_0$ 
 is a cycle without boundary.
 
\begin{figure}
\includegraphics[height=12cm]{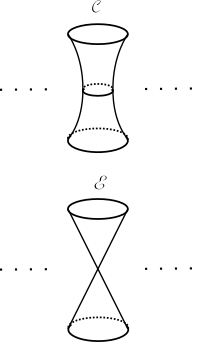}
\caption{Upper figure: An illustration of $U_p\cap X_t \cong \mathscr C\times \Delta$ for $t\neq 0$.
It retracts to $\mathscr C$.
If $r$ denotes the retraction from $U_p\cap X_t$ to $\mathscr C$, the composition of 
 the inclusion $\iota\colon \mathcal I\cap U_p\cap X_t\to U_p\cap X_t$ with $r$ is homotopic to a covering map of 
 $\mathscr C$ for small $t\neq 0$.
The degree of the covering is $a$.
Lower figure:  An illustration of $U_p\cap X_0 \cong \mathscr E\times \Delta$.
By taking the limit $t\to 0$, 
 the map $r\circ \iota$ degenerates to a map which is
 homotopic to a branched covering map
 of $\mathscr E$ of degree $a$ relative to the boundary 
 (that is, it maps the boundary of the domain to the boundary of $\mathscr E$).
This degree $a$ contributes to the weight of the corresponding edge of the tropical curve.
The total weight of the edge is the sum of such $a$ over all intersections of $I_0$ with
 the divisor $Z = X_{0, 1}\cap X_{0, 2}$.
% sum of the local intersection multiplicity
% between $C_{0, i_1}, \dots, C_{0, i_k}$ with $D = X_{0, 1}\cap X_{0, 2}$ at $p$.
}\label{fig:cone}
\end{figure} 
 
\end{rem}

\section{Construction of pre-log curves in the special fiber} \label{sec:pre-log}

The goal of this section is to study the existence of pre-log curves associated with a given 3-valent parametrized tropical curve $h\colon\Gamma\to S$.
The main result is Theorem \ref{thm:pre-log}.

Without loss of generality, we impose the following assumptions on $\Gamma$ for later purpose. 
\begin{assum}\label{assum:1}
\begin{enumerate}
\item $\Gamma$ is connected.
\item The lattice length of the image of each edge $e$ of $\Gamma$
 is an integer multiple of its weight $W_e$.
\end{enumerate}
\end{assum}
The condition $(2)$ is equivalent to stating that the length of 
 each edge of the metric graph $\Gamma$ is a positive integer.

Let $\mathcal D$ be a polyhedral subdivision of $S$ whose 1-skeleton includes the image $h(\Gamma)$.
We denote by $\mathcal D^0$ and $\mathcal D^1$ the set of 0-cells and 1-cells of $\mathcal D$, respectively.
Since finite base field extensions rescale the lattice $N\subset N_\bbR$, we can assume, 
 after a finite base field extension, that $\mathcal D$ is integral, i.e.\ $\mathcal D^0$ consists solely of lattice points.
The preimage $h\inv(\mathcal D^0)$ induces a subdivision of $\Gamma$, which we still denote by $\Gamma$.
Note that, under this subdivision, $\Gamma$ may now contain 2-valent vertices.
Here, we summarize notations concerning Mumford's family of degenerating tori.

\begin{defin}\label{def:M1}
The subdivision $\mathcal D$ induces a degeneration $\mathcal X$ of a complex torus over a small disc $B$
 (\cite[Section 6]{Mumford_Analytic_1972}).
For a non-negative integer $n$, let 
\[
X_n\coloneqq\mathcal X\times_B\Spec\Bbb C[t]/(t^{n+1}),
\] 
 where $t$ is the coordinate on the disc.
In particular, $X_0$ is the special fiber of $\mathcal X$.
For $\rho\in\mathcal D^0$, let $X_{0,\rho}$ be the irreducible component of $X_0$
 corresponding to $\rho$, which is a toric variety over $\bbC$.
We denote its toric boundary by $\partial X_{0,\rho}$.
For $\sigma\in\mathcal D^1$, let $X_{0,\sigma}$ denote the corresponding 1-stratum of $X_0$.
\end{defin}

\begin{defin}\label{def:M2}
The subdivision $\mathcal D$ also induces a $\overline{\Lambda}$-periodic polyhedral subdivision $\widetilde{\mathcal D}$ of $N_\bbR$.
It induces a toric variety $\widetilde{\mathcal X}$ over $\Spec\Bbb C[t]$ of infinite type, whose restriction
 to the disc $B$ is a covering of $\mathcal X$ with the covering group $\Lambda$.
Let $\tX_0$ denote the special fiber of $\widetilde{\mathcal X}$.
\end{defin}

Now, we introduce degenerate curves on $X_0$.

\begin{defin} \label{def:pre-log}
	A \emph{pre-log curve} in $X_0$ associated with the parametrized tropical curve $h\colon\Gamma\to S$ consists of a proper holomorphic
    curve $C_0$ over $\bbC$ and a map $\varphi_0\colon C_0\to X_0$ satisfying the following conditions:
	\begin{enumerate}
	\item The curve $C_0$ has at worst nodal singularities.
	\item The dual intersection graph of $C_0$ is isomorphic to $\Gamma$.
	\item All irreducible components of $C_0$ are isomorphic to $\bbP^1_\bbC$.
	\item For all $v\in V(\Gamma)$, we have 
	\[
	\varphi_0(C_{0,v})\subset X_{0,h(v)},
	\]
	 where $C_{0,v}$ is the irreducible component of $C_0$ corresponding to $v$.
	\item For all $e\in E(\Gamma)$, we have 
	\[
	\varphi_0(p_e)\in X_{0,h(e)}\setminus \partial X_{0,h(e)},
	\]
	 where $p_e$ denotes the node of $C_0$ associated with $e$.
	\item For all $v\in V(\Gamma)$, we have
	\[({\varphi_0}|_{C_{0,v}})^*(\partial X_{0,h(v)})=\sum_{e\ni v} W_e\cdot p_e.\]
	\end{enumerate}
We call such $\varphi_0$ a pre-log curve \emph{of type $(\Gamma, h)$} or simply of type $\Gamma$
 if $h$ is clear from the context.
Similarly, we define pre-log curves of type $(\Gamma, h)$ for tropical curves $h\colon\Gamma\to N_{\Bbb R}$.
\end{defin}
\noindent
\begin{rem}
A map satisfying these conditions, particularly the conditions (2) and (3) with $\Gamma$ at most 3-valent, 
 is called maximally degenerate pre-log curve in 
 \cite{Nishinou_Toric_2006}.
In this paper, we just call it a pre-log curve for simplicity.
\end{rem}

Our method for constructing pre-log curves is to glue local pieces together, taking into account the boundary conditions.
Let us begin by studying the boundary values of local pieces.

Let $T$ be the metric graph consisting of one finite vertex $v$, $n$ infinite vertices, and $n$ edges $e_1,\dots,e_n$ connecting the infinite vertices to the finite vertex.
Let 
\[
g\colon T\to N_\bbR
\]
 be a parametrized tropical curve in $N_\bbR$ with 
\[
g(v)=0\in N_\bbR.
\]
For $i=1,\dots,n$, let $W_i$ be the weight of the edge $e_i$.
Write 
\[
w_{v,e_i}=W_i(p_i,q_i)\in N,
\]
 expressed using the standard basis of $N\cong \bbZ^2$.
Here, $(p_i, q_i)$ is the primitive integral vector in the direction of $g(e_i)$.

Let $Y$ denote the (non-proper) toric variety over $\bbC$ associated with the fan determined by the image $g(T)$.

\begin{defin}\label{def:pre-logcurve}
	A \emph{pre-log curve in $Y$ associated with the tropical curve $g\colon T\to N_\bbR$}
	 is a map $\phi_0\colon\bbP^1\to Y$ such that the pullback of the toric divisor corresponding to the ray $g(e_i)$ is exactly one point with multiplicity $W_i$ for $i=1,\dots,n$.
\end{defin}

Now, we restrict ourselves to the case where $T$ contains exactly three edges.
Moreover, let 
\[
g_0\colon T\to N_\bbR
\]
 denote the special case of $g$ where the three weight vectors are equal to 
\[
\begin{pmatrix}
1 \\ 0
\end{pmatrix},\;\; 
\begin{pmatrix}
0 \\ 1
\end{pmatrix},\;\; 
\begin{pmatrix}
-1 \\ -1
\end{pmatrix},
\] 
 respectively.
Let $Y_0$ denote the (non-proper) toric variety over $\bbC$ associated with the fan 
 determined by the image $g_0(T)$.
This variety is isomorphic to $\bbP^2$ minus three points.

The image $g(T)$ is equal to the image of the composition of $g_0$ and the linear map defined by the matrix
\[
L = \begin{pmatrix}
W_1 p_1 & W_2 p_2\\
W_1 q_1 & W_2 q_2
\end{pmatrix}.
\]
This linear map induces a toric map
\[
\Phi_L\colon Y_0\to Y
\]
between toric surfaces.
Note that the weight $W_v$ of the vertex $v$ is equal to $\abs{\det L}$ (cf.\  \cref{def:weight}). 

Let $(\be_1^{\vee}, \be_2^{\vee})$ be the dual standard basis of the dual space $N^\vee$.
Let $x_1, x_2$ be the corresponding rational functions on $Y_0$.
Consider the vectors in $N^\vee$ with coordinates
\[
\begin{pmatrix}
-q_1 \\ p_1
\end{pmatrix},\;\; 
\begin{pmatrix}
-q_2 \\ p_2
\end{pmatrix},\;\; 
\begin{pmatrix}
-q_3 \\ p_3
\end{pmatrix}.
\]
They are primitive integer vectors in the annihilator subspaces $g(e_1)^\perp$,  $g(e_2)^\perp$, and $g(e_3)^\perp$, respectively.
Let $X_1, X_2$ and $X_3$ be the corresponding rational functions on $Y$.

\begin{lem}\label{lem:Phi_L}
	The pullbacks of the functions $X_1, X_2$, and $X_3$ by the map $\Phi_L$ are equal to
	\[
	x_2^{\frac{\det L}{W_1}},\quad x_1^{-\frac{\det L}{W_2}},\quad\left(\frac{x_1}{x_2}\right)^{\frac{\det L}{W_3}}
	\]
	respectively.
\end{lem}
\begin{proof}
	We have
	\[\begin{pmatrix}
	W_1 p_1 & W_1 q_1 \\
	W_2 p_2 & W_2 q_2
	\end{pmatrix}\cdot
	\begin{pmatrix}
	-q_1\\
	p_1
	\end{pmatrix}=
	\begin{pmatrix}
	0\\
	-W_2 p_2 q_1 + W_2 q_2 p_1
	\end{pmatrix}=
	\begin{pmatrix}
	0\\
	\frac{\det L}{W_1}
	\end{pmatrix}.\]
	Thus, the pullback of $X_1$ by the map $\Phi_L$ is equal to $x_2^\frac{\det L}{W_1}$.
	Similar computations apply to the other two functions.
\end{proof}

\begin{lem}\label{lem:prod}(see also \cite[Proposition 5.3]{Mikhalkin_Qindex_2017})
	Let $\gamma$ denote the greatest common divisor of $W_1, W_2, W_3$, and express
	\[
	W_i = \gamma W_i',\quad i = 1, 2, 3.
	\]
	For any pre-log curve $\phi_0\colon \bbP^1\to Y$ associated with the tropical curve $g\colon T\to N_\bbR$, let $\mu_i$ be the value of the function $X_i$ at the intersection of $\phi_0(\bbP^1)$ with the toric divisor corresponding to the ray $g(e_i)$.
	%Later we will refer to $\mu_i$ as the coordinate of the intersection point.
	Then, the numbers $\mu_i$ satisfy the relation
	\[
	\mu_1^{W_1'} \mu_2^{W_2'} \mu_3^{W_3'} = (-1)^{\frac{W_v}{\gamma}}.
	\]
	Conversely, for any numbers $\nu_1, \nu_2, \nu_3$ satisfying the relation
	\[
	\nu_1^{W_1'} \nu_2^{W_2'} \nu_3^{W_3'} = (-1)^{\frac{W_v}{\gamma}},
	\]
	there exists a pre-log curve $\phi_0\colon \bbP^1\to Y$ associated with the tropical curve $g\colon T\to N_\bbR$ 
	such that  $\mu_i=\nu_i$, for  $i =1,2,3$.
\end{lem}
\begin{proof}
	Let $\psi_0\colon \bbP^1\to Y_0$ be a pre-log curve associated with the tropical curve $g_0\colon T\to N_\bbR$.
	The image of $\psi_0$ is given by an equation of the form
	\[
	\beta_1 x_1+\beta_2 x_2-1 = 0,
	\]
	where $\beta_1$ and $\beta_2$ are non-zero complex numbers.
	It intersects the toric boundary of $Y_0$ at three points with coordinates
	\begin{align*}
	(x_1, x_2) &= \left(\frac{1}{\beta_1}, 0\right),\\
	(x_1, x_2) &= \left(0, \frac{1}{\beta_2}\right),\\
	\left(\frac{1}{x_1},\frac{x_2}{x_1}\right) &= \left(0,-\frac{\beta_1}{\beta_2}\right).
	\end{align*}
	
	Note that the composition $\phi_0\coloneqq\Phi_L\circ\psi_0\colon\bbP^1\to Y$ is a pre-log curve associated with the tropical curve $g\colon T\to N_\bbR$.
	Furthermore, every pre-log curve in $Y$ associated with $g$ arises in this way (cf.\ \cite[Lemma 5.3]{Nishinou_Toric_2006}).
	\personal{See also \cite[Lemma 8.19]{Mikhalkin_Enumerative_2005} or \cite[Lemma 4.19]{Gross_Tropical_2011}.}
	By \cref{lem:Phi_L}, we have
	\begin{equation}\label{eq:c}
	\mu_1 = \beta_2^{-\frac{\det L}{W_1}},\quad
	\mu_2 = {\beta_1^{\frac{\det L}{W_2}}},\quad
	\mu_3 = \left(-\frac{\beta_2}{\beta_1}\right)^{\frac{\det L}{W_3}}.
	\end{equation}
	Thus, we have 
	\[
	\begin{array}{ll}
	\mu_1^{W_1'} \mu_2^{W_2'} \mu_3^{W_3'}  &= \mu_1^\frac{W_1}{\gamma} \mu_2^\frac{W_2}{\gamma} \mu_3^\frac{W_3}{\gamma} \\
	& = \beta_2^{-\frac{\det L}{W_1}\cdot\frac{W_1}{\gamma}} \beta_1^{\frac{\det L}{W_2}\cdot\frac{W_2}{\gamma}} \bigg(-\frac{\beta_2}{\beta_1}\bigg)^{\frac{\det L}{W_3}\cdot\frac{W_3}{\gamma}}\\
	& = (-1)^\frac{\det L}{\gamma}\\
	& = (-1)^{\frac{W_v}{\gamma}}.
	\end{array}
	\]

	Conversely, let $\nu_1, \nu_2, \nu_3$ be complex numbers satisfying the relation
	\[
	\nu_1^{W_1'} \nu_2^{W_2'} \nu_3^{W_3'} = (-1)^{\frac{W_v}{\gamma}}. 
	\]
	By the discussion above, to construct a pre-log curve in $Y$ 
	associated with the tropical curve $g\colon T\to N_\bbR$ satisfying $\mu_i=\nu_i$, for $i=1,2,3$, it suffices to prove that $(\nu_1, \nu_2, \nu_3)$ can be written in the form of \cref{eq:c}.
	Take complex numbers $\beta_1$ and $\beta_2$ so that
	\[
	\nu_1 = \beta_2^{-\frac{\det L}{W_1}},\quad \nu_2 = \beta_1^{\frac{\det L}{W_2}}.
	\]
	Then, 
	\[
	\nu_3 = \zeta_3 \left(-\frac{\beta_2}{\beta_1}\right)^\frac{\det L}{W_3},
	\]
	where $\zeta_3$ is a $W_3'$-th root of unity. 
	Note that $\beta_2$ can be multiplied by any $\frac{W_v}{W_1}$-th root of unity $\zeta_1$, 
	and $\beta_1$ can be multiplied by any $\frac{W_v}{W_2}$-th root of unity $\zeta_2$, without changing $\nu_1$ and $\nu_2$.
	Thus, it suffices to find $\zeta_1$ and $\zeta_2$ such that
	\begin{equation}\label{eq:zeta}
	\left(\frac{\zeta_1}{\zeta_2}\right)^\frac{\det L}{W_3}=\zeta_3\inv.
	\end{equation}
	Write
	\[
	\zeta_1 = e^{2\pi i\frac{lW_1}{W_v}},\quad 
	\zeta_2 = e^{2\pi i\frac{mW_2}{W_v}},\quad
	\zeta_3 = e^{2\pi i\frac{n}{W_3'}},
	\]
	where $l, m$, and $n$ are integers.
	Substituting these into \cref{eq:zeta}, we see that it is enough to have
	\[\frac{l W_1}{W_3}\cdot\frac{\det L}{W_v} - \frac{m W_2}{W_3}\cdot\frac{\det L}{W_v} \equiv -\frac{n}{W_3'} \mod 1,\]
	which is equivalent to
	\begin{equation} \label{eq:lmn}
	l W_1 - m W_2 \equiv -\operatorname{sign}(\det L)\cdot n\gamma \mod W_3,	
	\end{equation}
	where
	\[\operatorname{sign}(\det L)=\begin{cases}
	1 & \text{if }\det L>0,\\
	-1 & \text{if }\det L<0.
	\end{cases}\]
	Since $\gamma$ is the greatest common divisor of $W_1$, $W_2$, and $W_3$, there exist integers 
	$l,m$, and $n$ such that \cref{eq:lmn} holds. 
	This completes the proof of the lemma.
\end{proof}

Now, we begin the construction of pre-log curves associated with a given 3-valent parametrized tropical curve $h\colon\Gamma\to S$.

We first recall some notations from \cref{sec:intro_(realization)}.
We lift $h$ to a $\overline{\Lambda}$-periodic tropical curve 
\[
\widetilde h\colon\widetilde{\Gamma}\to N_\bbR.
\]
Here, $\overline{\Lambda}$ is the free abelian group generated by the vectors $\overline{\lambda}_1$ and
 $\overline{\lambda}_2$ in \cref{eq:oLambda}.
Let
\[
\Delta\subset N_\bbR
\]
 be a parallelogram fundamental domain of the $\overline{\Lambda}$-action, bounded by four sides $B_1,\dots,B_4$, such that
\[
\overline{\lambda}_1(B_3)=B_1,\;\; \overline{\lambda}_2(B_4)=B_2,
\]
 and $\widetilde h(\widetilde{\Gamma})$ intersects the boundary $\partial\Delta$ transversally.
Let $\widetilde e_1,\dots,\widetilde e_l$ (resp.\ $\widetilde f_1,\dots,\widetilde f_m$) 
 be the edges of $\widetilde{\Gamma}$ whose images intersect $B_1$ (resp.\ $B_2$).
Let $(a_i,b_i)$ (resp.\ $(c_j,d_j)$) be the weight vector of 
 $\widetilde e_i$ (resp.\ $\widetilde f_j$) pointing from the inside to the outside of $\Delta$.
Let us define a complex number $\sigma$ by
\[\sigma\coloneqq\prod_{i=1}^l (\alpha_{12}^{a_i/\delta}\alpha_{11}^{-b_i/\delta})\cdot\prod_{j=1}^m(\alpha_{22}^{c_j/\delta}\alpha_{21}^{-d_j/\delta}),\]
 here $\delta$ is the greatest common divisor of all edge weights of $\Gamma$, 
 and the constants $\alpha_{11}, \dots, \alpha_{22}$ are those in \cref{eq:assump}.
\begin{thm} \label{thm:pre-log}
	There exists a pre-log curve associated with $h\colon\Gamma\to S$ if and only if 
	$\sigma=(-1)^{\sum_{v\in V(\Gamma)} W_v/\delta}$.
\end{thm}

\begin{rem}\label{rem:exponent}
As we mentioned in the introduction, the exponent $\sum_{v\in V(\Gamma)} W_v/\delta$ is always
 even, see \cite[Lemma 4.6]{Blomme_Enumeration_of_curves_1}.
Thus, we can replace $(-1)^{\sum_{v\in V(\Gamma)} W_v/\delta}$ by one.
\end{rem}

Since the proof is a little technical, before we begin the proof, we provide an overview of the argument.
The proof begins with any pre-log curve
\[
\psi_v\colon \Bbb P^1\to X_{0, h(v)}\subset X_0
\]
 corresponding to a vertex $v$ of $\Gamma$. 
The main strategy is extending it by adding pre-log curves corresponding to other vertices of $\Gamma$
 one by one.
In this construction,
 if $v$ and $v'$ are adjacent vertices connected by an edge $e$,
 then the images of $\psi_v$ and $\psi_{v'}$ must intersect on the divisor corresponding to 
 the edge $e$.
 
To manage this extension process systematically, we order the vertices $V(\Gamma)$ as $\{v_1', \dots, v_n'\}$
 in the following specific way (in the notation of the proof, $v_i'$ corresponds to $v_{n-i}$).
Namely, let $\Gamma_k'$ be the subgraph of $\Gamma$ formed by 
 the vertices $\{v_1', \dots, v_k'\}$ and all the edges connecting them. 
Note that the graph $\Gamma_k'$ may not be connected.
Then, 
\begin{itemize}
\item there should be at most two edges connecting the vertex $v_{k+1}'$ and $\Gamma_k'$, for all $1\leq k\leq n-2$, and
\item the last vertex $v_n'$ must be a 3-valent vertex.
\end{itemize}

On the side of pre-log curves, this implies the following.
Namely, assume we have constructed a curve 
\[
\psi_k'\colon C_{0,k}'\to X_0,
\] 
 where $C_{0, k}'$ is the pre-stable curve whose dual intersection graph is $\Gamma_k'$, 
 and the restriction of $\psi_k'$ to each irreducible component $C_{0, v_i'}$, $1\leq i\leq k\leq n-2$,
 is a pre-log curve in $X_{0, h(v_i')}$.
We want to extend $\psi_k'$ to a map $\psi_{k+1}'$ from the 
 pre-stable curve $C_{0, k+1}'$ whose dual intersection graph is $\Gamma_{k+1}'$,
 and this is always possible.
The reason is that if $\{e_j\}$ is the set of edges connecting $v_{k+1}'$ and $\Gamma_k'$, 
 then the cardinality of $\{e_j\}$ is at most two.
Therefore, the restriction $\psi_{k+1}'|_{C_{0, k+1}'}$ is constrained by the condition that
 it must intersect a given set of points on $X_{0, h(v_{k+1}')}$ whose cardinality is at most two.
This is always possible, essentially for the same reason that for any given two points in $\Bbb P^2$,
 there is a line through them. 

Thus, the problem reduces to determining the condition for extending $\psi_{n-1}'$ to the last irreducible component
 $C_{0, v_n'}$.
The necessity of the equation $\sigma=(-1)^{\sum_{v\in V(\Gamma)} W_v/\delta}$ follows from repeatedly 
 applying Lemma \ref{lem:prod}.
Proving sufficiency is more technically involved, though essentially proved by 
 Chinese remainder theorem.
See Figure \ref{fig:modify} for an illustration explaining the main point of this part of the proof.

\begin{proof}
	We introduce an ordering $v_1,\dots,v_n$ on $V(\Gamma)$ as follows.
	We start with any 3-valent vertex $v_1\in V(\Gamma)$.
	By subdividing $\Gamma$ if necessary, we assume that the image of each edge connected to $v_1$ does not intersect the image of $\partial\Delta$ under the covering map $N_\bbR\to S$.
	Having constructed the sequence $v_1,\dots,v_i$, 
	 we choose $v_{i+1}\in V(\Gamma)\setminus\{v_1,\dots,v_i\}$ such that $v_{i+1}$ is connected to $v_j$ for some $j=1,\dots, i$.
	
	For $i=1,\dots,n$, let 
	\[
	\Gamma_i\subset\Gamma
	\]
	 be the closed subgraph consisting of 
	the vertices $v_{i},\dots,v_n$, and all the edges between them.
	Let $T$ be the metric graph consisting of one finite 3-valent vertex, 
	 three infinite vertices and three edges connecting the 
	 infinite vertices to the finite vertex.
	If $v_i$ is a 3-valent vertex of $\Gamma$, define
\[
g_{v_i}\colon T\to N_\bbR
\]
 as the parametrized tropical curve in $N_\bbR$ such that the finite vertex of $T$ maps to $0\in N_\bbR$, and the weight vectors associated with the three infinite edges of $T$ are equal to those associated with the three edges of $\Gamma$ connected to $v_i$.
	The case when $v_i$ is 2-valent is straightforward and will be omitted.
	
	Let 
	\[
	\psi_{v_n}\colon P_{v_n}=\bbP^1\to X_{0,h(v_n)}\subset X_0
	\]
	 be a pre-log curve associated with the tropical curve $g_{v_n}\colon T\to N_\bbR$.
	Let us write 
	\[
	C_{0,n}\coloneqq P_{v_n}
	\]
	 and 
	 \[
	 \phi_n\coloneqq\psi_{v_n}.
	 \]
	Suppose we have constructed 
	\[
	\phi_{i+1}\colon C_{0,i+1}\to X_0
	\]
	for some $i\ge 2$, such that the dual graph of $C_{0,i+1}$ is $\Gamma_{i+1}$.
	We now proceed to construct 
	\[
	\phi_i\colon C_{0,i}\to X_0
	\]
	 as follows.
	
	Let $C_{0,i}$ be the union of $C_{0,i+1}$ and $P_{v_i}\coloneqq\bbP^1$, such that the dual graph of $C_{0,i}$ is $\Gamma_i$.
	By the construction of the ordering $v_1,\dots,v_n$, there exists at least one edge connected to $v_i$ that does not belong to $\Gamma_i$.
	Note that for any pre-log curve 
	\[
	\psi_{v_i}\colon \Bbb P^1\to X_{0, h(v_i)}
	\]
	 of type $(T, g_{v_i})$, there is a line (i.e., a curve of degree one)
	\[
	\psi_1\colon \Bbb P^1\to \Bbb P^2
	\]
	 and a natural toric map 
	 \[
	 P\colon \Bbb P^2\to X_{0, h(v_i)}
	 \]
	satisfying
	\[
	\psi_{v_i} = P\circ\psi_1,
	\]
	 see \cite[Lemma 5.3]{Nishinou_Toric_2006}.
	Then, since for any one or two points on $\Bbb P^2$, there is a line through it, 
	 the same property holds on pre-log curves of type $(T, g_{v_i})$.
	Here, we take these point constraints away from the fixed
	 points of the torus action.
	In particular, we can take these points to be the image of the nodes of $C_0$ corresponding to the 
	edges of $\Gamma_i$ attached to $v_i$.
	Therefore,
	 we can extend the map $\phi_{i+1}\colon C_{0,i+1}\to X_0$ to a map $\phi_i\colon C_{0,i}\to X_0$ such that the restriction $\psi_{v_i}\coloneqq\phi_i|_{P_{v_i}}$ is a pre-log curve in $X_{0,h(v_i)}$ associated with the tropical curve $g_{v_i}\colon T\to N_\bbR$.
	
	We iterate this process from $i=n-1$ to $i=2$.
	Finally, let $C_{0,1}$ be the union of $C_{0,2}$ and $P_{v_1}\coloneqq\bbP^1$ such that the dual graph of $C_{0,1}$ is $\Gamma_1=\Gamma$.
	We need to extend the map $\phi_2\colon C_{0,2}\to X_0$ to a map $\phi_1\colon C_{0,1}\to X_0$ such that the restriction $\psi_{v_1}\coloneqq\phi_1|_{P_{v_1}}$ is a pre-log curve in $X_{0,h(v_1)}$ associated with the tropical curve $g_{v_1}\colon T\to N_\bbR$.
	
	For $i=2,\dots,n$, let $\{e_{v_i, j}\}$ be the set of edges incident to $v_i$.
	Here, $j$ runs over $\{1, 2\}$ or $\{1, 2, 3\}$ according to whether $v_i$ is 2-valent or 3-valent.
	Let $\mu_{v_i,j}$ be the coordinate of the intersection point between $\psi_{v_i}(P_{v_i})$ and the stratum $X_{0,h(e_{v_i,j})}$, as in \cref{lem:prod}.	
	
\begin{lem}\label{lem:fill}
	Let $\nu_1, \nu_2, \nu_3$ be the coordinates of the intersection points between $\phi_2(C_{0, 2})$ and the three boundary components of $X_{0,h(v_1)}$.
	These satisfy the relation
	\begin{equation} \label{eq:fill}
	\sigma\cdot \nu_1^{W_{e_{v_1, 1}}/\delta}\nu_2^{W_{e_{v_1, 2}}/\delta}\nu_3^{W_{e_{v_1, 3}}/\delta}
	 = \prod_{i=2}^n (-1)^{W_{v_i}/\delta}.
	\end{equation}
%	where $W_{v_i,j}$ is shorthand for $W_{e_{v_i,j}}$, the weight of the $j$-th edge connected to the vertex $v_i$.
\end{lem}
\begin{proof}
	Let $e$ be an edge of $\Gamma$ not connected to $v_1$, with endpoints $v_i$ and $v_{i'}$.
	Assume that $e=e_{v_i,j}=e_{v_{i'},j'}$.
	If $h(e)$ does not intersect the image of $\partial\Delta$ under the covering map $N_\bbR\to S$, we have 
	\[
	\mu_{v_i,j}\mu_{v_{i'},j'}=1.
	\]
	Otherwise, the product $\mu_{v_i,j}\mu_{v_{i'},j'}$ incorporates the coordinate transformations
	 induced by the action of $\Lambda$ on $\widetilde{\mathcal X}$.
	More precisely, let $\widetilde e$ be the representative of $e$ in $\widetilde{\Gamma}$ that intersects the boundary $B_1$ or $B_2$ of $\Delta$. 
	Let $W_e(p,q)\in N$ be the weight vector of $e$, pointing from the inside to the outside of $\Delta$.
	In the case where $\widetilde e$ intersects $B_1$, we have 
	\[
	\mu_{v_i,j}\mu_{v_{i'},j'}=\alpha_{12}^p\alpha_{11}^{-q}.
	\]
	 Therefore, 
	\[
	\mu_{v_i,j}^{W_e/\delta}\mu_{v_{i'},j'}^{W_e/\delta}=\alpha_{12}^{W_e p/\delta}\alpha_{11}^{-W_e q/\delta}.
	\]
	In the case where $\widetilde e$ intersects $B_2$, we have
	\[
	\mu_{v_i,j}\mu_{v_{i'},j'}=\alpha_{22}^p\alpha_{21}^{-q}.
	\]
	Therefore,
	\[
	\mu_{v_i,j}^{W_e/\delta}\mu_{v_{i'},j'}^{W_e/\delta}=\alpha_{22}^{W_e p/\delta}\alpha_{21}^{-W_e q/\delta}.
	\]	
	Taking the product of $\mu_{v_i,j}^{W_e/\delta}\mu_{v_{i'},j'}^{W_e/\delta}$ over all edges of $\Gamma$ except $e_{v_1,1}, e_{v_1,2}, e_{v_1,3}$, we obtain
	\[\prod_{i=2}^n \prod_j \mu_{v_i,j}^{W_{e_{v_i, j}}/\delta} = \sigma\cdot \nu_1^{W_{e_{v_1, 1}}/\delta}\nu_2^{W_{e_{v_1, 2}}/\delta}\nu_3^{W_{e_{v_1, 3}}/\delta}.\]
	
	On the other hand, if $v_i$ is 2-valent, we have 
	\[
	\mu_{v_i,1}\mu_{v_i,2}=1.
	\]
	If $v_i$ is 3-valent, by \cref{lem:prod}, we have 
	\[
	\mu_{v_i,1}^{W_{e_{v_i, 1}}/\delta} \mu_{v_i,2}^{W_{e_{v_i, 2}}/\delta} \mu_{v_i,3}^{W_{e_{v_i, 3}}/\delta} 
	  =  (-1)^{W_{v_i}/\delta}.
	\]
	Taking the product over all vertices $v_i$ for $i = 2, \dots, n$, we have
	\[\prod_{i=2}^n \prod_j \mu_{v_i,j}^{W_{e_{v_i, j}}/\delta} = \prod_{i=2}^n (-1)^{W_{v_i}/\delta}.\]
	Combining the two equations above, we obtain \cref{eq:fill}.
\end{proof}

Let $\gamma_{v_1}$ be the greatest common divisor of the edge weights 
 $W_{e_{v_1, 1}}$, $W_{e_{v_1, 2}}$, and $W_{e_{v_1, 3}}$.
By \cref{lem:prod}, we can fill in the pre-log curve $\psi_{v_1}\colon P_{v_1}\to X_{0,h(v_1)}$ 
 if and only if the numbers $\nu_1, \nu_2, \nu_3$ from \cref{lem:fill} satisfy the equation
\begin{equation} \label{eq:product_desired}
\nu_1^{W_{e_{v_1, 1}}/\gamma_{v_1}}\nu_2^{W_{e_{v_1, 2}}/\gamma_{v_1}}\nu_3^{W_{e_{v_1, 3}}/\gamma_{v_1}} = (-1)^{W_{v_1}/\gamma_{v_1}}. 
\end{equation}
By substituting the $\gamma_{v_1}/\delta$-th power of (\ref{eq:product_desired})
 to (\ref{eq:fill}), we see that the existence of a pre-log curve associated with the map 
 $h\colon\Gamma\to S$ implies the equation $\sigma=(-1)^{\sum_{v\in V(\Gamma)} W_v/\delta}$.
This proves one direction of the claim in \cref{thm:pre-log}.\\

Now, for the other direction, let us assume that $\sigma=(-1)^{\sum_{v\in V(\Gamma)} W_v/\delta}$ holds,
 and we aim to fill in the last piece 
\[
\psi_{v_1}\colon P_{v_1}\to X_{0,h(v_1)}.
\]
\cref{lem:fill} implies that the numbers $\nu_1, \nu_2, \nu_3$ satisfy the equation
\begin{equation} \label{eq:product_v_1}
\nu_1^{W_{e_{v_1, 1}}/\gamma_{v_1}}\nu_2^{W_{e_{v_1, 2}}/\gamma_{v_1}}\nu_3^{W_{e_{v_1, 3}}/\gamma_{v_1}} = \zeta\cdot (-1)^{W_{v_1}/\gamma_{v_1}}, 
\end{equation}
for some $\gamma_{v_1}/\delta$-th root of unity $\zeta$.

\begin{lem}\label{lem:modification}
	We can modify the map $\phi_2\colon C_{0, 2}\to X_0$ so that after the modification, 
	 the equation \eqref{eq:product_desired} holds.
\end{lem}
\begin{proof}
The modification of $\phi_2$ is achieved by adjusting each $\psi_{v_i}\colon P_{v_i}\to X_{0, h(v_i)}$ for
 $i = 2, \dots, n$, while preserving the incidence relations of curves.
Let 
\[
\gamma_{v_1}/\delta = \epsilon_1^{m_1}\cdots \epsilon_k^{m_k}
\]
 denote the prime decomposition of $\gamma_{v_1}/\delta$.
Accordingly, the multiplier $\zeta$ factors as
\[
\zeta_1\cdots \zeta_k,
\]
 where $\zeta_j$ is an $\epsilon_j^{m_j}$-th root of unity.
\personal{By Chinese reminder theorem.}

Consider the prime $\epsilon_j$.
Since $\delta$ is the greatest common divisor of all the edge weights of $\Gamma$, 
 there is a vertex $u$ of $\Gamma$ and an edge $e$ incident to $u$ such that $W_e/\delta$ is not divisible by $\epsilon_j$.
We choose $u$ such that there is a simple path $\mathcal P$ in $\Gamma$ 
 from  $v_1$ to $u$, consisting entirely of edges whose weights are divisible by $\delta\epsilon_j$.
Note that since $\gamma_{v_1}$ is the greatest common divisor of the edge weights 
 $W_{e_{v_1, 1}}$, $W_{e_{v_1, 2}}$, and $W_{e_{v_1, 3}}$ of the edges
 incident to $v_1$, the weight of any edge $e_{v_1, i}$ is divisible by $\delta\epsilon_j$.
In particular, there is such a simple path $\mathcal P$.

Let 
\[
u_1=v_1, \, u_2,\,  u_3,\,  \dots,\,  u_{l-1},\,  u_l=u
\]
 be the vertices on the path $\mathcal P$.
Note that $u_l$ must be a 3-valent vertex.
Let 
\[
e_1,\dots,e_{l-1}
\]
 be the edges on the path $\mathcal P$.
By reordering the labels of the edges incident to the vertices $u_1,\dots,u_l$ if necessary, we may assume
\[e_1=e_{u_1,1}=e_{u_2,1},\quad e_2=e_{u_2,2}=e_{u_3,1},\quad \dots,\quad e_{l-1}=e_{u_{l-1},2}=e_{u_l,1}.\]

We are equipped with pre-log curves $\psi_{u_i}\colon P_{u_i}\to X_{0,h(u_i)}$ for $i=2,\dots,l$.
Let $\mu_{u_i,1}, \mu_{u_i,2}, \dots$ denote 
 the coordinates of the intersection between $\psi_{u_i}(P_{u_i})$ 
 and the boundary components of $X_{0,h(u_i)}$ as described in \cref{lem:prod}.

\begin{figure}
\includegraphics[height=12cm]{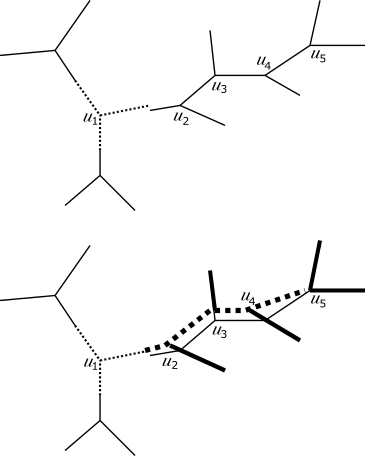}
\caption{Illustrations of the modification process for the case $l = 5$.
These illustrations do \emph{not} represent tropical curves,
 but instead demonstrate the intersection properties of the irreducible components
 of $C_0$ using the dual graph.
Upper picture: The solid part corresponds to the image of $C_{0, 2}$. 
We aim to fill in the last component
 $C_{0, u_1}$ so that the three intersections with the toric divisors of $X_{0, u_1}$ align with 
 those of the image of $C_{0, 2}$.
However, in general, only two of these conditions can be satisfied simultaneously, while the third can not.
This is reflected in the picture by the misaligned edges between the vertices $u_1$ and $u_2$.
Lower picture: The situation is corrected by modifying the maps $\psi_{u_i}$, $i = 2,3,4,5$.
For $\psi_{u_i}$ with $i = 2,3,4$, two of the intersections between the image $\psi_{u_i}(C_{0, u_i})$ and
 the toric divisors of $X_{0, u_i}$ are allowed to move (depicted by the bold dotted lines), while the remaining 
 intersection stays in place without moving (pictured by the bold solid line).
For $\psi_{u_5}$, we can modify it so that only one of the intersection between $\psi_{u_5}(C_{0, u_5})$ and
 the toric divisors of $X_{0, u_5}$ moves, while the remaining two intersections stay in place.
The map on the remaining part of  $C_{0, 2}$ is left unchanged.}\label{fig:modify}
\end{figure}

\begin{claim}\label{claim:phi2}
We can modify the map $\phi_2\colon C_{0, 2}\to X_0$ so that the modified map $\widetilde{\phi}_2$
 satisfies the following conditions:
\begin{itemize}
\item $\phi_2 = \widetilde{\phi}_2$ on the complement of the irreducible components $P_{v_2}, \dots, P_{v_l}$ of $C_{0, 2}$.
\item Let $\widetilde{\nu}_i$, for $i = 1, 2, 3$, be the coordinates of the intersection points between 
 $\widetilde{\phi}_2(C_{0, 2})$ and the three boundary components of $X_{0,h(v_1)}$.
	Then, the equations
	\[
	\widetilde{\nu}_1^{W_{e_{v_1, 1}}/\gamma_{v_1}} = \zeta_j^{-1}\nu_1^{W_{e_{v_1, 1}}/\gamma_{v_1}},\;\;
	 \widetilde{\nu}_2 = \nu_2,\;\; \widetilde{\nu}_3 = \nu_3
	\]
	hold.
\end{itemize}
\end{claim}
\proof
Note that $u_1=v_1$, $\gamma_{v_1}=\gamma_{u_1}$, $\nu_1=\mu_{u_2,1}$, and $W_{e_{v_1, 1}}=W_{e_{u_2, 1}}$.
It suffices to show that the maps $\psi_{u_i}\colon P_{u_i}\to X_{0,h(u_i)}$ for $i=2, \dots, l$, can be modified
 so that the following properties hold.
Namely, if $\widetilde{\mu}_{u_i, j}$ denotes the analogue of $\mu_{u_i, j}$ for the new maps, 
 the equations
\begin{enumerate}
\item[(i)] $\widetilde{\mu}_{u_2, 1}^{W_{e_{v_1, 1}}/\gamma_{v_1}} = \zeta_j^{-1}\mu_{u_2, 1}^{W_{e_{v_1, 1}}/\gamma_{v_1}}$,
\item[(ii)] $\widetilde{\mu}_{u_i, 2} = \widetilde{\mu}_{u_{i+1},1}$, for $i = 2, \dots, l-1$, 
\item[(iii)] $\widetilde{\mu}_{u_i, 3} = \mu_{u_i, 3}$, for $i = 2, \dots, l-1$, if $u_i$ is a 3-valent vertex of $\Gamma$, and
\item[(iv)] $\widetilde{\mu}_{u_l, 2} = \mu_{u_l, 2}$ and $\widetilde{\mu}_{u_l, 3} = \mu_{u_l, 3}$
\end{enumerate}
 hold.
The equation (i) is equivalent to 
 $\widetilde{\nu}_1^{W_{e_{v_1, 1}}/\gamma_{v_1}} = \zeta_j^{-1}\nu_1^{W_{e_{v_1, 1}}/\gamma_{v_1}}$.
The equations in (ii) are part of the conditions required for the modified maps to extend to a map $\widetilde{\phi}_2$
 from $C_{0, 2}$ to $X_0$.
Namely, they ensure that the modified maps are compatible at the intersections of their domains.
The equations (iii) and (iv) are also the conditions for the new maps to extend to $\widetilde{\phi}_2$.
Specifically, these conditions ensure that the new maps glue with the non-modified part, 
 that is, the restriction of $\phi_2$ to 
 the complement of the irreducible components $P_{v_2}, \dots, P_{v_l}$ of $C_{0, 2}$.
See Figure \ref{fig:modify} for the illustration of these conditions.
 
%Now, we modify the values $\mu_{u_i,1},\mu_{u_i,2}$ 
% along the path $\mathcal P$ in a consistent way so that 
% the power $\nu_1^{W_{v_1,1}/\gamma_{v_1}}$ is multiplied by $\zeta_j\inv$.

For any integer $z$, let $p(z)$ denote the power of $\epsilon_j$ in $z$, 
 that is, $z/\epsilon_j^{p(z)}$ is an integer indivisible by $\epsilon_j$.
We first aim to change $\mu_{u_2,1}$ so that the power $\mu_{u_2,1}^{W_{e_{u_2, 1}}/\gamma_{u_1}}$ is multiplied by $\zeta_j\inv$.
Note that $\zeta_j\inv$ is an $\epsilon_j^{m_j}$-th root of unity, and
\[
p(W_{e_{u_2, 1}}/\gamma_{u_1})+m_j=p(W_{e_{u_2, 1}}/\gamma_{u_1})+p(\gamma_{u_1}/\delta)=p(W_{e_{u_2, 1}}/\delta).
\]
Therefore, it suffices to multiply $\mu_{u_2,1}$ by a suitable $\epsilon_j^{p(W_{e_{u_2, 1}}/\delta)}$-th root of unity, which we denote by $\xi_1$.

If $u_2$ is 2-valent, the existence of the pre-log curve $\psi_{u_2}\colon P_{u_2}\to X_{0,h(u_2)}$ requires that $\mu_{u_2,1}\mu_{u_2,2}=1$.
In order to keep this equation while multiplying $\mu_{u_2,1}$ by $\xi_1$, it is sufficient to multiply $\mu_{u_2,2}$ by $\xi_2\coloneqq\xi_1\inv$.
If $u_2$ is 3-valent, the existence of the pre-log curve $\psi_{u_2}\colon P_{u_2}\to X_{0,h(u_2)}$ requires
 that the coordinates $\mu_{u_2,1}, \mu_{u_2,2}$ and $\mu_{u_2,3}$ satisfy the equation
\begin{equation} \label{eq:product_u_2}
\mu_{u_2,1}^{W_{e_{u_2, 1}}/\gamma_{u_2}}\mu_{u_2,2}^{W_{e_{u_2, 2}}/\gamma_{u_2}}\mu_{u_2,3}^{W_{e_{u_2, 3}}/\gamma_{u_2}}=(-1)^{W_{u_2}/\gamma_{u_2}},
\end{equation}
 as stated in Lemma \ref{lem:prod}.
Multiplying $\mu_{u_2,1}$ by $\xi_1$, that is, taking 
\[
\widetilde{\mu}_{u_2, 1} = \xi_1\mu_{u_2, 1}
\]
 changes the left hand side 
 of the equation by $\xi_1^{W_{e_{u_2, 1}}/\gamma_{u_2}}$.
Since $\xi_1$ is an $\epsilon_j^{p(W_{e_{u_2, 1}}/\delta)}$-th root of unity, and we have the equations
\[p(W_{e_{u_2, 1}}/\delta)-p(W_{e_{u_2, 1}}/\gamma_{u_2})=p(\gamma_{u_2}/\delta),\]
\[p(\gamma_{u_2}/\delta)+p(W_{e_{u_2, 2}}/\gamma_{u_2})=p(W_{e_{u_2, 2}}/\delta),\]
in order to preserve the equation \eqref{eq:product_u_2} and the condition (iii) above, 
 it suffices to multiply $\mu_{u_2,2}$ by a 
 suitable $\epsilon_j^{p(W_{e_{u_2, 2}}/\delta)}$-th root of unity, which we denote by $\xi_2$.
Thus, we take 
\[
\widetilde{\mu}_{u_2, 2} = \xi_2\mu_{u_2, 2}.
\]
The number $\mu_{u_2, 3}$ is unchanged: 
\[
\widetilde{\mu}_{u_2, 3} = \mu_{u_2, 3}.
\]
These choices determine the map $\psi_{u_2}\colon P_{u_2}\to X_{0,h(u_2)}$.

The condition (ii) above is equivalent to the condition that
\[
\mu_{u_2,2}\mu_{u_3,1} = \widetilde{\mu}_{u_2,2}\widetilde{\mu}_{u_3,1}
\]
 is a fixed constant, as in the proof of \cref{lem:fill}.
Therefore, $\mu_{u_3,1}$ must be multiplied by $\xi_2\inv$.
Next, we change $\mu_{u_3,2}$ while keeping $\mu_{u_3, 3}$ fixed.
This ensures the existence of the pre-log curve $\psi_{u_3}\colon P_{u_3}\to X_{0,h(u_3)}$
 and the condition (iii) above.

By iterating this process for $i=3,\dots, l-1$, we multiply $\mu_{u_i,1}$ 
 by some $\epsilon_i^{p(W_{e_{u_i, 1}}/\delta)}$-th root of unity, and $\mu_{u_i,2}$ 
 by some $\epsilon_i^{p(W_{e_{u_i, 2}}/\delta)}$-th root of unity, while keeping $\mu_{u_i, 3}$, if it exists, unchanged.
In this way, we construct the maps $\psi_{u_i}\colon P_{u_i}\to X_{0,h(u_i)}$, $i = 2, \dots, l-1$,
 satisfying the conditions (i), (ii), and (iii) above.

For $i=l$, consider the equation
\begin{equation} \label{eq:product_u_l}
\mu_{u_l,1}^{W_{e_{u_l, 1}}/\gamma_{u_l}}\mu_{u_l,2}^{W_{e_{u_l, 2}}/\gamma_{u_l}}\mu_{u_l,3}^{W_{e_{u_l, 3}}/\gamma_{u_l}}=(-1)^{W_{u_l}/\gamma_{u_l}}.
\end{equation}
Note that $\mu_{u_l,1}$ is multiplied by an $\epsilon_j^{p(W_{e_{u_{l-1}, 2}}/\delta)}$-th root of unity.
By the choice of the vertex $u_l=u$, $\gamma_{u_l}/\delta$ is not divisible by $\epsilon_j$.
Therefore, we have
\[p(W_{e_{u_{l-1}, 2}}/\delta)-p(W_{e_{u_l, 1}}/\gamma_{u_l})=p(\gamma_{u_l}/\delta)=0,\]
 since $W_{e_{u_{l-1},2}} = W_{e_{u_l, 1}}$.
This implies that the power $\mu_{u_l,1}^{W_{e_{u_l, 1}}/\gamma_{u_l}}$ in \eqref{eq:product_u_l} remains unchanged.
In other words, the equation \eqref{eq:product_u_l} holds without requiring 
 any changes to $\mu_{u_l,2}$ or $\mu_{u_l,3}$.
This determines the map $\psi_{u_l}\colon P_{u_l}\to X_{0,h(u_l)}$. 
By gluing the maps
 $\psi_{u_i}\colon P_{u_i}\to X_{0,h(u_i)}$, $i = 2, \dots, l$, with the restriction of $\phi_2$
 to the complement of the irreducible components $P_{v_2}, \dots, P_{v_l}$ of $C_{0, 2}$, 
 we obtain the map $\widetilde{\phi}_2$.
Note that since $\phi_2 = \widetilde{\phi}_2$ 
 on the complement of $P_{v_2}, \dots, P_{v_l}$,
 the conditions $\widetilde{\nu}_2 = \nu_2$ and $\widetilde{\nu}_3 = \nu_3$ hold.
This proves Claim \ref{claim:phi2}.\qed\\

%To summarize, we have changed the values $\mu_{u_i,1}, \mu_{u_i,2}$ along the path $\mathcal P$ in a consistent way so that the power $\nu_1^{W_{e_{v_1, 1}}/\gamma_{v_1}}$ is multiplied by $\zeta_j\inv$.

Repeating this procedure for each $\zeta_j$, $j=1,\dots,n$, we obtain the desired equation
 \eqref{eq:product_desired},  thereby completing the proof of \cref{lem:modification}.
\end{proof}

Finally, we conclude the proof of \cref{thm:pre-log} by \cref{lem:modification} and \cref{lem:prod}.
\end{proof}

\section{Deformations to general fibers}\label{sec:deformation}

In this section, we study the deformations of pre-log curves constructed in Theorem \ref{thm:pre-log}.
% from the special fiber to the generic fiber.
We use logarithmic deformation theory
 \cite{Kato_Logarithmic_1989,Kato_Log_smooth_deformation_1996}.
The crucial point is the calculation of the obstruction based on the local study of semiregular subvarieties
 developed in \cite{Nishinou_Obstruction_2018}.

\subsection{Log smooth deformation theory}\label{subsec:logdeformation}
Let $\mathcal X\to B$ be a Mumford family of degenerating complex tori of dimension two, 
 and suppose there is a torically transverse map
\[
\varphi_0\colon C_0\to X_0,
\]
 where $C_0$ is a pre-stable curve and $X_0$ is the central fiber of $\mathcal X$, 
 see Section \ref{sec:tropicalization}.
In the previous section, we studied the condition for the existence of $\varphi_0$.
In this section, we study the deformation of $\varphi_0$ to general fibers.
When the map $\varphi_0$ is not an immersion at the nodes of $C_0$, 
 the formalism of log deformation theory is particularly useful.
Let $k$ be a non-negative integer.
Let 
\[
O_n = (\Spec \Bbb C[t]/(t^{n+1}), \mathcal M_{O_n})
\]
 be the (thickened) log point, where
\[
\mathcal M_{O_n} =\Bbb N\oplus (\Bbb C[t]/(t^{n+1}))^{\times}
\]
 and
 the structure map 
\[
\mathcal M_{O_n}\to \Bbb C[t]/(t^{n+1})
\] 
 sends the generator $1\in \Bbb N$ to $t\in \Bbb C[t]/(t^{n+1})$.

Recall that a toric variety has a natural enhancement to a log variety given by the toric divisors.
Since any point on $\mathcal X$ has an open neighborhood
 which is naturally analytically isomorphic to an open subset of a toric variety, 
 $\mathcal X$ also has a natural structure of a log variety.
By viewing the base $B$ as an open subset of the toric variety $\Bbb A^1$, 
 the map $\mathcal X\to B$ is log smooth.
For $C_0$, we equip it with a log structure satisfying the following two conditions:
\begin{itemize}
\item It is log smooth and integral over $O_0$.
Moreover, it is strict over $O_0$ away from the nodes (that is, 
 the log structure is isomorphic to that pulled back from $O_0$).
\item The map $\varphi_0$ lifts to a map between log varieties over $O_0$.
\end{itemize}

Explicitly, let $e$ be an edge whose image  $h(e)$ has the integral length $m$.
Let $X_{0, h(e)}$ be the divisor of $X_0$ corresponding to the edge $e$.
Then, at the generic point $\eta$ of $X_{0, h(e)}$, the log structure is given by 
 the monoid 
\[
\mathcal M_{\eta} = S_m\oplus \mathcal O_{\eta}^{\times}.
\]
Here, 
\[
S_m =\Bbb N^2\oplus_{\Bbb N}\Bbb N,
\]
 where the generator $1$ of the 
 monoid $\Bbb N$ acts on the first factor by adding $(1, 1)$ and on the second factor by adding $m$.
In other words, the element $(1, 1)$ of the first factor is identified with the element $m$ of the second.
At $\eta$, the divisor $X_{0, h(e)}$ is defined by the equation $xy = t^m$ using toric charts of $\mathcal X$, 
 and the structure map is defined by 
\[
\Bbb N^2\oplus_{\Bbb N}\Bbb N\ni ((a, b), c)\mapsto x^ay^bt^c.
\]

Then, the log structure of $C_0$ at a node $p$ which maps to $X_{0, h(e)}$ is described as follows.
Namely, the monoid is given by 
\[
\mathcal M_p = S_{m/W_e}\oplus \mathcal O_{p}^{\times},
\]
% where the monoid $\Bbb N$ acts on the first factor by adding $(1, 1)$ and on the second factor by adding $e/w_e$.
 where $W_e$ denotes the weight of the edge $e$.
Note that $W_e$ is a divisor of $m$ by Assumption \ref{assum:1}.
It is equal to the ramification index of the restriction of $\varphi_0$ to each branch of $C_0$ at $p$. 
Specifically, there exist local coordinates $z$ and $w$ for the branches of $C_0$ at the node $p$ such that
 the equations
\[
\varphi_0^*x = z^{W_e},\;\; \varphi_0^*y = w^{W_e}
\]
 hold.
The structure map for $\mathcal M_p$ is defined by
\begin{equation}\label{eq:logcurve}
(a, b, c)\mapsto (\zeta^{-1}z)^aw^bt^c,
\end{equation}
 where $\zeta$ is a $W_e$-th root of unity.
The map from $\varphi_0^*\mathcal M_{\eta}$ to $\mathcal M_p$ is induced by the monoid map 
 from $S_m$ to $S_{m/W_e}$, which is the multiplication by $W_e$ on the first factor, and the
 identity on the second factor.
The choice of a root of unity at each node introduces non-isomorphic log structures on $C_0$.
In particular, there are $\prod_{e\in E(\Gamma)} W_e$ possible choices of log structures, 
 see \cite[Proposition 7.1]{Nishinou_Toric_2006} and \cite[Proposition 4.23]{Gross_Tropical_2011} for details.
 
\begin{defin}\label{def:lognormal}
The log normal sheaf of the map $\varphi_0$
 is defined as the quotient sheaf
\[
\mathcal N_{\varphi_0} = \varphi_0^*\Theta_{X_0}/\Theta_{C_0}.
\]
Here, $\Theta_{X_0}$ and $\Theta_{C_0}$ are log tangent sheaves of $X_0$ and $C_0$, respectively.
\end{defin}
Since $X_0$ is locally naturally isomorphic to a toric variety, the sheaf $\Theta_{X_0}$ is 
 free, and can be naturally identified with $N\otimes\mathcal O_{X_0}$.
The sheaf $\Theta_{C_0}$ is also locally free, and we have the following.
\begin{prop}\label{prop:geninj}
The sheaf $\mathcal N_{\varphi_0} = \varphi_0^*\Theta_{X_0}/\Theta_{C_0}$ is 
 an invertible sheaf.
\end{prop}
\proof
At a node $p$ of $C_0$, it suffices to show that the map between the dual sheaves
\[
\varphi_0^*\Omega_{X_0}:= (\varphi_0^*(\Theta_{X_0}))^{\vee} \to \Omega_{C_0}:=(\Theta_{C_0})^{\vee}
\]
 is surjective at $p$.
Using the notation introduced above, the stalk $(\Omega_{C_0})_p$ is generated by 
 the logarithmic 1-form 
\[
\frac{dz}{z} = -\frac{dw}{w}.
\]
Then, the logarithmic 1-form $\frac{dx}{x}$, 
 which is a section of $\varphi_0^*\Omega_{X_0}$, is sent to $W_e\frac{dz}{z}$ by the above map.
This proves the claim.

Away from the nodes, it suffices to prove that the map $\varphi_0$ is an immersion.
Let $v$ be a vertex of $\Gamma$, and let 
\[
\varphi_0|_{C_{0, v}}\colon C_{0, v}\to X_{0, h(v)}
\]
 be the restriction of $\varphi_0$ 
 (see Definition \ref{def:pre-log} for the notation).
First, we assume that the vertex $v$ is 3-valent.
The restriction of $\Gamma$ to a small open neighborhood of $v$ naturally extends to 
 a tropical curve 
\[
g_v\colon T\to \Bbb R^2
\]
 with one vertex and three unbounded edges.
Assume the vertex of $T$ is mapped to the origin in $\Bbb R^2$.
Then, the map $\varphi_0|_{C_{0, v}}$ defines a pre-log curve in $Y$ associated with 
 the tropical curve $(T, g_v)$, 
 where $Y$ is the non-proper toric variety associated with the fan determined by the image $g_v(T)$,
 see Definition \ref{def:pre-logcurve}.

On the other hand, let 
\[
g_0\colon T\to \Bbb R^2
\]
 be the tropical curve whose vertex is
 mapped to the origin, and the weight vectors of the edges are 
\[
(1, 0), \;\;
(0, 1), \;\;
(-1, -1).
\]
As in Section \ref{sec:pre-log}, there is an integral linear map $L\colon \Bbb R^2\to \Bbb R^2$
 satisfying 
\[
g_v = L\circ g_0.
\]
If $Y_0$ is the toric variety associated with the fan determined by the image $g_0(T)$,
 the map $L$ induces a toric map 
\[
P_L\colon Y_0\to Y.
\]
The map $P_L$ is unramified away from the toric divisors of $Y_0$.
Moreover, for any pre-log curve 
\[
\phi\colon \Bbb P^1\to Y
\]
 associated with $g_v$, there is a pre-log curve 
\[
\psi\colon \Bbb P^1\to Y_0
\]
 associated with $g_0$ such that
\[
\phi = P_L\circ\psi
\]
 holds (\cite[Lemma 5.3]{Nishinou_Toric_2006}).
In particular, the map $\varphi_0|_{C_{0, v}}$ factors through a map to $Y_0$, 
 and this map to $Y_0$ is an embedding whose image is a line (note that $Y_0$
 is an open subset of $\Bbb P^2$).
It follows that the map $\varphi_0|_{C_{0, v}}$ is an immersion away from the nodes of $C_0$.

When the vertex $v$ is 2-valent, the map $\varphi_0|_{C_{0, v}}$ is a ramified covering onto its image,
 where the ramification occurs at the two nodes of $C_0$ contained in $C_{0, v}$.
It follows that the restriction $\varphi_0|_{C_{0, v}}$ is an immersion away from the nodes.
 \qed\\ 

Suppose we have a deformation $\varphi_n\colon C_n\to \mathcal X$ of the log morphism 
 $\varphi_0\colon C_0\to \mathcal X$,
 where $C_n$ is a deformation of $C_0$ over $O_n$.
Then, our purpose is to extend it to the diagram
\[
\xymatrix{
C_n\ar[r]\ar[d] & C_{n+1}\ar[r]\ar[d] & \mathcal X \ar[d]\\
O_n\ar[r] & O_{n+1} \ar[r] & B
}
\]

This set up is similar to that of \cite[Section 7]{Nishinou_Toric_2006},
 and if we can show the existence of some log morphism $\varphi_{n+1}\colon C_{n+1}\to \mathcal X$, 
 then by the general theory of log smooth deformations, the space of such deformations is 
 a torsor under the group $H^0(C_0, \mathcal N_{\varphi_0})$, see \cite[Lemma 7.2]{Nishinou_Toric_2006}.
Here, 
 $\mathcal N_{\varphi_0} = \varphi_0^*\Theta_{X_0}/\Theta_{C_0}$ is the log normal sheaf 
 (see Definition \ref{def:lognormal}).

\begin{rem}
As we mentioned after \cref{eq:logcurve}, there are choices of log structures on $C_0$.
However, since the log normal sheaf $\mathcal N_{\varphi_0}$ does not depend on these choices, 
 the choice of a log structure on $C_0$ makes little impact on the argument that follows.
See also Remark \ref{rem:lift_log}.
\end{rem}

In \cite{Nishinou_Toric_2006} the existence of $\varphi_{n+1}$ is almost immediate, given that
 $\varphi_0$ is a map from a rational pre-stable curve to a degenerate toric variety.
In fact, by general deformation theory, the obstruction to the existence of such a map belongs to the cohomology 
 group $H^1(C_0, \mathcal N_{\varphi_0})$, which vanishes in the context of \cite{Nishinou_Toric_2006}.
In contrast, in the present case, the group $H^1(C_0, \mathcal N_{\varphi_0})$ is one-dimensional.
This implies that, given a deformation $\varphi_n\colon C_n\to \mathcal X$ of $\varphi_0$, 
 there is a potential obstruction to extending it to $\varphi_{n+1}$.
Thus, we need to check that the obstruction actually vanishes at every step $n$.
This section is dedicated to addressing this issue.
In particular, the argument presented here shares little overlap with that of \cite{Nishinou_Toric_2006}.

\subsection{Outline of the proof}
Since the proof of the existence of a deformation is somewhat lengthy, 
 we provide an outline of the argument in this subsection.
In Subsection \ref{subsec:N}, we compute the cohomology group $H^1(C_0, \mathcal N_{\varphi_0})$,
 where the potential obstruction to the deformation resides.
Specifically, we compute its dual, since it is easier to handle.
By Serre duality for nodal curves, the dual space of $H^1(C_0, \mathcal N_{\varphi_0})$ is naturally isomorphic to 
 the group 
\[
H^0(C_0, \mathcal N_{\varphi_0}^{\vee}\otimes \omega_{C_0}),
\] 
 where $\omega_{C_0}$ is the dualizing sheaf of $C_0$.
It turns out that the group $H^0(C_0, \mathcal N_{\varphi_0}^{\vee}\otimes \omega_{C_0})$
 is one-dimensional (see Proposition \ref{prop:dimension_of_obstruction}), and generated by 
 the pull-back of the relative logarithmic complex volume form of $\mathcal X$ (see Definition \ref{def:volumeform}).
Since we need to calculate the obstruction explicitly, having an explicit generator of the 
 group $H^0(C_0, \mathcal N_{\varphi_0}^{\vee}\otimes \omega_{C_0})$ is crucial.
Subsection \ref{subsec:pairing} recalls a method from \cite{Nishinou_Obstruction_2018}
 to compute the obstructions, given such an 
 explicit generator of $H^0(C_0, \mathcal N_{\varphi_0}^{\vee}\otimes \omega_{C_0})$, 
 see Proposition \ref{prop:coupling}.
 
After these preparations, we proceed with the calculation of the obstruction.
Assume we have a deformation $\varphi_n\colon C_n\to \mathcal X$ of $\varphi_0$.
To apply Proposition \ref{prop:coupling}, we need not only a representative of 
 a generator of the group $H^0(C_0, \mathcal N_{\varphi_0}^{\vee}\otimes \omega_{C_0})$,
 but also the obstruction cocycle itself, which belongs to 
 $H^1(C_0, \mathcal N_{\varphi_0})$.
A standard way to construct such a cocycle is as follows.
Let $\{\mathscr U_i\}$ be an open covering of the underlying topological space of $C_n$
 which is sufficiently fine so that the restriction $\varphi_0|_{\mathscr U_i}$ is an embedding if 
 $\mathscr U_i$ does not contain a node.
By the log smoothness of $\mathcal X$ over the base, 
 for each $\mathscr U_i$, there is a deformation $\varphi_{n+1}|_{\mathscr U_i}$ of $\varphi_n|_{\mathscr U_i}$.
Then, the difference between these deformations on the intersection $\mathscr U_i\cap \mathscr U_j$
 naturally gives a section of $\mathcal N_{\varphi_0}|_{\mathscr U_i\cap \mathscr U_j}$, 
 and the set of these sections forms the desired cocycle.

In general, computing the obstruction cocycle based on this construction is difficult.
However, there is a special case where this can be done, namely, the so-called semi-regular case.
For embedded curves on a smooth surface, the semiregularity means that 
 the sections of the dual space of the obstructions are obtained by restricting sections of
 the canonical sheaf of the ambient surface.
This is the case in our situation, since a generator of the dual space 
 $H^0(C_0, \mathcal N_{\varphi_0}^{\vee}\otimes \omega_{C_0})$ is the restriction of the logarithmic volume form of $X_0$.
However, we must take into account that the surface $X_0$ is singular and deforms,
 as well as the fact that the map $\varphi_0$
 is not an embedding.

To address this point, in Subsection \ref{subsec:std}, 
 we introduce a specific type of open covering of $C_0$, distinct from the covering $\{\mathscr U_i\}$ mentioned earlier, 
 which is associated with adjacent two vertices of the graph $\Gamma$.
By rescaling and subdivision, we can assume that for any 3-valent vertex of $\Gamma$, 
 all adjacent vertices are 2-valent, as discussed at the beginning of Subsection \ref{subsec:std}.
Consider two adjacent vertices $v_1$ and $v_2$ of $\Gamma$, where one of them is 3-valent.
Let $e_{v_1v_2}$ be the 
 edge connecting them.
A small neighborhood of $e_{v_1v_2}$ naturally extends to a tropical curve $(\Gamma_{v_1v_2}, h_{v_1v_2})$,
 which consists of one bounded edge $e_{v_1v_2}$ and four unbounded edges.
We make use of the fact that, up to isomorphisms, there is a unique tropical curve with two vertices
 $(\Gamma_s, h_s)$ in $\Bbb R^2$, such that any $(\Gamma_{v_1v_2}, h_{v_1v_2})$
 is the image of $(\Gamma_s, h_s)$ by an integral linear map $\Phi$ (see Lemma \ref{lem:Phi}).
There are degenerating families of toric varieties $\overline{\mathcal Y}_{v_1v_2}$ and 
 $\overline{\mathcal Y}_{V_1V_2}^s$ associated with the tropical curves 
 $(\Gamma_{v_1v_2}, h_{v_1v_2})$ and $(\Gamma_s, h_s)$, respectively.
Moreover, there is a map 
\[
P_{\Phi}\colon \overline{\mathcal Y}_{V_1V_2}^s\to \overline{\mathcal Y}_{v_1v_2}
\]
 induced by $\Phi$ (see Definition \ref{def:Pphi}).
An open subset $\mathcal Y_{v_1v_2}$ of the variety $\overline{\mathcal Y}_{v_1v_2}$ naturally embeds in $\mathcal X$.
Thus, we can restrict the map $\varphi_0$ to a suitable open subset of $C_0$ so that its image is contained in 
 $\mathcal Y_{v_1v_2}$.
Let us write it as 
\[
\varphi_0|_{\mathscr D}\colon \mathscr D\to \mathcal Y_{v_1v_2},
\] 
 where $\mathscr D$ is an open subset of the nodal union of two copies of regular rational curves.
An important observation is that the map $\varphi_0|_{\mathscr D}$ factors through 
 an open subset $\mathcal Y_{V_1V_2}^s$ of $\overline{\mathcal Y}_{V_1V_2}^s$.
Furthermore, if we have a deformation $\varphi_n\colon C_n\to \mathcal X$ of $\varphi_0$, 
 then the restriction of $\varphi_n$ to $\mathscr D$ also factors through $\mathcal Y_{V_1V_2}^s$ 
 (see Lemma \ref{lem:v_1v_2} and Proposition \ref{prop:liftmap}):

\[
\xymatrix{
C_n\supset \mathscr D \ar[rr]^{\psi_{n, V_1V_2}}\ar[rrd]_{\varphi_n|_{\mathscr D}} 
 && \mathcal Y_{V_1V_2}^s \ar[d]^{P_{\Phi}}\\
 && \mathcal Y_{v_1v_2} 
}
\]

The merit of lifting the map $\varphi_n$ is that the map $\psi_{n, V_1V_2}$ becomes an embedding,
 allowing the argument of \cite{Nishinou_Obstruction_2018} to be applied.
Note that if $v_3$ is another vertex adjacent to $v_1$, a similar diagram
 exists involving another open subset $\mathscr D'\subset C_0$ and the map $\psi_{n, V_1V_3}$.
To construct a deformation of $\varphi_n$, it suffices to construct 
 deformations of its restrictions $\varphi_n|_{\mathscr D}$ compatibly in the sense that, on the intersection
 of open subsets such as $\mathscr D\cap \mathscr D'$, the deformations coincide.
On the other hand, if we can deform $\psi_{n, V_1V_2}$ compatibly in a similar sense, 
 then these deformations also give a deformation of $\varphi_n$. 
 
In Subsection \ref{subsec:lift}, we formalize the precise meaning of the compatibility of the maps
 $\psi_{n, V_1V_2}$ (see Proposition \ref{prop:compatibility} and the paragraphs that follow).
Here, the assumption (\ref{eq:assump}) in the introduction is invoked.
Based on this, in Subsections \ref{subsec:obstcocycle},
 we show that the obstruction cocycle can be computed by taking deformations of $\psi_{n, V_1V_2}$
 and their differences on the intersection.
For technical convenience, these deformations are considered on smaller open subsets 
 $\{\mathscr U_i\}$ of $\mathscr D$.
We verify that the cocycle defined in this way is well-defined in Lemma \ref{lem:cocycle},
 and show that it indeed represents the obstruction to deforming $\varphi_n$ in Proposition \ref{prop:locallifts}.
These are the most technical parts of the proof, and we stated the motivation for these constructions 
 at the beginning of Subsection \ref{subsec:lift}.

Thus, the problem reduces to computing the obstruction class defined via the lifts of 
 $\varphi_n$ to $\psi_{n,V_1V_2}^s$.
Although $X_0$ is singular, the image of the open subsets of the form $\mathscr D\cap \mathscr D'$
 does not intersect the singular locus of $X_0$.
Consequently, the argument from \cite{Nishinou_Obstruction_2018}
 can be applied in this context.
This computation is carried out in Subsection \ref{subsec:maincal}, and we conclude that the obstruction vanishes 
 (see Theorem \ref{thm:main}).

\subsection{Description of the dual of obstruction space}\label{subsec:N}
Let $\varphi_0\colon C_0\to X_0$ be a pre-log curve in $X_0$,
 associated with a given connected,
 at most\footnote{A 3-valent tropical curve we start with may acquire 2-valent vertices 
 due to subdivisions in the process of construction of pre-log curves.} 3-valent parametrized tropical curve $h\colon\Gamma\to S$, as described in \cref{sec:pre-log}.
%We will compute the obstruction to deforming the pre-log curve in $X_0$ into a parametrized algebraic curve in $X$. 
We use the notations introduced at the beginning of \cref{sec:pre-log}.
The toroidal scheme $\mathcal X$ induces a log smooth structure on $X_0$ over $(B, 0)$.
Here, $B$ is a small disc centered at $0$, equipped with a log structure associated with the divisor $0$.
The curve $C_0$ is equipped with a natural log structure such that the map $\varphi_0$ extends to a log morphism
 over $B$, as described in Subsection \ref{subsec:logdeformation}.
Let 
$
\mathcal N_{\varphi_0} = \varphi_0^*\Theta_{X_0}/\Theta_{C_0}
$
be the log normal sheaf of $\varphi_0$, see Definition \ref{def:lognormal}.

\begin{defin}
A \emph{flag} of the graph $\Gamma$ is a pair $(v,e)$ consisting of a vertex $v$ and an edge $e$ incident to $v$.
Let $F(\Gamma)$ denote the set of flags of $\Gamma$.
\end{defin}

The following is a special case of the result proved in \cite[Theorem 57]{Nishinou_2025}.
For the reader's convenience, we provide a proof.
\begin{lem}\label{lem:description_of_dual}
	Consider the cohomology group 
	$H^0(C_0,\mathcal N_{\varphi_0}^\vee\otimes\omega_{C_0})$, where $\omega_{C_0}$ denotes the dualizing sheaf of $C_0$.
	It is naturally isomorphic to the group of maps
	\[
	u\colon F(\Gamma)\longrightarrow N^\vee_\bbC, \qquad (v,e)\longmapsto u_{v,e}
	\]
	satisfying the following conditions:
	\begin{enumerate}
		\item \label{lem:description_of_dual:u} For every flag $(v, e)$, we have
		$u_{v,e}\in w_{v,e}^{\perp}\subset N^\vee_\bbC$, 
		 where $w_{v,e}$ is the weight vector of the edge $e$ at $v$,
		 see Definition \ref{def:parametrized_tropical_curve_in_R^2}.
		\item \label{lem:description_of_dual:e} For every edge $e$ of $\Gamma$ with endpoints $v,v'$,
		we have
		\[u_{v,e}+u_{v', e} = 0.\]
		\item \label{lem:description_of_dual:v} For every vertex $v$ of $\Gamma$, we have
		\[\sum_{e\ni v} u_{v,e} = 0,\]
		where the sum is taken over all the edges incident to $v$.
	\end{enumerate}
	Here, the group structure on the maps is given by the pointwise addition.
\end{lem}
\begin{proof}
	Let $\tC_0\to C_0$ be the normalization map.
	Then, the dualizing sheaf $\omega_{C_0}$
	is the invertible subsheaf of the pushforward of the sheaf of 1-forms on $\tC_0$
	 with logarithmic poles at the points lying over the nodes of $C_0$.
	Specifically, $\omega_0$ corresponds to those meromorphic forms 
	 whose sum of residues at the two points lying over every node of $C_0$ is zero 
	 (cf.\ \cite[\S 10.2]{Arbarello_Geometry_of_algebraic_curves_II} 
	 and \cite[\S II.6]{Barth_Compact_complex_surfaces}).
	Let us take  be an element 
	\[
	\Psi\in H^0(C_0, \mathcal N_{\varphi_0}^{\vee}\otimes\omega_{C_0}).
	\]
	For a vertex $v$ of $\Gamma$, let $\{e_j\}_{j\in J}$ be the edges incident to $v$, where
	 $J = \{1, 2\}$ or $\{1, 2, 3\}$.
	Denote by $C_{0,v}$ the irreducible component of $C_0$ corresponding to $v$, 
	 and by $\{p_{e_j}\}_{j\in J}$ the nodes corresponding to the edges $\{e_j\}_{j\in J}$.
	Recall that $\mathcal X$ is the quotient of the infinite type toric variety $\widetilde{\mathcal X}$
	 (see Definition \ref{def:M2}), and we have 
	\[
	\Theta_{\widetilde{\mathcal X}/B}\simeq N\otimes\mathcal O_{\widetilde{\mathcal X}}.
	\]
	Since the action of $\Lambda$ in \cref{eq:assump} induces the identity map on the sections of
	 the constant subsheaf $N\otimes 1$ of the sheaf of relative logarithmic  
	vector fields  $\Theta_{\widetilde{\mathcal X}/B}$, we obtain
	\[
	\Theta_{\mathcal X/B}\simeq N\otimes\mathcal O_{\mathcal X}.
	\]
	This property holds more generally for any $\Lambda$,
	not necessary restricted to the specific form given in \cref{eq:assump}.
	In particular, by restricting this isomorphism to the special fiber, we have
	\[
	\Theta_{X_0}\simeq N\otimes\mathcal O_{X_0}.
	\]
	
	Recall that the sheaf $\mathcal N_{\varphi_0}$ is a quotient sheaf of $\varphi_0^*\Theta_{X_0}$ defined by the 
	 following exact sequence of locally free sheaves on $C_0$:
	\[
	0\to \Theta_{C_0}\to \varphi_0\Theta_{X_0}\to \mathcal N_{\varphi_0}\to 0.
	\]
	Dualizing it and tensoring $\omega_{C_0}$, we have
	\[
	0\to \mathcal N_{\varphi_0}^{\vee}\otimes \omega_0\to \varphi_0^*(\Theta_{X_0})^{\vee}\otimes\omega_{C_0}
	 \to (\Theta_{C_0})^{\vee}\otimes_{C_0}\to 0.
	\]
	Here, by the above argument, we have
	\[
	\varphi_0^*(\Theta_{X_0})^{\vee}\cong N^{\vee}\otimes \mathcal O_{C_0}.
	\]
	Thus, the restriction $\Psi|_{C_{0,v}}$ can be viewed as a differential form on $C_{0,v}$ 
	with values in $N^\vee_\bbC$, allowing logarithmic poles at the nodes $\{p_{e_j}\}_{j\in J}$.
	For $j\in J$,
	 define $u_{v,e_j}$ to be the vector valued residue of $\Psi|_{C_{0,v}}$ at the node $p_{e_j}$.

	Since the fiber of $\Theta_{C_0}\subset \varphi_0^*\Theta_{X_0}$ at $p_{e_j}$ is generated by the weight vector $w_{v,e_j}\in N$, the exact sequence implies that the fiber at $p_{e_j}$ is generated by an annihilator of
	 the vector $w_{v,e_j}$.
	This is precisely the condition (\ref{lem:description_of_dual:u}) for $u_{v, e_j}$.
	The condition (\ref{lem:description_of_dual:e}) follows from the explicit description of the dualizing sheaf $\omega_{C_0}$
	given above.
	Namely, the sum of the residues at the two points lying over a node must be zero.
	The condition (\ref{lem:description_of_dual:v}) follows from the residue theorem.
	
	Conversely, the same reasoning shows that any map $u\colon F(\Gamma)\to N^\vee_\bbC$ satisfying the conditions (1-3) uniquely determines an element of $H^0(C_0, \mathcal N_{\varphi_0}^{\vee}\otimes\omega_{C_0})$.
	Thus, the lemma is proved.
\end{proof}

Now, we provide an explicit description of elements of $H^0(C_0,\mathcal N_{\varphi_0}^\vee\otimes\omega_{C_0})$.
Let $\{a_1, a_2\}$ be any basis of $N^{\vee}$, and let $x$ and $y$ denote the corresponding functions on 
$\widetilde{\mathcal X}$.
The 2-form $\frac{dx}{x}\wedge\frac{dy}{y}$ defined on $\widetilde{\mathcal X}$ is the normalized
 logarithmic complex volume form relative to the base.
This form is independent of the choice of toric coordinates, provided that an orientation of $N_{\Bbb R}^{\vee}$ is fixed
 and the chosen basis respects this orientation.
It is holomorphic on the complement of the central fiber, and has logarithmic poles
 along the singular locus of the central fiber. 
Additionally, it descends to $\mathcal X$ since it is invariant under the action of $\Lambda$.

\begin{defin}\label{def:volumeform}
We write this form by $\Psi_{{\mathcal X}}$
 and its restriction to $X_0$ by $\Psi_{X_0}$.
Let $\Psi_{C_0}$ be the pullback of $\Psi_{X_0}$ to $C_0$ by $\varphi_0$.
\end{defin}

Let $\Omega_{X_0} = (\Theta_{X_0})^{\vee}$ be the sheaf of logarithmic 1-forms on $X_0$.
\begin{lem}\label{lem:adj}
The pullback $\varphi_0^*(\wedge^2\Omega_{X_0})$ is isomorphic to the sheaf 
 $\mathcal N_{\varphi_0}^{\vee}\otimes \omega_{C_0}$.
In particular, $\Psi_{C_0}$ provides a non-zero section of $\mathcal N_{\varphi_0}^{\vee}\otimes \omega_{C_0}$.
\end{lem}
\proof
By the description of the dualizing sheaf $\omega_{C_0}$, it follows that $\omega_{C_0}$ is 
 isomorphic to the dual of $\Theta_{C_0}$.
Then, dualizing the defining exact sequence 
\[
0\to \Theta_{C_0}\to \varphi_0^*\Theta_{X_0}\to \mathcal N_{\varphi_0}\to 0,
\] 
 as in the proof of the previous proposition, we obtain
\[
0\to \mathcal N_{\varphi_0}^{\vee}\to \varphi_0^*\Omega_{X_0}\to \omega_{C_0}\to 0.
\]
The claim follows from this.\qed

\begin{prop} \label{prop:dimension_of_obstruction}
	The space $H^0(C_0,\mathcal N_{\varphi_0}^\vee\otimes\omega_{C_0})$ has dimension one, and
	 generated by $\Psi_{C_0}$.
\end{prop}
\begin{proof}
Since we have 
\[
\mathcal N_{\varphi_0}^\vee\otimes\omega_{C_0}\cong \varphi_0^*(\wedge^2\Omega_{X_0})\cong
 \wedge^2 N_{\Bbb C}\otimes \mathcal O_{C_0}
\]
 by Lemma \ref{lem:adj}, the claim immediately follows.
However, since it is better to be acquainted with the explicit form of the sections of
 $\mathcal N_{\varphi_0}^\vee\otimes\omega_{C_0}$ for understanding the arguments in the following sections, 
 we provide a proof based on the description in Lemma \ref{lem:description_of_dual}.
 
	As the graph $\Gamma$ is at most 3-valent, \cref{lem:description_of_dual}(\ref{lem:description_of_dual:u})
	 and (\ref{lem:description_of_dual:v})
	 imply that, for every vertex $v$, any one of the vectors $\{u_{v,e_i}\}$ determines the other vectors uniquely.
Namely, note that by \cref{lem:description_of_dual}(\ref{lem:description_of_dual:u}), 
 the direction of the vectors $u_{v, e_i}$ are fixed.
Thus, it suffices to determine their coefficients.
The claim is obvious when the vertex is 2-valent by \cref{lem:description_of_dual}(\ref{lem:description_of_dual:v}).
When the vertex is 3-valent, it is likewise straightforward to see that the condition 
 \cref{lem:description_of_dual}(\ref{lem:description_of_dual:v}) implies that fixing the coefficient of one of 
 the vectors determines those of the other two vectors as well.

	Combined with the condition \cref{lem:description_of_dual}(\ref{lem:description_of_dual:e}), the connectedness of $\Gamma$
	 ensures that specifying $u_{v, e}$ for a single flag $(v, e)$ 
	 determines the map $u\colon F(\Gamma)\to N^\vee_\bbC$ completely.
	Thus, the space $H^0(C_0,\mathcal N_{\varphi_0}^\vee\otimes\omega_{C_0})$ has dimension at most one.
	By Lemma \ref{lem:adj}, it is non-zero and generated by $\Psi_{C_0}$.
\end{proof}

\subsection{The pairing between the obstruction class and its dual}\label{subsec:pairing}
Now, we turn to the calculation of the obstruction class via \v{C}ech cohomology.
First, we introduce a representation of \v{C}ech 1-cocycles with coefficients in
 invertible sheaves
 on nodal curves, tailored for our purpose.
Let $C_0$ be a pre-stable curve.
We introduce an open covering $\{\mathscr U_i\}$ of $C_0$ as follows.
\begin{defin}\label{def:scrU}
Let 
$
\{\mathscr U_i\}_{i\in I}
$
 be an open covering of $C_0$ satisfying the following properties:
\begin{itemize}
\item Each $\mathscr U_i$ is homeomorphic either to a disc or to the union of two discs attached
 at their centers.
For the latter case, denote the irreducible components of $\mathscr U_i$
 by $\mathscr U_{i}^{(1)}$ and $\mathscr U_{i}^{(2)}$.
These subsets are locally closed in $C_0$.
\item 
Every node of $C_0$ is contained in an exactly one open subset in $\{\mathscr U_i\}$.
\item If $\mathscr U_i \cap \mathscr U_j \neq \emptyset$ for $\mathscr U_i \neq \mathscr U_j$, then $\mathscr U_i \cap \mathscr U_j$ is homeomorphic to a disc.
\end{itemize}
\end{defin}

Let $\mathcal L$ be an invertible sheaf on $C_0$.
An $\mathcal L$-valued \v{C}ech 1-cocycle is given by a collection of holomorphic sections defined on
 the intersections $\mathscr U_i\cap \mathscr U_j$.
If these sections can be expressed as the differences of meromorphic sections on $\mathscr U_i$
 or on the locally closed subsets $\mathscr U_{i}^{(j)}$,
  this representation facilitates the computation of the cohomology class that the cocycle represents.
Specifically, the cohomology class can be calculated from the data of poles of the meromorphic sections.
In the following, we utilize this observation to calculate the obstruction.

In this subsection, we provide a precise statement of the above claim.
Let 
\[
\xi_i
\]
 be a meromorphic section of $\mathcal L|_{\mathscr U_i}$ 
 when $\mathscr U_i$ is a disc.
Let 
\[
\xi_{i}^{(j)},\;\; j = 1, 2
\]
 be meromorphic sections of $\mathcal L|_{\mathscr U_{i}^{(j)}}$, 
 when $\mathscr U_i$ is the union of discs.
We require the following condition:
\begin{equation}\label{eq:xi}
\text{The differences $\xi_i-\xi_j$ and $\xi_i-\xi_{k}^{(j)}$ are holomorphic.}
\end{equation}
These differences are considered on the intersection $\mathscr U_i\cap \mathscr U_j$ and
 $\mathscr U_i\cap \mathscr U_{k}^{(j)}$, respectively.

\begin{lem}\label{lem:xi}
The set of sections $\{\xi_i\}$ and $\{\xi_{i}^{(j)}\}$ define a \v{C}ech 1-cocycle $\xi$ 
 with values in $\mathcal L$ associated with the covering $\{\mathscr U_i\}_{i\in I}$.
\end{lem}
\proof
If both $\mathscr U_i$ and $\mathscr U_j$ are discs, then 
 assign the section $\xi_i-\xi_j$ to $\mathscr U_i\cap\mathscr U_j$.
If $\mathscr U_i$ is a disc and $\mathscr U_k$ is a union of discs such that 
 $\mathscr U_i\cap \mathscr U_k\neq \emptyset$, then there is a unique $j\in \{1, 2\}$ such that 
\[
\mathscr U_i\cap \mathscr U_k = \mathscr U_i\cap \mathscr U_{k}^{(j)}.
\]
In this case, assign the section $\xi_i-\xi_{k}^{(j)}$ to $\mathscr U_i\cap \mathscr U_k$.
The cocycle condition follows directly from the construction. \qed\\

Now, let $\Psi$ be an element of $H^0(C_0, \mathcal L^{\vee}\otimes \omega_{C_0})$, 
 the dual space of $H^1(C_0, \mathcal L)$.
It can be shown that 
 the pairing between the cohomology class associated with $\xi = \{\xi_i\}\cup \{\xi_{i}^{(j)}\}$ and $\Psi$ is 
 calculated as follows.
Namely, let 
\[
\{q_{\alpha}\}_{\alpha\in A}\subset C_0
\]
 be the set of points where some $\xi_i$ or $\xi_{i, j}$ has a pole, excluding the nodes of $C_0$.
For each $\alpha$, choose any locally closed subset $\mathscr U_i$ or $\mathscr U_{i}^{(j)}$
 which contains $q_{\alpha}$, 
 and denote the corresponding local section by $\xi_{\alpha}$.
Let 
$
\{p_i\}
$
 be the set of nodes of $C_0$.
By assumption, each of the $p_i$ is contained in a unique open subset $\mathscr U_i$.
Let 
\[
p_{i}^{(1)}, \;\; p_{i}^{(2)}
\]
 be the corresponding points on the irreducible components
 $\mathscr U_{i}^{(1)}$ and $\mathscr U_{i}^{(2)}$, respectively. 
These components have the associated local sections $\xi_{i}^{(1)}$ and $\xi_{i}^{(2)}$.
For the given meromorphic sections $\xi_{\alpha}$ and $\xi_{i}^{(j)}$ on the corresponding locally closed subsets
 of $C_0$, 
 the fiberwise pairing between them and $\Psi$ yield  meromorphic 1-forms $\Psi(\xi_{\alpha})$
 and $\Psi(\xi_{i}^{(j)})$ on the respective domains.
Then, the following proposition holds.
\begin{prop}\cite[Proposition 10]{Nishinou_Obstruction_2018}\label{prop:coupling}
The pairing between the classes $\xi\in H^1(C_0, \mathcal L)$ and 
$\Psi\in H^0(C_0, \mathcal L^{\vee}\otimes \omega_{C_0})$ is given by the following formula:
\[
\langle\Psi, \xi\rangle
 = \sum_{\alpha\in A} \res_{q_{\alpha}}\Psi(\xi_{\alpha}) + \sum_{p_{i}^{(j)}} \res_{p_{i}^{(j)}}\Psi(\xi_{i}^{(j)}).
\] 
Here, $\res_p \phi$ denotes the residue of a meromorphic 1-form $\phi$ at the point $p$, and
 $\{q_{\alpha}\}_{\alpha\in A}\subset C_0$ is the set of poles of the local sections 
  $\xi_i$ which are not the nodes of $C_0$, as described above.
\qed
\end{prop}
Several remarks regarding this formula are provided below.
\begin{rem}\label{rem:coupling}
\begin{enumerate}
\item By the condition (\ref{eq:xi}), the sum in Proposition \ref{prop:coupling} is
 independent of the choice of $\xi_{\alpha}$ associated with a pole $q_{\alpha}$.
\item Since sections of the dualizing sheaf $\omega_{C_0}$ may exhibit logarithmic singularities at the nodes of $C_0$,
 there can be contributions to the sum even if $\xi_{i}^{(j)}$ is holomorphic at $p_{i}^{(j)}$.
\item It can be shown that this definition induces the natural non-degenerate pairing between 
 $H^0(C_0, \mathcal L^{\vee}\otimes \omega_{C_0})$ and $H^1(C_0, \mathcal L)$.
In particular, a class $\xi$ in $H^1(C_0, \mathcal L)$ vanishes if and only if the pairing is zero for any class
 $\Psi$ in $H^0(C_0, \mathcal L^{\vee}\otimes \omega_{C_0})$.
Moreover, any class in $H^1(C_0, \mathcal L)$ can be represented by some set of local 
 sections $\{\xi_i\}$ and $\{\xi_{i}^{(j)}\}$ as described above.
See \cite{Nishinou_Obstruction_2018} for details.
\item In this paper, since $H^0(C_0, \mathcal L^{\vee}\otimes \omega_{C_0})$ is one-dimensional 
 for the sheaf $\mathcal L = \mathcal N_{\varphi_0}$ relevant to us, a class $\xi\in H^1(C_0, \mathcal N_{\varphi_0})$
 vanishes if and only if its pairing with the generator $\Psi_{C_0}$ of
 $H^0(C_0, \mathcal N_{\varphi_0}^{\vee}\otimes \omega_{C_0})$ vanishes, 
 see Proposition \ref{prop:dimension_of_obstruction}.
\end{enumerate}
\end{rem}

\subsection{Standard tropical curve with two vertices}\label{subsec:std}
To simplify the calculation of the obstructions to deforming the given map $\varphi_0\colon C_0\to X_0$,
 we modify the associated tropical curve slightly.
Specifically, by applying a base change, introducing 2-valent vertices, and refining the subdivision of the
 2-torus $S = N_{\Bbb R}/\overline{\Lambda}$, we can assume the following about the tropical curve $(\Gamma, h)$:
\begin{itemize}
\item For every 3-valent vertex, all the adjacent vertices are 2-valent.
\item The integral length of the image of any edge of $\Gamma$ is equal to its weight.
In other words, in the metric graph $\Gamma$, 
 the length of each edge is one, in terms of Definition \ref{def:metric_graph}.
\end{itemize}
Note that since $(\Gamma, h)$ is a tropical curve on $S$, there is no infinite vertex.

Let 
\[
h_s\colon \Gamma_s\setminus V_{\infty}(\Gamma_s)\to \Bbb R^2
\]
 be the tropical curve defined as follows.
Namely, the connected open graph $\Gamma_s\setminus V_{\infty}(\Gamma_s)$
 has two vertices $V_1$ and $V_2$.
Here, both vertices are 2-valent, or $V_1$ is 3-valent and $V_2$ is 2-valent.
In particular, there is a single bounded edge $E$ connecting $V_1$ and $V_2$.
All the edges have weight one.
The map $h_s$ satisfies the following conditions:
\begin{itemize}
\item $h_s$ maps $V_1$ to the origin, and $V_2$ to $(1, 0)$.
\item When $V_1$ is 3-valent, the unbounded edges attached to $V_1$ have weight vectors $(0, 1)$ and $(-1, -1)$. 
\end{itemize}
These conditions, together with the balancing condition, uniquely determine $h_s$.
We fix an orientation of the target space $\Bbb R^2$ of $h_s$
so that the ordered set of vectors 
\[
(1, 0),\;\; (0, 1), \;\; (-1, -1) 
\]
is compatible with the orientation.
\begin{defin}\label{def:std}
	We call this tropical curve $(\Gamma_s, h_s)$ with two vertices \emph{standard}.
\end{defin}

Now, take a pair of adjacent vertices $v_1$ and $v_2$ of $\Gamma$.
Here, $v_1$ is 2- or 3-valent, while $v_2$ is always 2-valent.
Then, the restriction of $h$ to a neighborhood $\mathscr V_{v_1v_2}$ of the edge $e_{v_1v_2}$ connecting $v_1$ and $v_2$
gives a tropical curve 
\[
h_{v_1v_2}\colon \Gamma_{v_1v_2}\to N_{\Bbb R}\cong\Bbb R^2
\]
in a natural manner, determined uniquely up to 
integral affine transformations on $\Bbb R^2$.
Here, $\Gamma_{v_1v_2}$ is an open graph with two vertices obtained from $\mathscr V_{v_1v_2}$
by extending the edges other than $e_{v_1v_2}$ to infinity.
By applying parallel transportation, we assume that the vertex $v_1$ is placed at the origin $(0, 0)\in N_{\Bbb R}$.
Recall that we have fixed an orientation on $N_{\Bbb R}$, hence also on $S$.

\begin{lem}\label{lem:Phi}
	There is an isomorphism between $\Gamma_s$ and $\Gamma_{v_1v_2}$
	as non-weighted metric graphs, so that under the identification of them by this isomorphism,
	the following holds. 
	There exists a unique orientation-preserving 
	integral linear map $\Phi\in {\rm Hom}(\Bbb Z^2, N)$ satisfying
	\[
	\Phi_{\Bbb R}\circ h_s = h_{v_1v_2},
	\]
	where $\Phi_{\Bbb R}\colon \Bbb R^2\to N_{\Bbb R}$
	is the map naturally induced by $\Phi$.
	Moreover, when both $v_1$ and $v_2$ are 2-valent, then we can take $\Phi$ so that 
	\[
	\sharp\Coker \Phi = W_{e_{v_1v_2}}, 
	\]
	where $W_{e_{v_1v_2}}$ is the weight of the edge $e_{v_1v_2}$ (see Definition \ref{def:weight}).
\end{lem}
\proof
When $v_1$ is 3-valent, 
let $e_1$ and $e_2$ be unbounded edges of $\Gamma_{v_1v_2}$ such that the ordered set of 
edges
\[
e_{v_1v_2}, \;\; e_1,\;\; e_2
\]
is compatible with the orientation on $S$.
Then, 
the map $\Phi$ is uniquely determined by the condition that it sends the vectors $(0, 1)$ and $(-1, -1)$
to the weight vectors of the unbounded edges $e_1$ and $e_2$ attached to $v_1$, respectively.
The balancing condition forces that the weight vector of $E$, which is $(1, 0)$, is sent to 
the weight vector of $e_{v_1v_2}$. 
The claim for the case where both $v_1$ and $v_2$ are 2-valent is obvious.\qed

\begin{rem}
	The orientations on the domain and the target of $\Phi_{\Bbb R}$ are introduced 
	to uniquely fix $\Phi_{\Bbb R}$.
	Although this choice has little impact to the arguments in this paper, we usually select this choice.
\end{rem}

For notational simplicity, we denote $\Phi_{\Bbb R}$ by $\Phi$.
Note that the map $\Phi$ also induces a corresponding map in the setting of holomorphic curves.
Specifically, consider the image $h_{v_1v_2}(\Gamma_{v_1v_2})$ 
 placed on the affine plane $\Bbb R^2\times \{1\}\subset \Bbb R^3$,
 and let $\overline{\mathcal C}_{v_1v_2}$ denote the closure of the cone over this image.
It naturally carries the structure of an incomplete fan whose one-dimensional cones are the rays spanned
 by the vertices of $h_{v_1v_2}(\Gamma_{v_1v_2})$ as well as those contained in 
 $\overline{\mathcal C}_{v_1v_2}\cap (\Bbb R^2\times\{0\})$.
Similarly, let $\overline{\mathcal C}_s$ be the cone in $\Bbb R^3$ obtained from $(\Gamma_s, h_s)$ 
 by the same construction as $\overline{\mathcal C}_{v_1v_2}$.
\begin{defin}\label{def:mathcalY}
Let $\overline{\mathcal Y}_{v_1v_2}$ be the toric variety associated with the cone $\overline{\mathcal C}_{v_1v_2}$.
Let $\mathcal Y_{v_1v_2}$
 be the open toric subvariety of $\overline{\mathcal Y}_{v_1v_2}$,
 obtained by removing toric divisors corresponding to the rays in 
 $\overline{\mathcal C}_{v_1v_2}\cap \Bbb R^2\times\{0\}$.
Similarly, let $\mathcal Y_{V_1V_2}^s\subset \overline{\mathcal Y}_{V_1V_2}^s$ be the toric varieties 
 associated with the cone $\overline{\mathcal C}_s$ analogous to 
 $\mathcal Y_{v_1v_2} \subset
 \overline{\mathcal Y}_{v_1v_2}$.
\end{defin}
Note that $\mathcal Y_{v_1v_2}$
 has a natural embedding into $\mathcal X$.
There are natural toric maps $\overline{\mathcal Y}_{v_1v_2}\to\Spec\Bbb C[t]$
 and $\overline{\mathcal Y}_{V_1V_2}^s\to \Spec\Bbb C[t]$.

\begin{defin}\label{def:Y}
Let $\overline Y_{0, v_1v_2}$ be the fiber of $\overline{\mathcal Y}_{v_1v_2}\to\Spec\Bbb C[t]$ over 
 $0\in \Spec\Bbb C[t]$, the point corresponding to the maximal ideal $(t)$.
This variety $\overline Y_{0, v_1v_2}$ naturally embeds into $\mathcal X$.
Let $Y_{0, v_1v_2}$ be the fiber of 
 $\mathcal Y_{v_1v_2}\to\Spec\Bbb C[t]$ over $0\in \Spec\Bbb C[t]$.
Similarly, let $\overline Y_{0, V_1V_2}^s$ and $Y_{0, V_1V_2}^s$ be the central fibers
 of $\overline{\mathcal Y}_{V_1V_2}^s$ and $\mathcal Y_{V_1V_2}^s$, respectively.
\end{defin}

Note that the cone $\overline{\mathcal C}_{v_1v_2}$ is obtained from 
 $\overline{\mathcal C}_s$ by the integral linear map
\begin{equation}\label{eq:pmat}
\begin{pmatrix}
a_{11} & a_{12} & 0\\
a_{21} & a_{22} & 0\\
0 & 0 & 1
\end{pmatrix},
\end{equation}
 where $\Phi = \begin{pmatrix}
   a_{11} & a_{12}\\
   a_{21} & a_{22}
  \end{pmatrix}$.

\begin{defin}\label{def:Pphi}
Let
\[
P_{\Phi}\colon\overline{\mathcal Y}_{V_1V_2}^s\to \overline{\mathcal Y}_{v_1v_2}
\]
 be the toric map associated with the integral linear map (\ref{eq:pmat}).
We also denote its restriction to ${\mathcal Y}_{V_1V_2}^s$ by $P_{\Phi}$.
\end{defin}

Let $D_0$ be the nodal union
\[
C_{0, v_1}\cup C_{0, v_2}
\]
 of irreducible components of $C_0$ corresponding to the vertices $v_1, v_2\in V(\Gamma)$.
Let 
\[
p = C_{0, v_1}\cap C_{0, v_2}
\]
 be the node of $D_0$.
We assume $D_{0, v_1}$ has one or two marked points, distinct from $p$, corresponding to the nodes
 in $C_0$.
Similarly, we assume $D_{0, v_2}$ has a single marked point distinct from $p$.
Then, the dual intersection graph of $D_0$ is $\Gamma_{v_1v_2}$.
Here, recall that $v_1$ is either 2- or 3-valent, while $v_2$ is always 2-valent.
The following can be easily verified.

\begin{lem}\label{lem:v_1v_2}
Any pre-log curve 
\[
\phi_{0, v_1v_2}\colon D_0\to \overline Y_{0, v_1v_2}
\]
 of type $(\Gamma_{v_1v_2}, h_{v_1v_2})$ (see Definition \ref{def:pre-log})
 can be written as the composition
\[
P_{\Phi}\circ \psi_{0, V_1V_2},
\] 
 where $\psi_{0, V_1V_2}\colon D_0\to \overline Y_{0, V_1V_2}^s$ is a pre-log curve of type  $(\Gamma_s, h_s)$. 
More precisely, let $q$ be any element of $P_{\Phi}^{-1}(\phi_{0, v_1v_2}(p))$.
Let $W_{e_{v_1v_2}}$ be the weight of the unique bounded edge of $\Gamma_{v_1v_2}$.
Then, there are $W_{e_{v_1v_2}}^2$ distinct curves
\[
\psi_{0, V_1V_2}^{(i)}\colon D_0\to \overline Y_{0, V_1V_2}^s,\;\; i = 1, \dots, W_{e_{v_1v_2}}^2,
\] 
 satisfying the conditions
\begin{itemize}
\item $\phi_{0, v_1v_2} = P_{\Phi}\circ \psi_{0, V_1V_2}^{(i)}$, and
\item $\psi_{0, V_1V_2}^{(i)}(p) = q$.
\end{itemize} 
\end{lem}

\proof
Let $\overline Y_{0, 1}$ and $\overline Y_{0, 2}$ be the irreducible components of 
 $\overline Y_{0, v_1v_2}$. 
Assume that 
 the component $C_{0, v_i}$ is mapped to $\overline Y_{0, i}$ by $\phi_{0, v_1v_2}$, for $i = 1, 2$.
Similarly, let $\overline Y_{0, i}^s$ be the irreducible components of 
 $\overline Y_{0, V_1V_2}^s$ to which $C_{0, v_i}$
 is mapped by $\psi_{0, V_1V_2}$.

A map $\psi_{0,V_1V_2}$ satisfying the conditions
\[
\phi_{0, v_1v_2} = P_{\Phi}\circ \psi_{0, V_1V_2}, \;\; \psi_{0, V_1V_2}(p) = q
\]
 is 
 determined by the data of maps
\[
\psi_1\colon C_{0, v_1}\to \overline Y_{0, 1}^s,\;\; \psi_2\colon C_{0, v_2}\to \overline Y_{0, 2}^s,
\]
 satisfying 
\[
\phi_{0, v_1v_2}|_{C_{0, v_1}} 
 = P_{\Phi}\circ \psi_1,\;\; \phi_{0, v_1v_2}|_{C_{0, v_2}} = P_{\Phi}\circ \psi_2, \;\; \psi_1(p) = \psi_2(p) = q.
\]
Note that the image of $\psi_2$ is the closure of
 an orbit of the one-dimensional subtorus corresponding to the 
 weight vector $(1, 0)$ of the bounded edge of $\Gamma_s$.
 The map $\psi_2$ is an isomorphism onto it.
The restriction of $P_{\Phi}$ to the image of $\psi_2$ is a branched covering of covering degree 
 $W_{e_{v_1v_2}}$, branched at $q$ and at the point at infinity.
Therefore, $\psi_2$ is uniquely determined by the condition $\psi_2(p) = q$, up to isomorphisms.
There are $W_{e_{v_1v_2}}$ such isomorphisms given by the action of the subgroup of the one-dimensional subtorus
 consisting of the $W_{e_{v_1v_2}}$-th roots of unity.

When $v_1$ is 2-valent, the map $\psi_1$ is also determined by the condition $\psi_1(p) = q$ up to 
 $W_{e_{v_1v_2}}$ isomorphisms.
Therefore, there are $W_{e_{v_1v_2}}^2$ maps $\psi_{0, V_1V_2}$ satisfying the required conditions.

On the other hand, when $v_1$ is 3-valent, the restriction of $P_{\Phi}$ to $\overline Y_{0, 1}^s$ is a branched covering with 
 covering degree $\det \Phi$.
Therefore, there are $\det \Phi$ curves in $\overline Y_{0, 1}^s$ which projects to 
 $\phi_{0, v_1v_2}(D_{0, 1})$.
Since the map $P_{\Phi}$ ramifies $W_{e_{v_1v_2}}$ times at $q$, the set $P_{\Phi}^{-1}(\phi_{0, v_1v_2}(p))$ contains 
\[
\frac{\det \Phi}{W_{e_{v_1v_2}}}
\]
 points.
By symmetry, out of the $\det \Phi$ curves mentioned above, only $W_{e_{v_1v_2}}$ of them pass through $q$.
Taking into account the $W_{e_{v_1v_2}}$ choices of $\psi_2$ described above, we obtain the desired result.\qed

\begin{rem}\label{rem:lift_log}
\begin{enumerate}
\item Recall that at each node of $C_0$, there are $W_e$ choices of log structures, where $W_e$
 is the weight of the corresponding edge $e$ of $\Gamma$, see \cref{eq:logcurve}.
On the other hand, the set of $W_e^2$ curves described in the above lemma (here, $e = e_{v_1v_2}$)
 is preserved under the action of a symmetry subgroup of the ambient toric variety
 $\overline{\mathcal Y}_{V_1V_2}^s$.
Namely, the multiplication by the $W_e$-th roots of unity in the one-dimensional subtorus 
 corresponding to the weight vector $(1, 0)$ of the bounded edge $E$ of $\Gamma_s$
 acts on $\overline{\mathcal Y}_{V_1V_2}^s$, and also on
 the set of $W_e^2$ curves.
Note that this action is different from the action on the set of the maps $\psi_2$ described in the proof.
The latter action fixes $\psi_1$, and does not extend to an action on $\overline{\mathcal Y}_{V_1V_2}^s$.
\item The quotient space under this action consists of $W_e$ classes of curves.
By direct calculation, 
 these correspond to the choices of the log structures mentioned above.
Explicitly, let $\mathscr D$ be the open subset of $D_0$ defined in Definition \ref{def:scrD} below.
Then, deformations of a representative $\psi_{0, V_1V_2}|_{\mathscr D}$ of an element of the quotient set, 
 when composed with $P_{\Phi}$, 
 yield
 log deformations of $\phi_{0, v_1v_2}|_{\mathscr D}$ associated with one of the log structures in \cref{eq:logcurve}
 (precisely, the corresponding deformations of the underlying analytic curve).
Thus, the construction of deformations of $\varphi_0$ via the lifts such as $\psi_{0, V_1V_2}$ naturally incorporates 
 the data of log structures.
For this reason, the choice of log structures does not explicitly appear in our argument hereafter,
 but it is implicitly encoded in the choice of the lifts.
\end{enumerate}
\end{rem}

The following result is straightforward.
We write 
\[
m = W_{e_{v_1v_2}}
\]
 for notational simplicity.

\begin{lem}\label{lem:mringY}
The toric variety $\mathcal Y_{v_1v_2}$
 is isomorphic to 
\[
\Spec\Bbb C[x, y, z^{\pm 1}, t]/(xy - t^m).
\]
Here, the functions $x, y, z$, and $t$ are characters of the torus acting on 
 $\mathcal Y_{v_1v_2}$.
Similarly, ${\mathcal Y}_{V_1V_2}^s$ is isomorphic to 
\[
\Spec\Bbb C[X, Y, Z^{\pm 1}, t]/(XY - t).
\]
These functions can be taken so that the pullback by the map $P_{\Phi}$ induces the following correspondence:
\[
x\mapsto X^{m}Z^R,\;\; y\mapsto Y^{m}Z^{-R},\;\; z\mapsto Z^{R'},\;\; t\mapsto t.
\]
Here, $R$ and $R'$ are integers.
When both $v_1$ and $v_2$ are 2-valent, we can assume $R = 0 $ and $R' = 1$.\qed
\end{lem}

We take an affine neighborhood $\mathscr D$ of $p$ in $D_0 = C_{0, v_1}\cup C_{0, v_2}$ as follows.
\begin{defin}\label{def:scrD}
Let $\mathscr D$ be the open subvariety of $D_0$ defined by 
\[
\mathscr D = \psi_{0, V_1V_2}^{-1}(Y_{0, V_1V_2}^s).
\]
\end{defin}

When both $v_1$ and $v_2$ are 2-valent, 
 the variety $\mathscr D$ is isomorphic to 
\[
\Spec\Bbb C[s, u]/(su).
\]
While when $v_1$ is 3-valent, it is isomorphic to 
\[
\Spec\Bbb C[s, \frac{1}{s-\alpha}, u]/(su),
\]
 where $\alpha$ is a non-zero complex number.
Namely, 
 by suitably taking the coordinates $X, Y$ on $Y_{0, V_1V_2}^s$ and $s, u$ on $D_0$,
 a lift $\psi_{0, V_1V_2}|_{\mathscr D}$ of $\phi_{0, v_1v_2}|_{\mathscr D}$ 
 gives 
 \[
 X\mapsto s,\;\; Y\mapsto u,\;\; Z\mapsto a+bs,
 \]
 and $\phi_{0, v_1v_2}|_{\mathscr D}$ gives
 \[
 x\mapsto s^m(a+bs)^R,\;\; y\mapsto a^{-R}u^m,\;\; z\mapsto (a+bs)^{R'},
 \]
 where $a$ and $b$ are non-zero complex numbers satisfying
\[
-\frac{a}{b} = \alpha.
\]

Recall that the toric variety $\mathcal Y_{v_1v_2}$
 naturally embeds in $\mathcal X$.
Also, $\overline Y_{0, v_1v_2}$ embeds in $\mathcal X$.
For simplicity, we identify the embedded images with these varieties.
Then, the restriction of the given map $\varphi_0\colon C_0\to X_0$
 to the irreducible components of $C_0$ corresponding to the vertices $v_1$ and $v_2$ of $\Gamma$
 gives a pre-log curve $\phi_{0, v_1v_2}$ into $\overline Y_{0, v_1v_2}$ as described in Lemma \ref{lem:v_1v_2}.
Thus, we identify the open subset $\mathscr D$ of $D_0$ introduced above with an open subset of 
 $C_0$, and also identify $\phi_{0, v_1v_2}|_{\mathscr D}$ above with $\varphi_0|_{\mathscr D}$.
\begin{defin}\label{def:Dv1v2}
When we identify $D_0$ with $C_{0, v_1}\cup C_{0, v_2}$, we write the corresponding open subset 
 $\mathscr D$ as $\mathscr D_{v_1v_2}$.
For simplicity, we denote the restriction $\psi_{0, V_1V_2}|_{\mathscr D_{v_1v_2}}$ also by $\psi_{0, V_1V_2}$.
\end{defin}

Assume we have a deformation
\[
\varphi_n\colon C_n\to 
 X_n = \mathcal X%_{\overline{\mathcal C}_{v_1v_2}}
     \times_{\Spec\Bbb C[t]}\Spec\Bbb C[t]/(t^{n+1})
\]
 of $\varphi_0\colon C_0\to X_0$. 
Then, the restriction of the structure of the analytic space 
 $C_n$ to the topological space underlying  $\mathscr D_{v_1v_2}$ is 
 isomorphic to 
\[
\Spec\Bbb C[s, u, t]/(su-t, t^{n+1}),
\]
 when both $v_1$ and $v_2$ are 2-valent, and to
\[
\Spec\Bbb C[s, \frac{1}{s-\alpha}, u, t]/(su-t, t^{n+1}),
\]
 when $v_1$ is 3-valent.
 
When $v_1$ is 3-valent, 
 the pullback by the map $\varphi_n$
 induces following correspondences:
\[
\begin{array}{l}
x\mapsto s^m(a+bs)^R + tf(s, u, t),\\ 
y\mapsto u^m(a+bs)^{-R}+tg(s, u, t),\\ 
z\mapsto (a+bs)^{R'} + th(s, u, t),
\end{array}
\]
 where $f, g$, and $h$ are functions in $\Bbb C[s, \frac{1}{s-\alpha}, u, t]/(su-t, t^{n+1})$.
When both $v_1$ and $v_2$ are 2-valent, we can put $R = 0$, $R' = 1$, and $b = 0$ in the above expression.
We now state the following lifting property.
\begin{prop}\label{prop:liftmap}
The map $\varphi_n|_{\mathscr D_{v_1v_2}}$ lifts to a map
\[
\psi_{n, V_1V_2}|_{\mathscr D_{v_1v_2}}\colon C_n|_{\mathscr D_{v_1v_2}}\to Y_{n, V_1V_2}^s,
\] 
 so that 
 the equation 
\[
\varphi_n|_{\mathscr D_{v_1v_2}} = P_{\Phi}\circ \psi_{n, V_1V_2}|_{\mathscr D_{v_1v_2}}
\]
 holds.
Here, 
\[
Y_{n, V_1V_2}^s = \mathcal Y_{V_1V_2}^s\times_{\Spec\Bbb C[t]}\Spec\Bbb C[t]/(t^{n+1}),
\]
 and $P_{\Phi}\colon {\mathcal Y}_{V_1V_2}^s\to 
  \mathcal Y_{v_1v_2}$
  is the map introduced in Definition \ref{def:Pphi}.
\end{prop}
\proof
%We assume $n\geq m$.
%The case $n<m$ is easier, and we omit the details.
We assume that $v_1$ is 3-valent.
When both $v_1$ and $v_2$ are 2-valent, setting $R = 0$, $R' = 1$ and $b = 0$ 
 in the following argument proves the claim.
It suffices to prove that the correspondences
\[
\begin{array}{l}
x\mapsto s^m(a+bs)^R + tf(s, u, t),\\
y\mapsto u^m(a+bs)^{-R}+tg(s, u, t),\\
z\mapsto (a+bs)^{R'} + th(s, u, t)
\end{array}
\]
 lift to suitable correspondences
\[
\begin{array}{l}
X\mapsto s + tF(s, u, t),\\
Y\mapsto u + tG(s, u, t),\\
Z\mapsto a+bs + tH(s, u, t),
\end{array}
\]
 so that the relations
\[
\begin{array}{l}
(s + tF(s, u, t))^m(a+bs + tH(s, u, t))^R = s^m(a+bs)^R + tf(s, u, t),\\
(u + tG(s, u, t))^m(a + bs + tH(s, u, t))^{-R} = u^m(a+bs)^{-R}+tg(s, u, t),\\
(a + bs + tH(s, u, t))^{R'} = (a+bs)^{R'}+th(s, u, t)
\end{array}
\]
 hold modulo $t^{n+1}$.
Here, $F, G$, and $H$ are analytic functions.
It also suffices to prove the existence of such a lift on an analytic neighborhood of the node $p$
 of $\mathscr D_{v_1v_2}$, 
 since away from its image, the map $P_{\Phi}|_{{\mathcal Y}_{V_1V_2}^s}$ is unramified.
Namely, let $\mathscr D_1$ be one of the connected components of $\mathscr D_{v_1v_2}\setminus\{p\}$.
Then, there exists an analytic neighborhood $\mathscr B$ of the image $\psi_{0, V_1V_2}(\mathscr D_1)$
 in $\mathcal Y_{V_1V_2}^s$ such that the restriction of $P_{\Phi}$ to $\mathscr B$ is an immersion.
Thus, the deformation of $\varphi_n|_{\mathscr D_1}$ induces a deformation of $\psi_{n, V_1V_2}|_{\mathscr D_1}$.
In particular, if a lift of $\varphi_n$ exists on some open subset of $\mathscr D_1$, then it extends to the whole 
 $\mathscr D_1$ .

Now, we can take an $R'$-th root of $(a+bs)^{R'}+th(s, u, t)$, 
 which is $a+bs$ modulo $t$. 
This root determines $H$.
By dividing $s^m(a+bs)^R + tf(s, u, t)$ by $(a+bs + tH(s, u, t))^R$, 
 we obtain a function of the form 
\[
s^m + t\widetilde f(s, u, t),
\]
 where $\widetilde f$ is an analytic function.
Similarly, by dividing $u^m(a+bs)^{-R}+tg(s, u, t)$ by $(a+bs + tH(s, u, t))^{-R}$, 
 we obtain 
\[
u^m + t\widetilde g(s, u, t)
\]
 where $\widetilde g$ is another analytic function.
By expanding $\frac{1}{s-\alpha}$ around $s = 0$, we regard $\widetilde f$ and $\widetilde g$
 as polynomials in $t$, whose coefficients are series in $s$ and $u$.
 
To prove the proposition, it suffices to show that $s^m+t\widetilde f(s, u, t)$ can be written in the form
\begin{equation}\label{eq:F}
s^m+t\widetilde f(s, u, t) = (s + tF(s, u, t)^m,\;\;\;\; \text{mod $t^{n+1}$},
\end{equation}
 and similarly, $u^m + t\widetilde g(s, u, t)$ can be written in the form
\begin{equation}\label{eq:G}
u^m + t\widetilde g(s, u, t) = (u+tG(s, u, t))^m, \;\;\;\; \text{mod $t^{n+1}$},
\end{equation}
 for some $F$ and $G$.

Since $xy = t^m$, we have 
\[
(s^m + t\widetilde f(s, u, t))(u^m+t\widetilde g(s, u, t)) = t^m, \;\;\;\; \text{mod $t^{n+1}$}.
\]
Expanding and rearranging terms, it follows that
\begin{equation}\label{eq:highestterm}
t(s^m\widetilde g(s, u, t) + u^m\widetilde f(s, u, t)) +t^2 \widetilde f(s, u, t)\widetilde g(s, u, t)= 0\;\;\;\; \text{mod $t^{n+1}$},
\end{equation}
From this equation, we observe the following.
\begin{lem}
Every monomial of $\widetilde f$ contains $s^m$ as a factor, 
 after rewriting $t = su$ if necessary.
Similarly, every monomial of $\widetilde g$ contains $u^m$ as a factor.
\end{lem}
\proof
Recall that $\widetilde f$ and $\widetilde g$ are polynomials in $t$ whose coefficients are series in $s$ and $u$.
Using the relation $su = t$, we can eliminate any term containing $su$ as a factor in $\widetilde f$ and $\widetilde g$.
Then, let us write
\[
\widetilde f(s, u, t) = sf_0(s) + \overline f_0(u) + t(sf_1(s) + \overline f_1(u)) + \cdots,
\]
 and
\[
\widetilde g(s, u, t) = g_0(s) + u\overline g_0(u) + t(g_1(s) + u\overline g_1(u)) + \cdots,
\]
 where 
\begin{itemize}
\item $sf_0(s) + \overline f_0(u)$ is the constant term (with respect to $t$) of $\widetilde f(s, u, t)$.
\item $\overline f_0(0) = \widetilde f(0,0,0)$, ensuring $sf_0(s)$ is divisible by $s$.
\item 
For each $i$ ($0\leq i\leq m-1$), the constant term in the coefficient of $t^i$ in $\widetilde f(s, u, t)$
 is collected in $\overline f_i(u)$.
\end{itemize}
Similarly, for $\widetilde g(s, u, t)$:
\begin{itemize}
\item The constants are collected in $g_i(s)$.
\end{itemize}

Substitute these expressions to \cref{eq:highestterm}. 
By eliminating $su$ using the relation $su=t$, 
 the coefficient of $t$ in \cref{eq:highestterm} becomes
\[
s^mg_0(s) + u^m\overline f_0(u).
\] 
It follows that $g_0(s) = \overline f_0(u) = 0$.

Next, consider the coefficient of $t^2$ in \cref{eq:highestterm}.
It is given by 
\[
s^{m-1}\overline g_0(0) + s^mg_1(s) + u^{m-1}f_0(0) + u^m\overline f_1(u).
\]
Since $s^mg_1(s)$ only contains terms divisible by $s^m$, 
 it follows that $g_1(s) = 0$.
Similarly, we have $\overline f_1(u) = 0$.
Moreover, we have $\overline g_0(0) = f_0(0) = 0$.
Thus, $f_0(s)$ is divisible by $s$, and $\overline g_0(u)$ is divisible by $u$.
Therefore, we can now write 
\[
\widetilde f(s, u, t) = s^2f_0(s) + tsf_1(s) + t^2(sf_2(s)+\overline f_2(u)) + \cdots,
\]
 and
\[
\widetilde g(s, u, t) = u^2\overline g_0(u) + tu\overline g_1(u) + t^2(g_2(s) + u\overline g_2(u)) + \cdots.
\]
Here, we use the same symbols $f_0(s), \overline g_0(u)$ as before to save letters.

Considering the coefficients of $t^3$ in \cref{eq:highestterm} and repeating the preceding argument, we have
\[
\overline f_2(u) = g_2(s) = 0,
\]
 and also 
\[
f_0(0) = f_1(0) = \overline g_0(0) = \overline g_1(0) = 0.
\]
By iterating this argument up to the coefficient of $t^{m-1}$, we can write 
\[
\widetilde f(s, u, t) = \sum_{i=0}^{m-1} t^is^{m-i}f_i(s) + \sum_{i=m}^nt^if_i(s, u, t),
\]
 and 
\[
\widetilde g(s, u, t) = \sum_{i=0}^{m-1} t^iu^{m-i}\overline g_i(u) + \sum_{i=m}^nt^ig_i(s, u, t).
\]
It follows that if we replace $t$ by $su$, all the terms of $\widetilde f(s, u, t)$ can be divided by $s^m$, 
 and all the terms of $\widetilde g(s, u, t)$ can be divided by $u^m$.\qed\\

Thus, we can write $\widetilde f(s, u, t) = s^m \widetilde F(s, u, t)$ for some $\widetilde F$.
Then, we have
\[
s^m + t\widetilde f(s, u, t) = s^m(1+t\widetilde F(s, u, t)) = s^m(1+t\overline F(s, u, t))^m
\]
 for some $\overline F$.
Similarly, we can write 
\[
u^m+t\widetilde g(s, u, t) = u^m(1+t\overline G(s, u, t))^m.
\]
By taking $F = s\overline F$ and $G = u\overline G$ in Eqs.(\ref{eq:F}) and (\ref{eq:G}),
 the proposition is proved.\qed

\subsection{Lifting defining equations of curves to $\mathcal Y_s$}\label{subsec:lift}
In \cite{Nishinou_Obstruction_2018}, 
 the obstructions to deforming maps between varieties are calculated in the case where the target space is 
 a fixed smooth complex manifold, the domain has a dimension one less than the target,
 and the map is an immersion.
The calculation of the obstruction in this paper is based on that in  \cite{Nishinou_Obstruction_2018},
 but special attention is required due to the following factors:
\begin{enumerate}
\item $X_0$ is singular and the complex structure deforms in the family $\mathcal X$, and
\item the map $\varphi_0$ is not necessarily an immersion. 
\end{enumerate}
To address the second issue, we reduce the calculation to the standard case introduced in the previous subsection.
This subsection primarily deals with the potential complications arising from the deformation of $X_0$.
However, the second issue also plays a role in the argument.

\begin{rem}
For clarity, we distinguish the terms used in relation to the map $\varphi_n$.
The term \emph{lift} refers to maps $\psi_{n, V_1V_2}$ introduced in the previous subsection.
In contrast, (local) \emph{deformation} refers to maps defined over $\Bbb C[t]/(t^{n+2})$
 which reduce to $\varphi_n$ over $\Bbb C[t]/(t^{n+1})$.
\end{rem}

Assume we have constructed a deformation $\varphi_n\colon C_n\to X_n$ of $\varphi_0\colon C_0\to X_0$.
To construct an obstruction cocycle suitable for calculation, we utilize defining
 equations of the images of open subsets of $C_n$ (see Subsection \ref{subsubsec:stdcocycle} for details).
Here, the image of $\varphi_n$ is given a structure of an analytic space called the
 annihilator structure, see Subsection \ref{subsec:notation}.
For an open subset $\mathscr U_i$ of $C_n$ introduced 
 in Definition \ref{def:scrU}, the image by $\varphi_n$ is given by an equation of the form
\[
\sum_{i=0}^n t^i f_i(\bold x) = 0,
\]
 where each $f_i$ is an analytic function, and
 $\bold x$ is a fixed coordinate system on some open subset $\mathscr W_i$ of $X_n$ (here, we temporarily
 disregard the singularities of $X_n$
 for the ease of explanation).
To construct an obstruction cocycle, we need to deform the map $\varphi_n$ on such open subsets.
In a favorable situation (e.g., when $\varphi_n$ is locally an embedding from a non-singular domain),
 such a deformation can be determined, 
 up to isomorphisms, by a defining equation of the image of the deformed map.
Therefore, any equation of the form
\[
\sum_{i=0}^{n+1} t^i f_i(\bold x) = 0,
\]
 gives a local deformation of $\varphi_n$.
A natural and convenient choice for $f_{n+1}$ is $f_{n+1} = 0$.
On another open subset $\mathscr U_j$ of $C_n$, the image by $\varphi_n$ is represented by an equation on 
 a different open subset $\mathscr W_j$ of $X_n$. 
The differences between these local deformations of $\varphi_n$ then
 form a \v{C}ech 1-cocycle representing the obstruction class.
Usually, for each open subset $\mathscr U_j$ of 
 $C_n$, it suffices to construct a local deformation of $\varphi_n$ on \emph{some} open subset $\mathscr W_j$ of $X_n$.

In the present case, however, the second issue mentioned above introduces additional complications.
To address this, we utilize two open coverings, $\{\mathscr U_i\}$ and 
 $\{\mathscr D_{v_1v_2}\}$ of $C_0$.
The covering $\{\mathscr D_{v_1v_2}\}$ 
 is used to locally lift $\varphi_n$ to maps of the form $\psi_{n, V_1V_2}$ introduced in the previous subsection.
This is necessary to apply the arguments in \cite{Nishinou_Obstruction_2018}.
On the other hand, the covering $\{\mathscr U_i\}$ is better suited for calculations of cohomology classes.
For example, since we can assume that the map $\varphi_n$ is an embedding on $\mathscr U_i$ away from the
 nodes (and, as for the map $\psi_{n, V_1V_2}$, it is also an embedding even on a neighborhood of nodes),
 it is reasonable to expect that we can construct an obstruction cocycle using this covering, by the method using defining equations
 explained above.

However, contrary to the case explained above, we need to consider deformations of $\varphi_n|_{\mathscr U_i}$ 
 on \emph{every} open subset of the form $\mathcal Y_{v_1v_2}\subset \mathcal X$ containing the image
 $\varphi_n(\mathscr U_i)$.
This requirement arises for the following reason.
The lifts $\psi_{n, V_1V_2}$ of $\varphi_n$, which we need to take to address the second issue, depend on the choice of these
 $\mathcal Y_{v_1v_2}$.
To obtain a well-defined obstruction cocycle, we must ensure that the deformations 
 defined for each lift are compatible in an appropriate sense.
 
To understand the issue in this argument,
 suppose the image $\varphi_n(\mathscr U_i)$ is represented by an equation $F_1 = 0$ using the coordinates
 on $\mathcal Y_{v_1v_2}\times_{\Spec \Bbb C[t]}\Spec \Bbb C[t]/(t^{n+1})$.
Assume there is another $\mathcal Y_{v_3v_4}$ containing $\varphi_n(\mathscr U_i)$, and $F_1$ is transformed 
 to $G_1$ by the coordinate transformation between 
 $\mathcal Y_{v_1v_2}\times_{\Spec \Bbb C[t]}\Spec \Bbb C[t]/(t^{n+1})$ and
 $\mathcal Y_{v_3v_4}\times_{\Spec \Bbb C[t]}\Spec \Bbb C[t]/(t^{n+1})$.
In general, when the function $F_1$ is regarded as a function on 
 $\mathcal Y_{v_1v_2}\times_{\Spec \Bbb C[t]}\Spec \Bbb C[t]/(t^{n+2})$
 by some prescribed rule, such as taking $f_{n+1} = 0$ as discussed above, 
 it does not necessarily coincide with $G_1$ regarded as a function on 
 $\mathcal Y_{v_3v_4}\times_{\Spec \Bbb C[t]}\Spec \Bbb C[t]/(t^{n+2})$
 by the same rule.
We will see that under the assumption (\ref{eq:assump}), this can be avoided.

%\begin{rem}
%Note that such compatible deformations is seldom possible in general, when the target space deforms.
%Namely, suppose a family of varieties $\mathcal Z\to \Bbb C$
% and an embedding $Y_n\to \mathcal Z$ is given, where $Y_n$
% is defined over $\Bbb C[t]/(t^{n+1})$.
%Let $W_1, W_2$ and $W_3$ are coordinate neighborhoods of $\mathcal Z$, and an open subset $U_n$ of $Y_n$ is
% defined by the equations $f_{W_1} = 0, f_{W_2} = 0$ and $f_{W_3} = 0$, using coordinates on these open subsets.
%Here, the functions $f_{W_i}$ are defined over $\Bbb C[t]/(t^{n+1})$.
%Assume that $f_{W_i}$ is transformed to $f_{W_{i'}}$ under the coordinate transformation, for $i\neq i'$.
%Then, in general, it is not possible to extend the functions $f_{W_i}$ to functions defined over $\Bbb C[t]/(t^{n+2})$
% so that they have the property that $f_{W_i}$ is transformed to $f_{W_{i'}}$ under the coordinate transformation
% over $\Bbb C[t]/(t^{n+2})$.
% 
%In the present situation, this is possible even though the variety $X_0$ deforms, due to the monomial nature of the
% coordinate transformations of toric varieties, and the assumption Eq.(\ref{eq:assump}).
%\end{rem}
 
The compatibility of deformations is proved 
 in this subsection (see Corollary \ref{cor:welldef}),
using the assumption (\ref{eq:assump}).
This statement concerns $\varphi_n$ rather than the lifts $\psi_{n, V_1V_2}$.
However, a similar compatibility for these lifts $\psi_{n, V_1V_2}$ follows easily from the statement for $\varphi_n$.
Using this, we construct an obstruction cocycle to deforming $\varphi_n$
 from the data of the lifts $\psi_{n, V_1V_2}$ in the next subsection.

\subsubsection{Lifting functions over $\Bbb C[t]/(t^{n+1})$ to $\Bbb C[t]/(t^{n+2})$}\label{subsec:split}
Let 
$
\{\mathscr U_i\}_{i\in I}
$
 be an open covering of $C_0$ as described in Subsection \ref{subsec:pairing}.
Since the map $\varphi_0$ is an immersion away from the nodes (see the proof of 
 Proposition \ref{prop:geninj}), 
 by taking these open subsets sufficiently small, we can assume that $\varphi_0|_{\mathscr U_i}$ is an embedding if 
 $\mathscr U_i$
 does not contain a node.
Additionally, we can assume 
 that if both $\mathscr U_i$ and $\mathscr U_j$ contain nodes, then 
 $\mathscr U_i\cap \mathscr U_j = \emptyset$ unless $\mathscr U_i = \mathscr U_j$.
Then, 
 the restriction
\[
\varphi_0|_{\mathscr U_i\cap \mathscr U_j}\colon \mathscr U_i\cap \mathscr U_j\to X_0
\]
 is an embedding for any $\mathscr U_i$ and $\mathscr U_j$ such that $\mathscr U_i\neq \mathscr U_j$. 
We also define an open covering 
$
\{\mathscr W_k\}_{k\in I'}
$
 of $X_0$ as follows.
\begin{defin}\label{def:scrW}
Let $\{\mathscr W_k\}_{k\in I'}$ be an open covering of $X_0$ with the following properties:
\begin{itemize}
\item The index set $I'$ contains the index set $I$ for the covering $\{\mathscr U_i\}$ of $C_0$.
\item For each $\mathscr U_i$, its image $\varphi_0(\mathscr U_i)$ is 
 contained in $\mathscr W_i$.
By shrinking $\mathscr W_i$ if necessary, we assume $\varphi_0(\mathscr U_i)$ is the zero locus of 
 an analytic function on $\mathscr W_i$.
Moreover, when $\mathscr U_i$ does not contain a node, then $\mathscr W_i$ does not intersect 
 the toric boundary of any irreducible component of $X_0$.
\item We take each $\mathscr W_i$ to be isomorphic to the product of two discs when $\mathscr U_i$ does not
 contain a node, and to be isomorphic a normal crossing union of products of two discs,
 which embeds into a suitable neighborhood of the origin in 
\[
\{xy = 0\}\subset \Bbb C^3 = \{(x, y, z)\;|\; x, y, z\in\Bbb C\}
\]
 when $\mathscr U_i$ contains a node.
\end{itemize}
\end{defin}

Assume that an $n$-th order deformation $\varphi_n\colon C_n\to X_n$ of the pre-log curve 
$\varphi_0\colon C_0\to X_0$ exists, where $n\geq 0$. 
\begin{defin}\label{def:u_inw_in}
Let $\mathscr U_{i, n}$ denote the ringed space obtained by restricting $C_n$ to $\mathscr U_i$.
Similarly, 
 let $\mathscr W_{i,n}$ denote the ringed space obtained by restricting $X_n$ to $\mathscr W_i$. 
\end{defin}
Let 
\[
F_{i, n}=0
\]
 be a defining equation of the image $\varphi_{n}(\mathscr U_{i, n})$ in $\mathscr W_{i,n}$.
When $\mathscr U_i$ contains a node of $C_0$, by shrinking the open subsets if necessary, 
 we can assume that the ring of functions on $\mathscr W_{i,n}$ contains the ring
\[
\Bbb C[x, y, z^{\pm 1}, t]/(xy-t^m, t^{n+1})
\]
 using the notation in the previous subsection.
This is the reduction over $\Bbb C[t]/(t^{n+1})$ of the 
 ring of functions on a toric open subset containing $\mathscr W_{i,n}$.
The defining function $F_{i, n}$ can be chosen from this subring. 

For the calculation of the obstruction below, it is necessary to regard the function $F_{i, n}$, which is defined
 over $\Bbb C[t]/(t^{n+1})$, as a function defined over $\Bbb C[t]/(t^{n+2})$.
This is achieved via a splitting of the exact sequence
\begin{equation}\label{eq:splitexact}
\begin{split}
0\to t^{n+1}\Bbb C[x, y, z^{\pm 1}]/(xy)\to &\Bbb C[x, y, z^{\pm 1}, t]/(xy-t^m, t^{n+2}) \\ 
 & \to  \Bbb C[x, y, z^{\pm 1}, t]/(xy-t^m, t^{n+1})\to 0
\end{split}
\end{equation}
 of rings.
There is an obvious splitting (as $\Bbb C$-modules)
\[
\Bbb C[x, y, z^{\pm 1}, t]/(xy-t^m, t^{n+1})\to  \Bbb C[x, y, z^{\pm 1}, t]/(xy-t^m, t^{n+2}).
\]
Namely, given an element 
\[
F(x, y, z, t) \in \Bbb C[x, y, z^{\pm 1}, t]/(xy-t^m, t^{n+1}),
\] 
 represent it in a way which does not contain $xy$ or $t^{n+1}$, 
 and regard the result as an element of $\Bbb C[x, y, z^{\pm 1}, t]/(xy-t^m, t^{n+2})$.

\begin{rem}
Although there are other splittings of the sequence \cref{eq:splitexact},
 the one described above appears to be the most natural one.
We may obtain the same result by other choices, but it will complicate the argument.
\end{rem}

When $\mathscr U_i$ does not contain a node of $C_0$, we can assume that
 the ring of functions on $\mathscr W_{i,n}$ contains
 either 
\[
\Bbb C[x^{\pm 1}, z^{\pm 1}, t]/(t^{n+1})\;\;
 \text{or} \;\;
 \Bbb C[y^{\pm 1}, z^{\pm 1}, t]/(t^{n+1}).
\]
Assume that it contains the former.
Consider the exact sequence
\[
0\to \Bbb C[x^{\pm 1}, z^{\pm 1}]\to 
 \Bbb C[x^{\pm 1}, z^{\pm 1}, t]/(t^{n+2})\to \Bbb C[x^{\pm 1}, z^{\pm 1}, t]/(t^{n+1})\to 0.
\]
Again, there is an obvious splitting 
\[
\Bbb C[x^{\pm 1}, z^{\pm 1}, t]/(t^{n+1})\to \Bbb C[x^{\pm 1}, z^{\pm 1}, t]/(t^{n+2})
\]
 given by writing an element of $\Bbb C[x^{\pm 1}, z^{\pm 1}, t]/(t^{n+1})$ so that it does not
 contain $t^{n+1}$, and regarding it as an element of $\Bbb C[x^{\pm 1}, z^{\pm 1}, t]/(t^{n+2})$.

In this case, the defining function of $\varphi_n(\mathscr U_{i, n})$ in $\mathscr W_{i, n}$ is not contained in 
 $\Bbb C[x^{\pm 1}, z^{\pm 1}, t]/(t^{n+1})$ in general, but is contained in a larger ring.
This is because, at the image $\varphi_0(a)$ of some point $a\in \mathscr U_i$, 
 there may be several branches of $\varphi_0(C_0)$.
If $\mathscr U_i$ is sufficiently small, $\varphi_0(\mathscr U_i)$ is contained in one of these branches.
Therefore, the defining functions of $\varphi_0(\mathscr U_i)$ and of $\varphi_n(\mathscr U_{i, n})$ are not, 
 in general, polynomials in $x, z$ and $t$, 
 but rather polynomials in $t$ whose coefficients are analytic functions of $x$ and $z$.
Nevertheless, the natural extension of the splitting described above provides a unique lift of the defining functions to 
 $\Bbb C[t]/(t^{n+2})$.

\subsubsection{The behavior of lifts of functions under coordinate changes}
We note that there is some subtlety here, and this is the place where the assumption (\ref{eq:assump}) 
 in the introduction becomes relevant.
Namely, there are several open subsets $\mathcal Y_{v_1v_2} \hookrightarrow \mathcal X$
 introduced in Definition \ref{def:mathcalY},
 whose images contain $\varphi_0(\mathscr U_i)$.
The toric coordinates $x$, $y$, and $z^{\pm 1}$, which are regular on $\mathcal Y_{v_1v_2}$,
 depend on the choice of the vertices $v_1$ and $v_2$ of $\Gamma$.
Suppose that the open subset $\mathcal Y_{v_2v_3}$ also contains $\varphi_0(\mathscr U_i)$.
Since the above splitting process involves choosing an explicit representative for an element of a quotient ring, 
 this process can, in general, depend on the choice of these coordinates.
However, this does not occur in the present situation, see Corollary \ref{cor:welldef}.
The following simple result is key to showing this.
\begin{prop}\label{prop:compatibility}
Under the assumption (\ref{eq:assump}), the coordinate transformation
 between $\mathcal Y_{v_1v_2}$ and $\mathcal Y_{v_2v_3}$
 does not depend on the parameter $t$.
\end{prop} 
\proof
Let $\Delta$ be the parallelogram fundamental domain of the 
 $\overline{\Lambda}$-action on $N_{\Bbb R}\cong \Bbb R^2$ as in the
 introduction.
Assume all the vertices $v_1, v_2$ and $v_3$ and the edge connecting them are contained in $\Delta$.
As we noted in Lemma \ref{lem:mringY}, the toric variety $\mathcal Y_{v_1v_2}$
 is isomorphic to $\Spec\Bbb C[x, y, z^{\pm 1}, t]/(xy-t^m)$, and
 $\mathcal Y_{v_2v_3}$
 is isomorphic to $\Spec\Bbb C[x', y', z'^{\pm 1}, t]/(x'y'-t^m)$ 
 for some characters $x, y, z, x', y'$, and $z'$ of the torus acting on these varieties.
If $y, z$ and $y', z'$ restrict to coordinates on an open subset of $X_{0, v_2}$, 
 then it is straightforward to see that $y$ and $z$ are Laurent monomials of $y'$ and $z'$, and vice versa. 
Thus, the coordinate transformation between $\mathcal Y_{v_1v_2}$ and $\mathcal Y_{v_2v_3}$
 does not depend on the parameter $t$.

On the other hand, we also need to verify that the coordinate transformation induced by the action of $\Lambda$
 does not depend on $t$, either.
Let $\lambda_1 = (\alpha_{11}t^{n_{11}}, \alpha_{12}t^{n_{12}})$ be one of the generators of 
 $\Lambda$ as in (\ref{eq:assump}).
On the tropical curve, the corresponding action is the translation by the vector $\overline{\lambda}_1 = (n_{11}, n_{12})$.
Let $\widetilde v_i = (a_i, b_i)\in N$ be a vertex of the $\overline{\Lambda}$-periodic tropical curve
 (see Section \ref{sec:intro_(realization)})
\[
\widetilde h\colon\widetilde{\Gamma}\to N_{\Bbb R},
\]
 which projects to $v_i$.
We assume $\{\widetilde v_1, \widetilde v_2\}$ and $\{\widetilde v_2, \widetilde v_3\}$ are pairs of adjacent vertices.

The tropical curve $(\widetilde{\Gamma}, \widetilde h)$ induces a degeneration of 
 the family of tori $(\Bbb C^{\times})^2\times B^{\times}$, as described in Section \ref{sec:intro_(realization)}, to 
 a union of toric varieties $\widetilde X_0$.
Let 
\[
\widetilde{\mathcal X}\to B
\]
 be the extended family.
Then, the action of $\Lambda$
 extends to $\widetilde{\mathcal X}$ and, the quotient gives the degeneration $\mathcal X\to B$.

As in the case of $\mathcal X$, there are open subvarieties 
\[
\widetilde{\mathcal Y}_{\widetilde v_1\widetilde v_2},\;\;
 \widetilde{\mathcal Y}_{\widetilde v_2\widetilde v_3}
\]
 of $\widetilde{\mathcal X}$.
Let 
$
\Bbb C[x, y, z^{\pm 1}, t]/(xy-t^m)$ and $\Bbb C[x', y', z'^{\pm 1}, t]/(x'y'-t^{m'})
$
 be the rings of functions on $\widetilde{\mathcal Y}_{\widetilde v_1\widetilde v_2}$ and $\widetilde{\mathcal Y}_{\widetilde v_2\widetilde v_3}$,
 respectively, as above.
Similarly, for the pair of adjacent vertices $\{\widetilde v_1+\overline{\lambda}_1, \widetilde v_2+\overline{\lambda}_1\}$, 
 we have an open subvariety 
\[
\widetilde{\mathcal Y}_{(\widetilde v_1+\overline{\lambda}_1)(\widetilde v_2+\overline{\lambda}_1)}
\] 
 of $\widetilde{\mathcal X}$, 
 whose ring of functions can be written as 
 $\Bbb C[\overline x, \overline y, \overline z^{\pm 1}, t]/(\overline x\overline y-t^m)$, where
 $\overline x, \overline y$, and $\overline z$ are characters.
Under the assumption (\ref{eq:assump}),
 these functions $\overline x$, $\overline y$, and $\overline z$ can be chosen so that
 the action of $\lambda_1$ pulls them back to $x, y$, and $z$, respectively.
Therefore, the coordinate transformation between $\widetilde{\mathcal Y}_{\widetilde v_2\widetilde v_3}$ and 
 $\widetilde{\mathcal Y}_{(\widetilde v_1+\overline{\lambda}_1)(\widetilde v_2+\overline{\lambda}_1)}$ does not 
 depend on the parameter $t$, either.
This proves the claim.\qed\\

The key observation in the proof is that the generators of the ring of functions 
 $\Spec\Bbb C[\overline x, \overline y, \overline z^{\pm 1}, t]/(\overline x\overline y-t^m)$ of 
 $\widetilde{\mathcal Y}_{(\widetilde v_1+\overline{\lambda}_1)(\widetilde v_2+\overline{\lambda}_1)}$
 are Laurent monomials of the generators $x, y, z$, and $t$ of the ring of functions on 
 $\widetilde{\mathcal Y}_{\widetilde v_1\widetilde v_2}$.
This follows from the fact that both sets of generators are characters of the torus acting on 
 $\mathcal X$.
Specifically, $\overline x$ can be expressed as the product of $x$ and some (possibly negative) power of $t$,
 and similar expressions hold for $\overline y$ and $\overline z$.
Therefore, if the action of $\Lambda$ maps monomials to monomials, 
 the coordinate transformation does not depend on the parameter $t$.
However, if the action of $\Lambda$ does not satisfy the assumption (\ref{eq:assump}), 
 and a generator of $\Lambda$, such that $\lambda_1$, is of the form 
\[
\lambda_1 =
 (\alpha_{11}t^{n_{11}} + \alpha_{11, 1}t^{n_{11}+1} + \cdots, \alpha_{12}t^{n_{12}} + \alpha_{12,1}t^{n_{12}+1}+ \cdots ), 
\]
 then, 
 it takes $\overline y$ to a function of the form $(c_0+c_1t + c_2 t^2 + \cdots)y$.
In this case, the coordinate transformation between 
 $\widetilde{\mathcal Y}_{\widetilde v_2\widetilde v_3}$ and 
  $\widetilde{\mathcal Y}_{(\widetilde v_1+\overline{\lambda}_1)(\widetilde v_2+\overline{\lambda}_1)}$ depends on the parameter $t$.\\

From this proposition, it follows that the splitting of \cref{eq:splitexact} is well-defined, that is, 
 it is independent of the choice of 
 $\mathcal Y_{v_1v_2} \hookrightarrow \mathcal X$.
To elaborate, suppose we have a function 
\[
f(x, y, z, z^{-1}, t) = \sum_{i=0}^n f_i(x, y, z. z^{-1})t^i,
\]
 defined on an open subset $\mathscr W_{i,n}$ of $\mathcal Y_{v_1v_2}$,
 which is expressed using characters of the torus acting on $\mathcal X$.
Here, $f_i(x, y, z, z^{-1})$ are analytic functions in $x, y, z$, and $z^{-1}$.
We regard $f$ as a function on $\mathscr W_{i,n+1}$ by considering it as a function whose coefficient of
 $t^{n+1}$ is zero.
If $\varphi_0$ is an immersion, it provides a deformation of $\varphi_n$ on $\mathscr U_{i,n}$.
In general, $\varphi_0$ is not an immersion, but since the lifts $\psi_{n, V_1V_2}$ are immersions, 
 this issue can be circumvented (see Subsection \ref{subsec:obstcocycle} below). 
 
If $\mathscr W_{i,n}$ is contained in another open subset $\mathcal Y_{v_1v_3}$, whose ring of functions is
 $\Bbb C[x', y', (z')^{\pm 1}, t]/(x'y' - t^{m'})$, 
 we can express the function $f(x, y, z,z^{-1}, t)$ in terms of $x', y'$, and $(z')^{\pm 1}$ 
 using the coordinate transformation
 between $\mathcal Y_{v_1v_2}$ and $\mathcal Y_{v_1v_3}$.
Let $f(x', y', z', (z')^{-1}, t)$ be the function obtained in this way.
It is a function over $\Bbb C[t]/(t^{n+1})$.
Following the recipe in Subsection \ref{subsec:split}, we again regard it as a function on $\mathscr W_{i,n+1}$.
Then, by Proposition \ref{prop:compatibility}, we have the following.
\begin{cor}\label{cor:welldef}
The coordinate transformation between the varieties
 $\mathcal Y_{v_1v_2}$ and $\mathcal Y_{v_1v_3}$, considered over $\Bbb C[t]/(t^{n+2})$, 
 sends the function $f(x, y, z, z^{-1}, t)$
 into $f(x', y', z', (z')^{-1}, t)$.\qed
\end{cor}

Thus, the local deformation of $\varphi_n$ on $\mathscr U_{i, n}$ 
 using the defining equations is well-defined, and does not depend on
 the choice of $\mathcal Y_{v_1v_2}$.

\begin{rem}\label{rem:compati} 
This issue does not arise in the case where $\mathscr U_i$ contains a node, since, in this case there is a unique 
 $\mathcal Y_{v_1v_2} \hookrightarrow \mathcal X$ whose image contains $\varphi_0(\mathscr U_i)$.
\end{rem}

\subsection{A representative of the obstruction cocycle}\label{subsec:obstcocycle}
In this subsection, we first construct an obstruction cocycle to deforming $\varphi_n$ under the assumption that
 $\varphi_0$ is an immersion, 
 following a standard approach,
 see Lemma \ref{lem:obst}.
In general situations, however, $\varphi_0$ is not an immersion.
Therefore,
 we instead utilize its local lifts $\psi_{n, V_1V_2}$ introduced in Subsection \ref{subsec:std}.
Recall that the obstruction to deforming $\varphi_n$ belongs to the cohomology group $H^1(C_0, \mathcal N_{\varphi_0})$.
Therefore, we need to construct an $\mathcal N_{\varphi_0}$-valued cocycle from the data of $\psi_{n, V_1V_2}$.
For this purpose, we construct another invertible sheaf $\mathcal N$ associated with $\psi_{n, V_1V_2}$ in
 Subsections \ref{subsec:toriccover} and \ref{subsec:psinormal}.
The sheaf $\mathcal N$ turns out to be isomorphic to $\mathcal N_{\varphi_0}$.
Then, in Subsection \ref{subsec:constcocycle}, we construct an obstruction cocycle to deforming 
 a general $\varphi_n$
 based on the lifts $\psi_{n, V_1V_2}$.
Based on the calculations in Subsection \ref{subsec:lift}, 
 such a cocycle can be defined by tracing the standard arguments
 in Subsection \ref{subsubsec:stdcocycle}.
Specifically, the well-definedness of the cocycle is shown in Lemma \ref{lem:cocycle} using Corollary \ref{cor:welldef}.
The fact that this cocycle actually represents the obstruction class is shown in 
 Proposition \ref{prop:locallifts}.

\subsubsection{A construction of an obstruction cocycle to deforming $\varphi_n$
 when $\varphi_0$ is an immersion}\label{subsubsec:stdcocycle}
First, we note the following:
\[
\text{\it In Subsection \ref{subsubsec:stdcocycle}, the map $\varphi_0$ is assumed to be an immersion.}
\]
This condition does not hold for a general $\varphi_0$,
 which motivates the consideration of lifts $\psi_{n, V_1V_2}$.
Since the maps $\psi_{n, V_1V_2}$ are immersions, the calculations presented below can be applied to them, 
 in Subsection \ref{subsec:constcocycle}.
For further discussion, see Remark \ref{rem:nonimmersion}.

Let $F_{i, n}=0$ be a defining equation of the image $\varphi_{n}(\mathscr U_{i, n})$ in $\mathscr W_{i,n}$.
Similarly, let $F_{j, n} = 0$ be a defining equation of the image  $\varphi_{n}(\mathscr U_{j, n})$ in $\mathscr W_{j, n}$,
 and assume that $\mathscr U_i\cap \mathscr U_j\neq\emptyset$.
By shrinking these open subsets if necessary, we can further assume that 
 if $\mathscr U_i\cap \mathscr U_j$ is non-empty, we have
 \[
 \varphi_0(\mathscr U_i\cup \mathscr U_j) \cap (\mathscr W_i\cap \mathscr W_j) = \varphi_0(\mathscr U_i\cap \mathscr U_j).
 \]
Under this assumption, on the intersection $\mathscr W_{i,n}\cap \mathscr W_{j, n}$, the defining functions are related by
\[
F_{i, n} = G_{ij, n}F_{j, n},
\]
where $G_{ij, n}$ is a holomorphic function on 
 $\mathscr W_{i,n}\cap \mathscr W_{j, n}$ defined over $\Bbb C[t]/(t^{n+1})$,
whose reduction over $\Bbb C[t]/(t)$ is non-vanishing.
By regarding these functions as defined over $\Bbb C[t]/(t^{n+2})$ as in the previous subsection,
 we define a function $\nu_{ij, n+1}$ by the equation
\begin{equation}\label{eq:nu}
t^{n+1}\nu_{ij, n+1} = F_{i, n} - G_{ij, n}F_{j, n}
\end{equation}
on $\mathscr W_{i, n+1}\cap \mathscr W_{j, n+1}$.
Here, $\nu_{ij, n+1}$ can be regarded as a holomorphic function on $\mathscr W_{i, 0}\cap \mathscr W_{j, 0}$.

Assume that there is another open subset $\mathscr W_k$ containing $\varphi_0(\mathscr U_k)$, 
 such that 
\[
\mathscr U_i\cap \mathscr U_j\cap \mathscr U_k\neq \emptyset,
\]
and define $F_{k, n}$, $G_{jk, n}$, and $G_{ki, n} (= G_{ik, n}^{-1})$ as above.
These functions determine $\nu_{jk, n+1}$, $\nu_{ki, n+1}$ etc. on the relevant intersections.

\begin{lem}\label{lem:obst}
	On $\mathscr W_{i, 0}\cap \mathscr W_{j, 0}\cap \mathscr W_{k, 0} \cap \varphi_0(\mathscr U_i)$, 
	the identity
	\[
	\nu_{ik, n+1} = \nu_{ij, n+1} + G_{ij, 0}\nu_{jk, n+1}
	\]
	holds. 
\end{lem}
See \cite[Theorem 1]{Kodaira_Spencer_Semiregularity_1959}, \cite[Proposition 6]{Nishinou_RelativeSemiregularity_2020}
 and \cite[Lemma 11]{Nishinou_Obstruction_2018} for related results.
\proof 
%This result was established in \cite{Kodaira_Spencer_Semiregularity_1959} for the case where the target space is a
% fixed complex manifold, see also \cite{Nishinou_Obstruction_2018}.
%For relative deformations with smooth target spaces, this is proved in \cite{Nishinou_RelativeSemiregularity_2020}.
%In the present case, we must account for the singularities of $X_0$, and the variation of complex structures
% in the family $\mathcal X$.
%However, the singularities of $X_0$ does not pose a serious problem, since in this setting, 
% the deformations of maps and those obtained by deforming defining equations are essentially equivalent, see 
% Remark \ref{rem:nonimmersion}. 
%Similarly, the variation of complex structures
% in the family $\mathcal X$ does not cause significant difficulties, 
% as its deformation is constrained by the assumption 
% (\ref{eq:assump}).
%Thus, the proof closely follows the fixed-target case.
 
%When none of the open subsets $\mathscr U_i, \mathscr U_j$ and $\mathscr U_k$ contains a node of $C_0$, then 
% the proof aligns exactly with the fixed target case \cite[Theorem 1]{Kodaira_Spencer_Semiregularity_1959}.
%This is because,under these circumstances, we can take coordinates on 
% $\mathscr W_{i, n+1}, \mathscr W_{j, n+1}$ and $\mathscr W_{k, n+1}$ 
% such that the coordinate transformations between them do not 
% depend on the deformation parameter $t$, due to the assumption (\ref{eq:assump}), 
% see Proposition \ref{prop:compatibility}.
%
First, assume that
 none of the open subsets $\mathscr U_i, \mathscr U_j$ and $\mathscr U_k$ contains a node of $C_0$.
By definition of $\nu_{ij, n+1}$, we have
\[
\begin{split}
t^{n+1}\nu_{ik, n+1} & =
 F_{i, n} - G_{ik, n}F_{k, n} \\
 & = F_{i, n} - G_{ij, n}F_{j, n} + G_{ij, n}F_{j, n}  - G_{ik, n}F_{k, n}\\
 & = t^{n+1}\nu_{ij, n+1}  + G_{ij, n}(F_{j, n}  - G_{jk, n}F_{k, n})\\
  & \hspace{2in} + G_{ij, n}G_{jk, n}F_{k, n} - G_{ik, n}F_{k, n}\\
 & = t^{n+1}\nu_{ij, n+1}+ t^{n+1}G_{ij, n}\nu_{jk, n+1}  + (G_{ij, n}G_{jk, n} - G_{ik, n})F_{k, n},
\end{split}
\]
 which is an equation over $\Bbb C[t]/(t^{n+2})$.
Since 
\[
G_{ij, n}G_{jk, n} \equiv G_{ik, n} \;\;\text{mod $t^{n+1}$},
\]
 we have
\[
 (G_{ij, n}G_{jk, n} - G_{ik, n})F_{k, n}
  \equiv  (G_{ij, n}G_{jk, n} - G_{ik, n})F_{k, 0} \;\; \text{mod $t^{n+2}$}.
\]
Since $F_{k, 0} = 0$ on $\mathscr U_i\cap \mathscr U_j\cap \mathscr U_k$, this proves the claim.

Next, assume $\mathscr U_i$ contains a node of $C_0$.
In this case, by the way we take the open covering $\{\mathscr U_i\}$, 
 $\mathscr U_j$ and $\mathscr U_k$ do not contain a node when 
 $\mathscr U_i\cap \mathscr U_j\cap \mathscr U_k\neq \emptyset$.
As in the previous subsection, the ring of functions on $\mathscr W_{i,n}$ contains the ring 
 $\Bbb C[x, y, z^{\pm 1}, t]/(xy - t^m, t^{n+1})$, which is the ring of functions on a toric open subset containing $\mathscr W_{i,n}$.
%Here, $m$ is a positive integer.
%In fact, $m$ must be one when $\varphi_0$ is an immersion by the requirement from the log structures,
% though we do not need it in this proof.
Let $F_{i, n}=0$ be the defining equation of 
 the image $\varphi_n|_{\mathscr U_{i, n}}$ in $\mathscr W_{i,n}$ as above.
We can take $F_{i, n}$ from the ring $\Bbb C[x, y, z^{\pm 1}, t]/(xy - t^m, t^{n+1})$.
 
Assume that the images of $\mathscr U_j$ and $\mathscr U_k$ are contained in the component given by $y = 0$.
In the defining equation $t^{n+1}\nu_{ij, n+1} = F_{i, n} - G_{ij, n}F_{j, n}$ of $\nu_{ij, n+1}$, 
 we first regard $F_{i, n}$
 as an element of $\Bbb C[x, y, z^{\pm 1}, t]/(xy - t^m, t^{n+2})$ as in the previous subsection. 
Then, substitute 
 $y = \frac{t^m}{x}$ to $F_{i, n}$.
With this modification, the previous calculations remain unchanged. \qed\\

%Let $F_{j, n} = 0$ and $F_{k, n} = 0$ be the defining functions of the images of $\mathscr U_{j, n}$ and
% $\mathscr U_{k, n}$ in $\mathscr W_{j, n}$ and $W_{k, n}$, respectively.

%\begin{rem}
%Note that 
% $W_{i, 0, s}\cap W_{j, 0, s}\cap W_{k, 0, s} \cap \psi_0(\mathscr U_i)
%  = W_{i, 0, s}\cap W_{j, 0, s}\cap W_{k, 0, s} \cap \psi_0(\mathscr U_j)
% = W_{i, 0, s}\cap W_{j, 0, s}\cap W_{k, 0, s} cap \psi_0(\mathscr U_k)$
% holds.
%\end{rem}

Note that the pullback of the 
 set of functions $\{G_{ij, 0}\}$ to $C_0$ is the set of transition functions for the 
 normal sheaf of $\varphi_0$.
When $\varphi_0$ is an immersion, the normal sheaf of $\varphi_0$ and the log normal sheaf of $\varphi_0$
 are isomorphic.
Thus, the lemma implies that the pullback of $\{\nu_{ij, n+1}\}$ to 
 $C_0$ behaves as a \v{C}ech 1-cocycle 
 with values in the log normal sheaf of $\varphi_0$.
By construction, the \v{C}ech 1-cocycle $\{\nu_{ij, n+1}\}$ represents the obstruction
 to deforming $\varphi_n$ to a map over $\Bbb C[t]/(t^{n+2})$ (see \cite[Theorem 1]{Kodaira_Spencer_Semiregularity_1959}
 and Proposition \ref{prop:locallifts} below).

\begin{rem}\label{rem:nonimmersion}
The assumption that the map $\varphi_0\colon C_0\to X_0$ is an immersion is essential for the following reason.
Namely, suppose $p\in C_0$ is a point where $\varphi_0$ is singular, and let $\mathscr U$ be an open neighborhood
 of $p$.
Assume 
 a deformation $\varphi_n\colon C_n\to X_n$ has been constructed.
Let $F_n = 0$ be the defining equation of the image of $\mathscr U_n$ in a suitable open subset of $X_n$, 
 where $\mathscr U_n$ is the restriction of $C_n$ to $\mathscr U$. 
Then, if we regard $F_n = 0$ as an equation defined over $\Bbb C[t]/t^{n+2}$ as in the main text, 
 it provides a deformation of the subvariety $\varphi_n(\mathscr U_n)$, but it does not necessarily provide
 a deformation of the map $\varphi_n|_{\mathscr U_n}$.
Thus,  to apply this method for constructing obstruction cocycles, it is crucial that the maps are immersions.
This requirement motivates the introduction of the maps $\psi_{n, V_1V_2}$. 
\end{rem}

\subsubsection{Isomorphisms between lifts $\psi_{0, V_1V_2}$}\label{subsec:toriccover}
Since the map $\varphi_0$ is not generally an immersion, we cannot directly work with the cocycle $\{\nu_{ij, n+1}\}$
 constructed in Subsection \ref{subsubsec:stdcocycle}.
Instead, we construct another cocycle using the covering $\{\mathscr D_{v_1v_2}\}$
 introduced in Subsection \ref{subsec:std},
 to handle a general $\varphi_0$.
As shown in Proposition \ref{prop:liftmap}, the map $\varphi_n$ locally lifts to an embedding on this covering.
This property is crucial for applying the method of \cite{Nishinou_Obstruction_2018} 
 to the current setting, as discussed at the beginning of Subsection \ref{subsec:lift}. 
We will construct an obstruction cocycle to deforming $\varphi_n$ using such lifts.
In this subsection and the next, we develop the necessary foundations for this construction.
In this subsection, we study the relation between local lifts.
The main result, Lemma \ref{lem:auto}, shows that these lifts are transformed by the ambient isomorphisms,
 analogous to the case of usual covering transformations.
However, we need to account for the branching in the current setting.

Let $\Sigma$ and $\Sigma'$ be two fans in $\Bbb R^n = \Bbb Z^n\otimes_{\Bbb Z} \Bbb R$,
 and suppose $\Sigma'$ is the image of $\Sigma$ under a linear endomorphism
\[
L\colon \Bbb R^n\to \Bbb R^n
\]
 of maximal rank defined over $\Bbb Z$.
We identify $L$ with its matrix representation with respect to the standard basis of $\Bbb Z^n$.
Then, $L$ induces a branched covering map 
\[
P_L\colon X_{\Sigma}\to X_{\Sigma'}
\]
 between the toric varieties associated with the fans $\Sigma$ and $\Sigma'$.
\begin{defin}\label{def:covertransf}
We call a toric automorphism
\[
Q\colon X_{\Sigma}\to X_{\Sigma}
\]
 a covering transformation relative to $P_L$ if the equation
\[
P_L = P_L\circ Q
\]
 holds.
\end{defin}

Let $\{e_1, \dots, e_n\}$ be the standard basis of $\Bbb Z^n$.  
Let $\{f_1, \dots, f_n\}$ be its dual basis of $(\Bbb Z^n)^{\vee} = {\rm Hom}_{\Bbb Z}(\Bbb Z^n, \Bbb Z)$.
Then, $(\Bbb Z^n)^{\vee}$ is naturally identified with the set of characters of the torus that acts on 
 $X_{\Sigma}$ and $X_{\Sigma'}$.
These characters correspond to rational functions on these toric varieties.
Given an element $f$ of $(\Bbb Z^n)^{\vee}$,
 let
\[
z^f, w^f
\]
 be the corresponding rational functions on $X_{\Sigma}$ and $X_{\Sigma'}$, respectively.
The following should be well-known, but we provide a proof for the reader's convenience.

\begin{lem}
There exist exactly	 $|\det L|$ covering transformations relative to $P_L$.
\end{lem} 
\proof
As a first step, we establish the following claim.
\begin{claim}\label{claim:Q}
Suppose $Q$ is a covering transformation relative to $P_L$.
Then, the pullback of the functions $z^f$ on $X_{\Sigma}$ by $Q$ 
 takes the following form
\[
Q^*(z^{f_i}) = \zeta_i z^{f_i},
\]
 for each $i = 1, \dots, n$, where $\zeta_i$ is a root of unity. 
\end{claim}
\proof
Since $L$ has maximal rank, for each $i$, there is a monomial $w^{g_i}$ on $X_{\Sigma'}$
 for some vector $g_i\in (\Bbb Z^n)^{\vee}$
 such that the equation
\[
P_L^*(w^{g_i}) = z^{m_if_i}
\]
 holds.
Here, $m_i$ is an integer.
Since the map $Q$ satisfies $P_L = P_L\circ Q$, we have
\[
Q^*(z^{m_if_i}) = Q^*\circ P_L^*(w^{g_i}) = P_L^*(w^{g_i}) = z^{m_if_i}.
\]
The claim follows from this.\qed\\

For each $i$, we write
\[
P_L^*(w^{f_i}) = z^{h_i},
\]
 where $h_i\in (\Bbb Z^n)^{\vee}$ is given by
\[
h_i = L^T(f_i).
\]
Then, the set of covering transformations relative to $P_L$ is in bijection with the set of automorphisms $\alpha$
 of the ring of characters
\[
\Bbb C[z^{\pm f_1}, \dots, z^{\pm f_n}]
\]
 such that
\begin{itemize}
\item $\alpha$ is of the form described in Claim \ref{claim:Q}, that is, $\alpha(z^{f_i}) = \zeta_iz^{f_i}$
 for some root of unity $\zeta_i$, for each $i$, and
\item $\alpha(z^{h_i}) = z^{h_i}$ for each $i$.
\end{itemize}
In fact, the second condition suffices, but the first condition helps us analyze $\alpha$.

Write $\zeta_i$ as
\[
\zeta_i = e^{2\pi i q_i},
\]
 where $q_i\in [0, 1)$ is a rational number.
Then, the map $\alpha$ above satisfies the second condition if and only if 
\[
L\begin{pmatrix}
q_1 \\ 
\vdots\\
q_n
\end{pmatrix}
\]
 is an integer vector.
This is equivalent to the claim that the vector 
$
\begin{pmatrix}
q_1 \\ 
\vdots\\
q_n
\end{pmatrix}
$ 
 lies within 
\[
(L^{-1}\Bbb Z^n)\cap [0, 1)^n.
\]
Since we have
\[
\sharp ((L^{-1}\Bbb Z^n)/\Bbb Z^n) = |\det L|,
\]
 the lemma follows.\qed\\

Let $p\in X_{\Sigma'}$ be a general point.
Then, the inverse image $P_L^{-1}(p)$ consists of $|\det L|$ points.
\begin{lem}\label{lem:transitivecover}
The group of covering transformations relative to $P_L$ acts simply transitively on the set $P_L^{-1}(p)$.
\end{lem}
\proof
Let $q$ be a point in $P_L^{-1}(p)$.
Note that since $p$ is a general point, the value of the function $z^{f_i}$ at $q$ is non-zero for all $i$.
Let $Q$ be a covering transformation relative to $P_L$.
Then, if $q$ is fixed by $Q$, 
 the constants $\zeta_i$ in Claim \ref{claim:Q} must be one for all $i$.
Therefore, the map $Q$ is the identity map.
The lemma follows from this observation and the fact that the cardinality of the set $P_L^{-1}(p)$
 and that of the group of covering transformations relative to $P_L$ are the same.\qed\\

In particular, we have the following.
\begin{lem}\label{lem:mapcover}
Let $\phi\colon C\to X_{\Sigma'}$ be a torically transverse map from a connected irreducible curve to $X_{\Sigma'}$
 which is generically injective.
Let $\psi_1\colon C\to X_{\Sigma}$ and $\psi_2\colon C\to X_{\Sigma}$ be maps satisfying
\[
P_L\circ \psi_1 = P_L\circ \psi_2 = \phi.
\]
Then, there is a unique covering transformation $Q$ relative to $P_L$ such that
\[
Q\circ \psi_1 = \psi_2.
\]
\end{lem}
\proof
By the analyticity of the relevant maps,
 it suffices to prove that the equation $Q\circ \psi_1 = \psi_2$ holds on any open subset
 of $C$.
Let $p\in \phi(C)$ be a general point.
Since $\phi$ is generically injective, $\phi^{-1}(p)$ consists of a single point, which we denote by $p'$.
Then, by Lemma \ref{lem:transitivecover}, there is a unique 
 covering transformation $Q$ relative to $P_L$ such that
\[
Q(\psi_1(p')) = \psi_2(p').
\]
Since the map $P_L$ is locally an analytic isomorphism on small neighborhoods of $\psi_1(p')$
 and $\psi_2(p')$, the map $Q$ satisfies the condition $Q\circ \psi_1 = \psi_2$.\qed\\

We will apply Lemma \ref{lem:mapcover} to the local lifts $\psi_{0, V_1V_2}$ of $\varphi_0$.
However, in this case, the open subset $\mathscr D_{v_1v_2}$ of the curve $C_0$ introduced in Definition \ref{def:Dv1v2} 
 is reducible, so the statement of the lemma requires slight modification.
Note that $\mathscr D_{v_1v_2}$ has two irreducible components corresponding to two adjacent vertices of $\Gamma$,
 denoted $v_1$ and $v_2$.
We are assuming that at least one of these vertices, say $v_2$, is 2-valent.
Let $\psi_{0, V_1V_2}^{(i)}$ be a map lifting
 $\varphi_0|_{\mathscr D_{v_1v_2}}$, see Lemma \ref{lem:v_1v_2}.
By construction, the image of $\psi_{0, V_1V_2}^{(i)}$ is
 contained in the variety $Y_{0, V_1V_2}^s$, which has two irreducible components, 
 denoted $Y_{0, 1}^s$ and $Y_{0, 2}^s$.
We assume that the irreducible component of $\mathscr D_{v_1v_2}$ corresponding to $v_2$
 is mapped to $Y_{0, 2}^s$.
Let 
\[
G_{V_1V_2}\cong \Bbb C^*
\]
 be the one-dimensional subtorus of the torus acting on $Y_{0, 2}^s$
 corresponding to the direction of the edge connecting the vertices $V_1$ and $V_2$.
Now, the map given by 
\begin{itemize}
\item the multiplication of an element 
$
g\in G_{V_1V_2}
$ 
 on $Y_{0, 2}^s$, and
\item the identity map on $Y_{0, 1}^s$, 
\end{itemize}
defines an automorphism of $Y_{0, V_1V_2}^s$.

Then,
 as discussed in the proof of 
 Lemma \ref{lem:v_1v_2}, there is an action of a subgroup of $G_{V_1V_2}$
 on the set of maps $\psi_{0, V_1V_2}$ which lift $\varphi_0$ on $\mathscr D_{v_1v_2}$.
It is straightforward to see the following.
\begin{lem}\label{lem:auto}
Let $\psi_{0, V_1V_2}^{(i)}$ and $\psi_{0, V_1V_2}^{(i')}$ be lifts of
 $\varphi_0$ on $\mathscr D_{v_1v_2}$, as described in Lemma \ref{lem:v_1v_2}.
Then, there is a unique covering transformation $Q$ relative to the toric morphism $P_{\Phi}$,
 and a unique element $g\in G_{V_1V_2}\cong \Bbb C^*$
 corresponding to a suitable root of unity, satisfying
\[
g\circ Q\circ \psi_{0, V_1V_2}^{(i)} = \psi_{0, V_1V_2}^{(i')}.
\]
In particular, for any maps $\psi_{0, V_1V_2}^{(i)}$ and  $\psi_{0, V_1V_2}^{(i')}$,
 there is a unique automorphism $\overline Q$ of $Y_{0, V_1V_2}^s$ satisfying
\[
P_{\Phi}\circ \overline Q = P_{\Phi}
\]
 on $Y_{0, V_1V_2}^s$, and
\[
\overline Q\circ \psi_{0, V_1V_2}^{(i)} = \psi_{0, V_1V_2}^{(i')}
\]
 on $\mathscr D_{v_1v_2}$.
\qed
\end{lem}

\begin{rem}\label{rem:Q}
Note that the action of $G_{V_1V_2}$ fixes the component $Y_{0, 1}^s$ of $Y_{0, V_1V_2}^s$.
Therefore, while it acts on the space $Y_{0, V_1V_2}^s$, it is not a toric automorphism in the sense that it
 does not extend to an action on the larger space $\mathcal Y_{V_1V_2}^s$
 (see Remark \ref{rem:lift_log}).
\end{rem}

This claim holds regardless of the valences of the two adjacent vertices of $\Gamma$, 
 that is, it holds in either of the following cases: 
\begin{itemize}
\item One vertex is 3-valent, and the other is 2-valent.
\item Both vertices are 2-valent.
\end{itemize}

%We also note the following property.
%Recall that with the natural log structure on the toric variety $X_{\Sigma}$, 
% its log tangent sheaf $\Theta_{X_{\Sigma}}$ is canonically isomorphic to 
%\[
%\Bbb Z^n\otimes_{\Bbb Z} \mathcal O_{X_{\Sigma}},
%\]
% where the generator $e_i$ of $\Bbb Z^n$ corresponds to the log tangent vector field 
%\[
%z^{e_i}\partial_{z^{e_i}}.
%\]
%
%\begin{lem}\label{lem:diffQ}
%Let $Q$ be covering maps with respect to $P_L$.
%Then, it induces the identity map on $\Theta_{X_{\Sigma}}$.
%Namely, the differential of $Q$ maps the vector field $z^{e_i}\partial_{z^{e_i}}$ to itself.
%\end{lem}
%\proof 
%This is a direct consequence of Claim \ref{claim:Q}.\qed\\

\subsubsection{The log normal sheaf associated with the lifts $\psi_{0, V_1V_2}$}\label{subsec:psinormal}
In this subsection, we construct an invertible sheaf $\mathcal N$ on $C_0$.
Ultimately, the sheaf $\mathcal N$ turns out to be
 isomorphic to the already mentioned log normal sheaf $\mathcal N_{\varphi_0}$
 of the map $\varphi_0$, see Lemma \ref{lem:NN_phi}.
However, the point is that we construct $\mathcal N$ from the data of local lifts $\psi_{0, V_1V_2}$.
Thus, we can regard a cocycle constructed from $\psi_{0, V_1V_2}$ by tracing the calculation in 
 Subsection \ref{subsubsec:stdcocycle}, as an $\mathcal N_{\varphi_0}$-valued cocycle.

Let $\{v_1, v_2\}$ and $\{v_1, v_3\}$ be pairs of adjacent vertices of $\Gamma$, where $v_2\neq v_3$.
Let 
\[
P_{\Phi}\colon \mathcal Y_{V_1V_2}^s\to \mathcal Y_{v_1v_2}
\]
 and 
\[
P_{\Phi'}\colon \mathcal Y_{V_1V_3}^s\to \mathcal Y_{v_1v_3}
\]
 be the toric maps as described in Definition \ref{def:Pphi}.  
Let $Y_{0, V_1V_2}^s$ and $Y_{0, V_1V_3}^s$ be the central fibers of 
 $\mathcal Y_{V_1V_2}^s$ and $\mathcal Y_{V_1V_3}^s$, respectively.
%Let 
% $(\Gamma_s, h_s)$ and $(\Gamma_s', h_s')$ be the standard tropical curves associated with 
% $Y_{0, V_1V_2}^s$ and $Y_{0, V_1V_3}^s$, respectively, see Subsection \ref{subsec:std}.
%Note that $(\Gamma_s, h_s)$ and $(\Gamma_s', h_s')$ are isomorphic, 
% but it is better to distinguish their notations for clarity.
% 

Recall that we have an open covering $\{\mathscr D_{vv'}\}$ of $C_0$, see Definition \ref{def:Dv1v2}.
Let $\psi_{0, V_1V_2}\colon \mathscr D_{v_1v_2}\to Y_{0, V_1V_2}^s$
 be an arbitrary lift of $\varphi_0|_{\mathscr D_{v_1v_2}}$ that is compatible with the 
 log structure in the sense of Remark \ref{rem:lift_log} (2).
We choose such a lift for every element of $\{\mathscr D_{vv'}\}$.
The following result is straightforward.
\begin{lem}
Let $\mathcal N_{v_1v_2}$ be the log normal sheaf of $\psi_{0, V_1V_2}$.
Then, the map $P_{\Phi}$ induces an isomorphism
\[
\mathcal N_{v_1v_2}\cong \mathcal N_{\varphi_0}|_{\mathscr D_{v_1v_2}}.
\]
\end{lem}
\proof
Note that the map $P_{\Phi}$ induces a map between the log tangent sheaves of $Y_{0, V_1V_2}^s$
 and $Y_{0, v_1v_2}$, which is a fiberwise isomorphism.
Therefore, it induces an isomorphism
\[
\begin{array}{ll}
\mathcal N_{v_1v_2} & = \psi_{0, V_1V_2}^*\Theta_{Y_{0, V_1V_2}^s}/\Theta_{\mathscr D_{v_1v_2}}\\
  &\hspace{-.5in} \to \mathcal N_{\varphi_0}|_{\mathscr D_{v_1v_2}}
   = (\varphi_0|_{\mathscr D_{v_1v_2}})^*\Theta_{Y_{0, v_1v_2}}/\Theta_{\mathscr D_{v_1v_2}}
   = \psi_{0, V_1V_2}^*P_{\Phi}^*\Theta_{Y_{0, v_1v_2}}/\Theta_{\mathscr D_{v_1v_2}}.
\end{array}
\]
This proves the claim.
\qed\\

Similarly, there is a natural isomorphism 
\[
\mathcal N_{v_1v_3}\cong \mathcal N_{\varphi_0}|_{\mathscr D_{v_1v_3}}.
\]
Thus, we can glue $\mathcal N_{v_1v_2}$
 and $\mathcal N_{v_1v_3}$
 using the given gluing between $\mathcal N_{\varphi_0}|_{\mathscr D_{v_1v_2}}$
 and $\mathcal N_{\varphi_0}|_{\mathscr D_{v_1v_3}}$ on $\mathscr D_{v_1v_2}\cap \mathscr D_{v_1v_3}$.
If $v_1$ is 3-valent and there is another open subset $\mathscr D_{v_1v_4}$, 
 the gluing between $\mathcal N_{v_1v_2}$, $\mathcal N_{v_1v_3}$ and $\mathcal N_{v_1v_4}$
 satisfy the cocycle condition by construction.
Therefore, we obtain an invertible sheaf $\mathcal N$ on $C_0$.
The following result is immediate from the construction.
\begin{lem}\label{lem:NN_phi}
The sheaf $\mathcal N$ is isomorphic to $\mathcal N_{\varphi_0}$.\qed
\end{lem}

Note that in the construction of $\mathcal N$, we have fixed a lift $\psi_{0, V_1V_2}$ of 
 $\varphi_0|_{\mathscr D_{v_1v_2}}$, and we have also fixed an isomorphism
\[
\iota_{\psi_{0, V_1V_2}}\colon \mathcal N_{\psi_{0, V_1V_2}}\cong \mathcal N_{\varphi_0}|_{\mathscr D_{v_1v_2}}.
\]
Here, we write $\mathcal N_{v_1v_2}$ as $\mathcal N_{\psi_{0, V_1V_2}}$
 to emphasize that it is the log normal sheaf of $\psi_{0, V_1V_2}$. 
Recall that if $\psi_{0, V_1V_2}'$ is another lift of $\varphi_0|_{\mathscr D_{v_1v_2}}$, 
 then there is a unique covering map $Q$ relative to $P_{\Phi}$ satisfying
\[
\psi_{0, V_1V_2} = Q\circ \psi_{0, V_1V_2}'.
\]
The differential $DQ$ of $Q$ induces a natural isomorphism 
\[
\mathcal N_{\psi_{0, V_1V_2}'} \cong \mathcal N_{\psi_{0, V_1V_2}}.
\]

For future reference, we note the following simple claim.
\begin{lem}\label{lem:compatibility}
The collection of isomorphisms $\{\iota_{\psi_{0, V_1V_2}}\}$ is compatible with the isomorphisms $DQ$
 in the sense that 
\[
\iota_{\psi_{0, V_1V_2}} \circ DQ = \iota_{\psi_{0, V_1V_2}'}
\]
 holds.
\end{lem}
\proof
This follows directly from the defining property of $Q$, namely, $P_{\Phi} = P_{\Phi}\circ Q$,
 and the fact that the isomorphisms $\iota_{\psi_{0, V_1V_2}}$ and $\iota_{\psi_{0, V_1V_2}'}$
 are induced by $P_{\Phi}$.\qed

\subsubsection{Construction of obstruction cocycles via lifts}\label{subsec:constcocycle}
In this subsection, we construct an obstruction cocycle for deforming $\varphi_n$ by utilizing the local lifts
 $\psi_{n, VV'}$.
To apply the result from Subsection \ref{subsec:pairing}, we restrict the domain of 
 $\psi_{n, VV'}$ to the smaller open subsets $\mathscr U_i$ introduced there, see Definition \ref{def:scrU}.
We then apply the construction described in Subsection \ref{subsubsec:stdcocycle}.
To do so, we first need to take defining equations of the images of $\psi_{n, VV'}|_{\mathscr U_i}$.
To obtain a well-defined cocycle, we need to carefully handle the choices of 
 such equations.

Let $p$ be the node of $C_0$ corresponding to the edge $e_{v_1v_2}$.
Let $\mathscr D_{v_1v_2}$ be the open subset of 
 $C_{0, v_1}\cup C_{0, v_2}$ introduced in Definition \ref{def:Dv1v2}.
It can also be described as
\[
\mathscr D_{v_1v_2}
  = (C_{0, v_1}\cup C_{0, v_2})\cap \varphi_0^{-1}
       (\mathcal Y_{v_1v_2}).
\]
%By Proposition \ref{prop:liftmap}, 
By Lemma \ref{lem:v_1v_2},
 the restriction $\varphi_0|_{\mathscr D_{v_1v_2}}$
 admits a lift to a map 
\[
\psi_{0, V_1V_2}
  \colon \mathscr D_{v_1v_2}\to Y_{0, V_1V_2}^s.
\]
As in the previous subsection, we choose $\psi_{0, V_1V_2}$ so that it is compatible
 with the given log structure on $C_0$
 in the sense of Remark \ref{rem:lift_log}.

There are several such lifts, and we choose any one of them.
Let $\mathscr U_i$ be the unique element of the open covering $\{\mathscr U_i\}$
 of $C_0$ which contains the node $p$.
Assume that there exist other open subsets $\mathscr U_j$ and $\mathscr U_k$ of $C_0$ 
 that do not contain nodes, and whose images
 under $\varphi_0$ are contained in the component of $Y_{0, v_1v_2}$
 defined by $y = 0$.
Here, $Y_{0, v_1v_2} \cong \Spec\Bbb C[x, y, z^{\pm 1}]/(xy)$ as before.
Assume $\mathscr U_i\cap \mathscr U_j\cap \mathscr U_k\neq\emptyset$.

We apply the construction from Subsection \ref{subsubsec:stdcocycle} to these open subsets.
However, we use the lift $\psi_{n, V_1V_2}$ in place of the map $\varphi_n$.
To do this, we need to fix defining equations of the images of $\mathscr U_i, \mathscr U_j$ and $\mathscr U_k$
 by $\psi_{n, V_1V_2}$.
When the open subset contains a node, a slightly different argument is required.
So, we first consider the case of the open subsets $\mathscr U_j$ and $\mathscr U_k$, which do not contain a node.

Recall that we take $\mathscr U_j$ and $\mathscr U_k$ sufficiently small so that 
 $\varphi_0|_{\mathscr U_j}$ and $\varphi_0|_{\mathscr U_k}$
 are embeddings, and their images are contained in the
 open subsets $\mathscr W_j$ and $\mathscr W_k$ of $Y_{0, v_1v_2}$, respectively (see Definition \ref{def:scrW}).
Each of these open subsets of $Y_{0, v_1v_2}$ is a product of discs disjoint from the toric boundary. 
Then, the inverse images $P_{\Phi}^{-1}(\mathscr W_j)$ 
 and $P_{\Phi}^{-1}(\mathscr W_k)$ are disjoint unions of open subsets
 isomorphic to $\mathscr W_j$ and $\mathscr W_k$, respectively.
Among these, there exist unique connected components 
\begin{equation}\label{eq:scrW}
\mathscr W_{j, V_1V_2}^s, \;\; \mathscr W_{k, V_1V_2}^s
\end{equation} 
 of $P_{\Phi}^{-1}(\mathscr W_j)$
 and $P_{\Phi}^{-1}(\mathscr W_k)$,
 which contain 
 the images of 
 $\psi_{0, V_1V_2}|_{\mathscr U_j}$
 and $\psi_{0, V_1V_2}|_{\mathscr U_k}$, respectively.

Assume there is an $n$-th order deformation
 $\varphi_n\colon C_n\to X_n$ of the map $\varphi_0$.
By Proposition \ref{prop:liftmap}, $\varphi_n|_{\mathscr D_{v_1v_2}}$ admits a lift to a map 
 $\psi_{n, V_1V_2}$,
 which is a deformation of $\psi_{0, V_1V_2}$.
Let 
\[
F_{j, n} = 0, \;\; F_{k, n} = 0
\]
 be defining equations of the images of $\varphi_n|_{\mathscr U_{j, n}}$
 and $\varphi_n|_{\mathscr U_{k, n}}$ in $\mathscr W_{j, n}$ and $\mathscr W_{k, n}$, respectively.
Then, the equations
\[
P_{\Phi}^*F_{j, n}|_{\mathscr W_{j, n, V_1V_2}^s} = 0,\;\; P_{\Phi}^*F_{k, n}|_{\mathscr W_{k, n, V_1V_2}^s} = 0
\]
 provide defining equations of 
 the images of $\psi_{n, V_1V_2}|_{\mathscr U_{j, n}}$
 and $\psi_{n, V_1V_2}|_{\mathscr U_{k, n}}$.
Here, 
\begin{equation}\label{eq:scrW2}
\mathscr W_{j, n, V_1V_2}^s,\;\; \mathscr W_{k, n, V_1V_2}^s
\end{equation}
 are the restrictions of 
 $\mathcal Y_{V_1V_2}^s\times_{\Spec\Bbb C[t]}\Spec\Bbb C[t]/(t^{n+1})$,
 seen as an analytic ringed space,
 to $\mathscr W_{j, V_1V_2}^s$ and $\mathscr W_{k, V_1V_2}^s$, respectively.
This fixes defining equations of the images of $\mathscr U_{j, n}$ and $\mathscr U_{k, n}$ for any choice of 
 a lift $\psi_{n, V_1V_2}$ of $\varphi_n|_{\mathscr D_{v_1v_2}}$.

Next, consider the case of $\mathscr U_i$, which contains the node $p$.
In this case, the image of $\varphi_0|_{\mathscr U_i}$ is an analytic subset of the open subset $\mathscr W_i$
 of $Y_{0, v_1v_2}$, which is a normal crossing union of products of discs, see Definition \ref{def:scrW}.
In this case too, there is a unique connected component 
\[
\mathscr W_{i, V_1V_2}^s
\]
 of $P_{\Phi}^{-1}(\mathscr W_i)$,
 which contains the image of $\psi_{0, V_1V_2}|_{\mathscr U_i}$.
However, the map 
\[
P_{\Phi}|_{\mathscr W_{i, V_1V_2}^s}\colon \mathscr W_{i, V_1V_2}^s\to \mathscr W_i
\]
 is a branched covering with covering degree equal to the weight 
 $W_e$ of the edge $e$ of $\Gamma$ corresponding to the node $p$.
In particular, there are $W_e$ lifts of $\varphi_0|_{\mathscr U_i}$,
 each with its image contained in $\mathscr W_{i, V_1V_2}^s$.
Note that the images of these lifts
 need not be disjoint.
In fact, if both vertices $v_1$ and $v_2$ are 2-valent, all these lifts have the same image, but they are distinguished by 
 the action given by multiplying $W_e$-th roots of unity, see Remark \ref{rem:lift_log}.
One of these lifts is the restriction to $\mathscr U_i$ of the map $\psi_{0, V_1V_2}$ we choose.

Given a deformation $\varphi_n$ of $\varphi_0$, let $\psi_{n, V_1V_2}|_{\mathscr U_{i, n}}$ be the unique lift of 
 $\varphi_n|_{\mathscr U_{i, n}}$ which reduces to $\psi_{0, V_1V_2}|_{\mathscr U_{i}}$ over $\Bbb C[t]/(t)$.
Let
\[
\widetilde F_{i, n, V_1V_2} = 0
\]
 be a defining equation of the image of $\psi_{n, V_1V_2}|_{\mathscr U_{i, n}}$.

Suppose we choose a different lift $\psi_{n, V_1V_2}'$ of $\varphi_n|_{\mathscr D_{v_1v_2}}$, 
 whose image is contained in a connected component $(\mathscr W_{i, V_1V_2}^s)'$ 
 of $P_{\Phi}^{-1}(\mathscr W_i)$.
Then, 
 the defining equation of the image of $\psi_{n, V_1V_2}'|_{\mathscr U_{i, n}}$
 is taken to be the one obtained as
\[
Q^*\widetilde F_{i, n, V_1V_2} = 0,
\]
 where $Q$ is the unique covering transformation relative to $P_{\Phi}$ which sends $\psi_{0, V_1V_2}'$
 to $\psi_{0, V_1V_2}$, see Lemma \ref{lem:auto}.

Thus, we have fixed a defining equation of the image of any lift $\psi_{n, V_1V_2}|_{\mathscr U_i}$
 of $\varphi_n|_{\mathscr U_i}$, 
 for any open subset $\mathscr U_i$ in the covering $\{\mathscr U_i\}$, 
 and for any pair of adjacent vertices $v_1$ and $v_2$
 of $\Gamma$.

Then, for any pair of open subsets $\mathscr U_i$ and $\mathscr U_j$ satisfying
 $\mathscr U_i\cap \mathscr U_j\neq\emptyset$,
 we attach a local section 
\[
\nu_{ij, n+1}
\] 
 of the sheaf $\mathcal N$ 
 by utilizing these defining equations and the construction in Subsection \ref{subsubsec:stdcocycle}. 
To elaborate a bit more, note that for any two intersecting open subsets
 $\mathscr U_i$ and  $\mathscr U_j$ of $C_0$,
 there is an open subset $\mathscr D_{v_1v_2}$ containing
 both.
Then, we choose a lift 
\[
\psi_{n, V_1V_2}\colon \mathscr D_{v_1v_2}\to \mathcal Y_{V_1V_2}^s
\]
 of $\varphi_n|_{\mathscr D_{v_1v_2}}$. 
By the above construction, we have fixed the defining equations of
 the images of $\psi_{n, V_1V_2}|_{\mathscr U_i}$ and $\psi_{n, V_1V_2}|_{\mathscr U_j}$.
Using these equations and the construction in Subsection \ref{subsubsec:stdcocycle}, 
 we obtain the local section $\nu_{ij, n+1}$ of $\mathcal N$ on $C_0$.

However, as is evident from the description of the construction, 
 there are multiple	choices for these local sections.
Namely, given two open subsets $\mathscr U_i$ and $\mathscr U_j$, 
 there are several possible choices of 
 $\mathcal Y_{v_1v_2}$ containing both.
For different choices of $\mathcal Y_{v_1v_2}$ and $\mathcal Y_{v_3v_4}$ containing these open subsets, 
 we need to use different lifts $\psi_{n, V_1V_2}$ and $\psi_{n, V_3V_4}$.
Moreover, for a fixed $\mathcal Y_{v_1v_2}$,
 there are multiple
 choices of lifts such as $\psi_{n, V_1V_2}$ and $\psi_{n, V_1V_2}'$ as discussed above.

These different choices might yield different sections $\nu_{ij, n+1}$ for a fixed pair of open subsets,
 $\mathscr U_i$ and $\mathscr U_j$.
To obtain a well-defined 1-cocycle, we need to verify that this does not occur.
This is the content of the next lemma.
\begin{lem}\label{lem:cocycle}
The local sections  $\{\nu_{ij, n+1}\}$ of the log normal sheaf $\mathcal N_{\varphi_0}$ of $\varphi_0$,
 which is identified with the sheaf $\mathcal N$,
 obtained by this construction form a well-defined \v{C}ech 1-cocycle with values in $\mathcal N_{\varphi_0}$
 on $C_0$
 with respect to the covering $\{\mathscr U_i\}$.
\end{lem}
\proof
We verify that, for any pair of intersecting open subsets $\mathscr U_i$ and $\mathscr U_j$
 in $\{\mathscr U_i\}$, 
 the associated section does not depend on
\begin{itemize}
\item the choice of the covering
 $P_{\Phi}\colon\mathcal Y_{V_1V_2}^s\to \mathcal Y_{v_1v_2}$
 to which the maps $\varphi_0|_{\mathscr U_i}$ and $\varphi_0|_{\mathscr U_j}$ are lifted, and
\item the specific choice of a lift of $\varphi_0$ once
 $P_{\Phi}$ is fixed.
\end{itemize}
In fact, once this is verified, Lemma \ref{lem:obst} ensures that the cocycle condition is satisfied.

%Recall that $\mathcal Y_{V_1V_2}^s$ and $\mathcal Y_{\overline{\mathcal C}_{h_{v_1v_2}(\Gamma_{v_1v_2})}}$
% contain open toric subvarieties $\mathring{\mathcal Y}_{V_1V_2}^s$
% and $\mathring{\mathcal Y}_{\overline{\mathcal C}_{h_{v_1v_2}(\Gamma_{v_1v_2})}}$, respectively, 
% see Lemma \ref{lem:mringY}.
First, we address the dependence on the choice of $P_{\Phi}$.
For each $\mathcal Y_{v_1v_2}$, 
 we choose a lift $\psi_{n, V_1V_2}$ of 
 $\varphi_n|_{\mathscr D_{v_1v_2}}$ as above.
When an open subset $\mathscr U_i$ contains a node, there is a unique 
 $\mathcal Y_{v_1v_2}$ which contains its image.
Therefore, in this case, there is no choice for $P_{\Phi}$, and we can remove such open subsets from consideration.

Now, let $\mathscr U_j$ and $\mathscr U_k$ be open subsets
 of $C_0$ such that $\mathscr U_j\cap\mathscr U_k\neq \emptyset$, and
 neither contains a node.
Furthermore, assume that their images are contained in 
 both $\mathcal Y_{v_1v_2}$
 and $\mathcal Y_{v_1v_3}$.
Let 
\[
P_{\Phi}\colon \mathcal Y_{V_1V_2}^s\to \mathcal Y_{v_1v_2},\;\; P_{\Phi'}\colon \mathcal {Y}_{V_1V_3}^s\to \mathcal Y_{v_1v_3}
\]
 be the maps from toric varieties associated with standard tropical curves with two vertices.

Let 
\[
F_{j, n}(x, y, z, t) = 0
\]
 be the defining equation of the image of $\varphi_n|_{\mathscr U_j}$ chosen in the 
 construction above, expressed in the coordinate system of $\mathcal Y_{v_1v_2}$.
Let 
\[
F_{j, n}(x', y', z', t) = 0
\]
 be the same equation expressed
 in the coordinate system of $\mathcal Y_{v_1v_3}$.
The functions $F_{k, n}(x, y, z, t)$ and $F_{k, n}(x', y', z', t)$ are determined similarly.
Then, by the aforementioned construction, 
 we take the defining equation of the image of $\psi_{n, V_1V_2}|_{\mathscr U_j}$ as
\[
P_{\Phi}^*F_{j, n} = 0.
\]
Similarly, the image of $\psi_{n, V_1V_2}|_{\mathscr U_k}$ is given by 
\[
P_{\Phi}^*F_{k, n} = 0.
\]
Consequently, the section $\nu_{jk, n+1}$ is defined by 
\[
t^{n+1}\nu_{jk, n+1} = P_{\Phi}^*F_{j, n} -  P_{\Phi}^*G_{jk, n}P_{\Phi}^*F_{k, n}
\]
 on $\mathscr U_j\cap \mathscr U_k$,
 as in \cref{eq:nu}.
Here, $G_{jk, n}$ is an analytic function satisfying the condition
\[
F_{j, n} = G_{jk, n}F_{k, n}
\]
 over $\Bbb C[t]/(t^{n+1})$, and is extended to a function over $\Bbb C[t]/(t^{n+2})$
 by the procedure in Subsection \ref{subsec:split}.

Therefore, the section $\nu_{jk, n+1}$ is in fact the pullback by $P_{\Phi}$ of 
 the section $\nu_{jk, n+1}^{(1)}$ defined by 
\[
t^{n+1}\nu_{jk, n+1}^{(1)} = F_{j, n}- G_{jk, n}F_{k, n}
\]
 for $\varphi_n$.
Thus, if this section is independent of the choice of the subvarieties of the form $\mathcal Y_{v_1v_2}$ which contain 
 $\mathscr U_j$ and $\mathscr U_k$, the desired independence follows.
However, this is precisely the content of Corollary \ref{cor:welldef}.
Explicitly, if $G_{jk, n}(x', y', z', t)$ is the function obtained from $G_{jk, n}$ by the coordinate
 transformation from $\mathcal Y_{v_1v_2}$ to $\mathcal Y_{v_1v_3}$, and regarded as a function defined over
 $\Bbb C[t]/(t^{n+2})$, then the section $\nu_{jk, n+1}^{(2)}$ defined by
\[
t^{n+1}\nu_{jk, n+1}^{(2)} = F_{j, n}(x', y', z', t) - G_{jk, n}(x', y', z', t)F_{k, n}(x', y', z', t)
\]
 is also obtained from $\nu_{jk, n+1}^{(1)}$ by the coordinate transformation 
 from $\mathcal Y_{v_1v_2}$ to $\mathcal Y_{v_1v_3}$.

Next, we consider the dependence on the choices of $\psi_{n, V_1V_2}$ after fixing $\mathcal Y_{v_1v_2}$.
By Lemma \ref{lem:mapcover}, different lifts are related by covering transformations relative to $P_{\Phi}$.
Furthermore, the defining equations of the images of $\psi_{n, V_1V_2}|_{\mathscr U_i}$ have been chosen equivariantly
 with respect to these covering transformations.
Then, the claim is a consequence of Lemma \ref{lem:compatibility}. \qed\\

%Note that although we give a proof in the case of the map $\varphi_0$ defined over $\Bbb C[t]/t$, 
% the same proof holds for $\varphi_n$, when an $n$-th order deformation $\varphi_n$ of $\varphi_0$
% exists.
Let $\{\nu_{ij, n+1}\}$ be the \v{C}ech 1-cocycle on $C_0$ with values in $\mathcal N$
 obtained in this manner.
Now, we establish the following result.
\begin{prop}\label{prop:locallifts}
The cohomology class of $\{\nu_{ij, n+1}\}$ represents the obstruction to deforming $\varphi_n$ to a map over
 $\Bbb C[t]/(t^{n+2})$.
\end{prop}
\proof
By Lemma \ref{lem:NN_phi}, we regard $\{\nu_{ij, n+1}\}$ as a 1-cocycle with values in $\mathcal N_{\varphi_0}$.
Assume that the cohomology class of $\{\nu_{ij, n+1}\}$ vanishes.
This implies the existence of a section
 $\eta_i$ of the restriction $\mathcal N_{\varphi_0}|_{\mathscr U_i}$ 
 such that 
\[
\eta_i-G_{ij, 0}\eta_j = \nu_{ij, n+1}
\]
 on each $\mathscr U_i\cap \mathscr U_j$.
Recall that $G_{ij, 0}$ is a holomorphic function on  $\mathscr W_{i, 0}\cap \mathscr W_{j, 0}$ which restricts to 
 the transition function of the log normal sheaf of $\varphi_0$ on $\mathscr U_i\cap \mathscr U_j$.

By definition of $\mathcal N$ and Lemma \ref{lem:compatibility}, we can regard $\eta_i$ as a section of
 the log normal sheaf $\mathcal N_{\psi_{0, V_1V_2}}|_{\mathscr U_i}$ 
 of any lift $\psi_{0, V_1V_2}|_{\mathscr U_i}$ of $\varphi_0|_{\mathscr U_i}$.
When $\mathscr U_i$ does not contain a node, we have an exact sequence
\[
0\to \mathcal O_{\mathscr W_{i, 0, V_1V_2}^s}\to \mathcal O_{\mathscr W_{i, 0, V_1V_2}^s}(\psi_{0, V_1V_2}(\mathscr U_i))
 \to \mathcal N_{\psi_{0, V_1V_2}}|_{\mathscr U_i}\to 0.
\]
Here, $\mathcal O_{\mathscr W_{i, 0, V_1V_2}^s}(\psi_{0, V_1V_2}(\mathscr U_i))$ 
is the sheaf of analytic functions on $\mathscr W_{i, 0, V_1V_2}^s$
 which can have a first-order pole along $\psi_{0, V_1V_2}(\mathscr U_i)$.
Since we take $\mathscr W_i$ to be biholomorphic to a product of discs,
 we have 
\[
H^1(\mathscr W_{i, 0, V_1V_2}^s, \mathcal O_{\mathscr W_{i, 0, V_1V_2}^s})=0
\]
 by the Poincar\'e Lemma for holomorphic functions.
Thus, a section of $\mathcal N_{\psi_{0, V_1V_2}}|_{\mathscr U_i}$ can be extended to a section 
 of $\mathcal O_{\mathscr W_{i, 0, V_1V_2}^s}(\psi_{0, V_1V_2}(\mathscr U_i))$.

When $\mathscr U_i$ contains a node, we have a similar sequence, but now $\mathscr W_{i, 0, V_1V_2}^s$
 is the normal crossing union of two products of discs.
If $\mathscr W_{1}^s$ and $\mathscr W_{2}^s$ are the irreducible components of $\mathscr W_{i, 0, V_1V_2}^s$, 
 we have an exact sequence
\[
0\to \mathcal O_{\mathscr W_{i, 0, V_1V_2}^s}\to 
 (i_1)_*\mathcal O_{\mathscr W_{1}^s}\oplus (i_2)_*\mathcal O_{\mathscr W_2^s}
 \to (i_{12})_*\mathcal O_{\mathscr W_{1}^s\cap \mathscr W_2^s}\to 0, 
\]
 where 
\[
i_1\colon \mathscr W_1^s\to \mathscr W_{i, 0, V_1V_2}^s,\;\; i_2\colon \mathscr W_2^s\to \mathscr W_{i, 0, V_1V_2}^s,\;\;
 i_{12}\colon \mathscr W_1^s\cap \mathscr W_2^s\to \mathscr W_{i, 0, V_1V_2}^s
\]
 are inclusions.
Since 
\[
H^1(\mathscr W_{i, 0, V_1V_2}^s, (i_1)_*\mathcal O_{\mathscr W_1^s}\oplus (i_2)_*\mathcal O_{\mathscr W_2^s})
 \cong H^1(\mathscr W_1^s, \mathcal O_{\mathscr W_1^s})\oplus H^1(\mathscr W_2^s, \mathcal O_{\mathscr W_2^s}),
\]
 and $H^1(\mathscr W_i^s, \mathcal O_{\mathscr W_i^s})$ vanishes by the Poincar\'e Lemma, 
 the cohomology group 
\[
H^1(\mathscr W_{i, 0, V_1V_2}^s, (i_1)_*\mathcal O_{\mathscr W_1^s}\oplus (i_2)_*\mathcal O_{\mathscr W_2^s})
\]
 also vanishes.
Moreover, the map 
\[
H^0(\mathscr W_{i, 0, V_1V_2}^s, (i_1)_*\mathcal O_{\mathscr W_1^s}\oplus (i_2)_*\mathcal O_{\mathscr W_2^s})
 \to H^0(\mathscr W_{i, 0, V_1V_2}^s, (i_{12})_* \mathcal O_{\mathscr W_1^s\cap W_2^s})
\]
 is a surjection, since
 a holomorphic function on $\mathscr W_1^s\cap \mathscr W_2^s$ 
 can be extended to both $\mathscr W_1^s$ and $\mathscr W_2^s$
 as they are products of discs.
Therefore, we obtain
\[
H^1(\mathscr W_{i, 0, V_1V_2}^s, \mathcal O_{\mathscr W_{i, 0, V_1V_2}^s}) = 0.
\]
It follows that a section of $\mathcal N_{\psi_{0, V_1V_2}}|_{\mathscr U_i}$
 can be extended to a section in $\mathcal O_{\mathscr W_{i, 0, V_1V_2}^s}(\psi_{0, V_1V_2}(\mathscr U_i))$.
 
Let $h_i$ be an element in $\mathcal O_{W_{i, 0, V_1V_2}^s}(\psi_{0, V_1V_2}(\mathscr U_i))$
 which extends $\eta_i$. 
Assume that $\varphi_0(\mathscr U_i)$ is contained in another $\mathcal Y_{v_1v_3}$.
In this case, the open subset $\mathscr U_i$ does not contain a node of $C_0$.
Then, as in the proof of Lemma \ref{lem:cocycle}, the restrictions of the maps $P_{\Phi}\colon \mathcal Y_{V_1V_2}^s\to \mathcal Y_{v_1v_2}$
 and $P_{\Phi'}\colon \mathcal Y_{V_1V_3}^s\to \mathcal Y_{v_1v_3}$ 
 to the open subsets $\mathscr W_{i, 0, V_1V_2}^s$ and $\mathscr W_{i, 0, V_1V_3}^s$, respectively,  are isomorphisms onto the 
 same image $\mathscr W_i$.
Here, $\mathscr W_{i, 0, V_1V_3}^s$ is a lift of $\mathscr W_i$ to 
$\mathcal Y_{V_1V_3}^s$ containing the image of $\psi_{0, V_1V_3}|_{\mathscr U_i}$.
Therefore, there is a natural isomorphism between the sheaves $\mathcal O_{\mathscr W_{i, 0, V_1V_2}^s}(\psi_{0, V_1V_2}(\mathscr U_i))$ 
and $\mathcal O_{\mathscr W_{i, 0, V_1V_3}^s}(\psi_{0, V_1V_3}(\mathscr U_i))$.
We can take the extensions of $\eta_i$ compatibly in the sense that under this isomorphism,
 these extensions coincide.

Let $F_{i, n} = 0$ be the defining equation of the image 
 of $\psi_{n, V_1V_2}|_{\mathscr U_{i, n}}$, regarded as an equation over $\Bbb C[t]/(t^{n+2})$.
The the modified equation 
\[
F_{i, n} - t^{n+1}h_i=0
\]
 gives the image of a local deformation 
 $\psi_{n+1, V_1V_2}|_{\mathscr U_{i, n+1}}$ of $\psi_{n, V_1V_2}|_{\mathscr U_{i, n}}$.
By the above calculation, this equation is well-defined in the sense that it
 does not depend on the choices of $\mathcal Y_{v_1v_2}$
 and $\psi_{0, V_1V_2}$.
It then follows that the local deformation $\psi_{n+1, V_1V_2}|_{\mathscr U_{i, n+1}}$ is uniquely determined
 up to isomorphisms, independent of these choices.
By construction, 
 this deformation 
 has the property that the images of the composition $P_{\Phi}\circ \psi_{n+1, V_1V_2}|_{\mathscr U_{i, n+1}}$
 glue on $X_{n+1} = \mathcal X\times_{\Spec\Bbb C[t]}\Spec\Bbb C[t]/(t^{n+2})$.
Namely, we have
\[
(F_{i, n} - t^{n+1}h_i) - G_{ij, n}(F_{j, n} - t^{n+1}h_j)
 = t^{n+1}\nu_{ij, n+1} - t^{n+1}(h_i-G_{ij, 0}h_j).
\]
Since $h_i$ and $h_j$ are extensions of $\eta_i$ and $\eta_j$, respectively, it follows that
\[
t^{n+1}(h_i-G_{ij, 0}h_j) = t^{n+1}(\eta_i-G_{ij, 0}\eta_j) = t^{n+1}\nu_{ij, n+1}
\]
 on $\mathscr U_i\cap\mathscr U_j$.
Thus, we have
\[
F_{i, n} - t^{n+1}h_i = F_{j, n} - t^{n+1}h_j
\]
 on $\mathscr U_i\cap\mathscr U_j$.

This implies that the differences of the local deformations 
\[
P_{\Phi}\circ \psi_{n+1, V_1V_2}|_{\mathscr U_{i, n+1}}
\]
 of $\varphi_n|_{\mathscr U_{i, n}}$ give a representative of a class of $H^1(C_0, \Theta_{C_0})$.
Therefore, if we take an appropriate deformation 
 $C_{n+1}$ of $C_n$, 
 then, not only the images, but the maps $P_{\Phi}\circ \psi_{n+1, V_1V_2}|_{\mathscr U_{i, n+1}}$ themselves glue.
This defines an $(n+1)$-th order deformation of the given $\varphi_n$.
In other words, if the cohomology class of $\{\nu_{ij, n+1}\}$ vanishes, 
 an $(n+1)$-th order deformation of $\varphi_n$ exists.

Conversely, it is straightforward to see that if the cohomology class of $\{\nu_{ij, n+1}\}$
 does not vanish, then 
 an $(n+1)$-th order deformation of $\varphi_n$ does not exist. \qed\\

By this result, combined with Proposition \ref{prop:coupling}, 
 we conclude that the calculation of the obstruction to deforming $\varphi_n$ can be entirely done 
 on $\mathcal Y_{V_1V_2}^s$, as if these are coordinate neighborhoods of $\mathcal X$.
The merit of this is that since the lift $\psi_{0, V_1V_2}\colon U_a\to \mathcal Y_{V_1V_2}^s$ is an embedding, 
 the study of \cite{Nishinou_Obstruction_2018} can be applied.
In the following subsection, we apply this to prove our main result, Theorem \ref{thm:main}.

\subsection{Vanishing of the obstructions}\label{subsec:maincal}
Assuming that an $n(\geq 0)$-th order deformation $\varphi_n\colon C_n\to X_n$ of the pre-log curve 
 $\varphi_0\colon C_0\to X_0$ exists, 
 we proceed to compute the obstruction to deforming it to the $(n+1)$-th order.
The argument follows that of \cite{Nishinou_Obstruction_2018}, which, 
 in turn, was modeled on 
 \cite{Kodaira_Spencer_Semiregularity_1959}.
In the latter, the unobstructedness of 
 deformations of hypersurfaces in a compact complex manifold was shown, under the assumption that 
 the hypersurface satisfies the so-called semiregularity condition.
In the case of a curve $C$ in a surface $S$, the semiregularity means that the natural map 
\[
H^0(S, \mathcal K_S)\to H^0(C, \mathcal N_C^{\vee}\otimes \omega_C)
\]
 is surjective, where $\mathcal K_S$ is the canonical sheaf of $S$, $\mathcal N_C^{\vee}$ is the conormal sheaf of $C$, 
 and $\omega_C$ is the dualizing sheaf of $C$.
Here, we assume $C$ is reduced.
The key to the proof is to construct a 1-cocycle with values in $\mathcal O_S(C)$
 on the surface $S$ which restricts to the obstruction 1-cocycle on the curve.
Then, under the semiregularity assumption, the vanishing of the obstruction follows directly.

In \cite{Nishinou_Obstruction_2018}, the notion of semiregularity was extended to immersions, 
 and a similar unobstructedness result was proved.
However, due to possible singularities of the image arising from the non-injectivity of the map, 
 it is no longer possible to construct a 1-cocycle on the ambient space which restricts to 
 the obstruction class on the curve.
To circumvent this difficulty, we developed a method to represent the obstruction as the sum of 
 local contributions.
Fortunately, this approach also provides a way to handle the present case, where the ambient space $X_0$ is singular.
Specifically, we compute the local contribution to the obstruction via certain contour integral around a loop 
 encircling the singular point of the curve, and since such a loop is mapped to the smooth locus of 
 $X_0$, our method remains applicable.\\

%For this purpose, first we do some calculation on $\widetilde C_0$ and $\widetilde X_0$.

Let $\{\mathscr U_i\}_{i\in I}$ be an open covering of $C_0$ as in the previous subsection.
Let $\mathscr U_{i, n}$ denote the ringed space which is the restriction of $C_n$ to $\mathscr U_i$.
Also, 
 let $\mathscr W_{i,n}$ denote the ringed space which is the restriction of $X_n$ to $\mathscr W_i$. 
As in the previous subsection, based on Proposition \ref{prop:locallifts}, we carry out all the calculations in the standard 
 setting described in Subsection \ref{subsec:std}.
We take a pair $v_1, v_2$ of adjacent vertices of $\Gamma$, and consider the associated objects 
 $\mathcal Y_{V_1V_2}^s$, $\psi_{0, V_1V_2}$, etc.,
 as before.
Let $P_{\Phi}\colon \mathcal Y_{V_1V_2}^s\to \mathcal Y_{v_1v_2}$
 be a branched covering described in Definition \ref{def:Pphi}.

By Proposition \ref{prop:liftmap}, 
 the map $\varphi_n$ locally factors through a map to $\mathcal Y_{V_1V_2}^s$.
Let 
\[
\psi_{n, V_1V_2}|_{\mathscr U_{i, n}} = \psi_{i, n}\colon \mathscr U_{i, n}\to \mathscr W_{i, n, V_1V_2}^s
\]
 be such a map on $\mathscr U_{i, n}$, 
 where $\mathscr W_{i, n, V_1V_2}^s$ is a connected component of $P_{\Phi}^{-1}(\mathscr W_{i, n})$.
Although $\psi_{i, n}$ depends on the choice of $V_1$ and $V_2$, we drop the subscript $V_1V_2$
 from the notation $\psi_{i, n}$ for simplicity, 
 since by Lemma \ref{lem:cocycle}, this choice has little impact on the following argument.
Let $F_{i, n} = 0$ be a defining equation of $\varphi_n(\mathscr U_{i, n})$ in $\mathscr W_{i,n}$.
As in the previous subsection, we take 
 $P_{\Phi}^*F_{i, n}|_{\mathscr W_{i, n, V_1V_2}^s}=0$ 
 to be the defining equation of the image of $\psi_{i, n}$ in $\mathscr W_{i, n, V_1V_2}^s$,
 when $\mathscr U_i$ does not contain a node of $C_0$,
 and take $\widetilde F_{i, n, V_1V_2} = 0$ when $\mathscr U_i$ contains a node.
For simplicity, we write both $P_{\Phi}^*F_{i, n}|_{\mathscr W_{i, n, V_1V_2}^s}$ and $\widetilde F_{i, n, V_1V_2}$
 as 
\[
\widetilde F_{i, n},
\] 
 again because 
 the choice of $V_1$ and $V_2$ plays little role.
Similarly, let $\widetilde F_{j, n} = 0$ be the defining equation of the image of $\psi_{j, n}$ in $\mathscr W_{j, n, V_1V_2}^s$,
 where $\mathscr U_i\cap \mathscr U_j\neq \emptyset$.

On the intersection ${\mathscr W_{i, n, V_1V_2}^s}\cap {\mathscr W_{j, n, V_1V_2}^s}$, 
 these functions satisfy the relation
\[
\widetilde F_{i, n} = G_{ij, n}\widetilde F_{j, n},
\]
 where $G_{ij, n}$ is a holomorphic function defined over $\Bbb C[t]/(t^{n+1})$
 whose reduction over $\Bbb C[t]/(t)$ is non-vanishing.
Recall that by extending these functions to $\Bbb C[t]/(t^{n+2})$, we obtain the difference
\[
t^{n+1}\nu_{ij, n+1} = \widetilde F_{i, n} - G_{ij, n}\widetilde F_{j, n}
\]
 on ${\mathscr W_{i, n+1, V_1V_2}^s}\cap {\mathscr W_{j, n+1, V_1V_2}^s}$.
Here, $\nu_{ij, n+1}$ can be regarded as a holomorphic function on 
 $\mathscr W_{i, 0, V_1V_2}^s\cap \mathscr W_{j, 0, V_1V_2}^s$, 
 and the pullback of these to $C_0$ form an obstruction cocycle to deforming $\varphi_n$, 
 by Proposition \ref{prop:locallifts}.

\subsubsection{Multiplicative representation of the defining equations}\label{subsec:exprep}
To see the vanishing of the cohomology class of $\{\nu_{ij, n+1}\}$, 
 we apply Proposition \ref{prop:coupling}.
To do so, we need to represent $\{\nu_{ij, n+1}\}$ 
 in the form described in Subsection \ref{subsec:pairing}.
Namely, we need to take a set of meromorphic sections $\xi=\{\xi_i\}\cup \{\xi_{i, j}\}$ of $\mathcal N$
 on the locally closed subsets $\{\mathscr U_i\}\cup \{\mathscr U_{i, j}\}$
 associated with the covering $\{\mathscr U_i\}$, so that their differences on intersections
 coincide with the cocycle $\{\nu_{ij, n+1}\}$.
Once this representation is obtained, we can compute the pairing between
 $\{\nu_{ij, n+1}\}$ and the generator $\Psi_{C_0}$ of $H^0(C_0, \mathcal N^{\vee}\otimes\omega_0)$
 introduced in Definition \ref{def:volumeform},
 via the residue calculation.

However, it is in general difficult to obtain a set of meromorphic sections with the required property.
In this paper, following \cite{Nishinou_Obstruction_2018},
 we construct an alternative $C^{\infty}$ representative of the 
 cohomology class using a transcendental approach.
The residue calculation is replaced by a more flexible contour integral, and the well-definedness of the pairing
 follows from the Stokes theorem.
Subsections \ref{subsec:exprep} to \ref{subsec:fiberintegral} are devoted to this construction.\\

We can write the function $\widetilde F_{i, n}$, extended to $\Bbb C[t]/(t^{n+2})$, in the form
\[
\widetilde F_{i, n} = \widetilde F_{i, 0}\exp(f_{i, n+1}),
\]
 where $f_{i, n+1}$ is a function on $\mathscr W_{i, n+1, V_1V_2}^s$
 which can have poles along $\{\widetilde F_{i, 0} = 0\}$, 
 and vanishes when reduced over $\Bbb C[t]/(t)$.
Here, the function $\widetilde F_{i, 0}$ is also regarded as a function on 
  $\mathscr W_{i, n+1, V_1V_2}^s$
  by iteratively applying the construction in Subsection \ref{subsec:lift}.
Explicitly, since 
 the quotient $\frac{\widetilde F_{i, n}}{\widetilde F_{i, 0}}$ reduces to the constant 1 over $\Bbb C[t]/(t)$,
 we can take $f_{i, n+1}$ to be the logarithm $\log\left(\frac{\widetilde F_{i, n}}{\widetilde F_{i, 0}}\right)$, considered over
 $\Bbb C[t]/(t^{n+2})$.
Similarly, we define $f_{j, n+1}$.

\begin{defin}\label{def:f()}
Let 
\[
f_{i}(n+1),\;\; f_{j}(n+1)
\]
 be the coefficients of $t^{n+1}$ in 
 $f_{i, n+1}$ and $f_{j, n+1}$, respectively.
\end{defin}
Note that these functions
 can naturally be regarded
 as meromorphic functions on $\mathscr W_{i, 0, V_1V_2}^s$ and $\mathscr W_{j, 0, V_1V_2}^s$. 

We aim to establish a relationship between the obstruction cocycle $\{\nu_{ij, n+1}\}$ and the functions 
 $\{f_{i, n+1}\}$.
When neither $\mathscr U_i$ nor $\mathscr U_j$ contains a node, we can apply the calculation in 
 \cite[Lemma 14]{Nishinou_Obstruction_2018} directly.
See Lemma \ref{lem:residue} below. 
Assume that $\mathscr U_i$ contains a node $p$.
We have an edge of $\Gamma$ corresponding to the node $p$, and let $v_1$ and $v_2$ be the ends of that edge.
Recall that we take at least one of $v_1$ and $v_2$ to be a 2-valent vertex.
For the present, we assume $v_1$ is a 3-valent vertex, and $v_2$ is a 2-valent vertex.
Using the previous notation, there is a covering 
 $\mathcal Y_{V_1V_2}^s\to \mathcal Y_{v_1v_2}$, 
 where $\mathcal Y_{V_1V_2}^s$ is isomorphic to $\Spec\Bbb C[X, Y, Z^{\pm 1}, t]/(XY-t)$.
Recall that the varieties of the form $\mathcal Y_{VV'}^s$ are associated with the standard tropical curve with two 
 vertices (see Definition \ref{def:std}), whose unique bounded edge has lattice length one.
Consequently, the exponent of the parameter $t$ is one.
We assume that the component of $C_0$ corresponding to $v_1$ maps to
 the irreducible component of $Y_{0, V_1V_2}^s$ (the central fiber of $\mathcal Y_{V_1V_2}^s$)
 given by $Y = 0$.
By definition of $\nu_{ij, n+1}$, we have
\begin{equation}\label{eq:star}
t^{n+1}\nu_{ij, n+1} = \widetilde F_{i, 0}\exp(f_{i, n+1}) - G_{ij, n}\widetilde F_{j, 0}\exp(f_{j, n+1}),\;\;\;\;\text{mod $t^{n+2}$}.
\end{equation}
From the definition of the standard tropical curve with two vertices (see Subsection \ref{subsec:std}), 
 we can take the equation $\widetilde F_{i, 0} = 0$ in the form 
\[
X+bZ+c = 0,
\]
 where $b$ and $c$ are non-zero complex numbers.

Recall that $\mathscr U_j$ is an open subset of $C_0$ which intersects $\mathscr U_i$.
Note that we take $\mathscr U_j$ so that it does not contain a node of $C_0$. 
First, assume that the image of $\mathscr U_j$
 is contained in the irreducible component of $Y_{0, V_1V_2}^s$
 given by $Y = 0$.
Dividing \cref{eq:star} by $\widetilde F_{i, 0}\exp(f_{j, n+1})$, we obtain
\begin{equation}\label{eq:Nu1}
\frac{t^{n+1}\nu_{ij, n+1}}{\widetilde F_{i, 0}}
= \exp(f_{i, n+1}-f_{j, n+1}) - G_{ij, n}\frac{\widetilde F_{j, 0}}{\widetilde F_{i, 0}},\;\;\;\;\text{mod $t^{n+2}$}
\end{equation}
on $\mathscr W_{i, n+1, V_1V_2}^s\cap \mathscr W_{j, n+1, V_1V_2}^s$.
Note that since the function $\exp(f_{j, n+1})$ reduces to the constant 1 over $\Bbb C[t]/(t)$,
 dividing by it does not
affect the left hand side, as the equation is defined over $\Bbb C[t]/(t^{n+2})$.
%Also note that $t^{n+1}\frac{\nu_{ij, n+1}}{F_{i, 1}} = t^{n+1}\frac{\nu_{ij, n+1}}{F_{i, 0}}$.
Moreover,
 the function $G_{ij, n}\frac{\widetilde F_{j, 0}}{\widetilde F_{i, 0}}$ is a holomorphic function on 
 $\mathscr W_{i, n+1, V_1V_2}^s\cap \mathscr W_{j, n+1, V_1V_2}^s$.
In this case, we can again apply the calculation in \cite[Lemma 14]{Nishinou_Obstruction_2018}, 
 and we deduce Lemma \ref{lem:residue} below. 
% $f_{i}(n+1)-f_{j}(n+1) = \frac{\nu_{n+1, ij}}{F_{i, 0}} + \kappa$
% 	holds on $W_{i, 0, V_1V_2}^s\cap W_{j, 0, V_1V_2}^s$, 
% 	where $\kappa$ is a holomorphic function.

Next, assume that $\mathscr U_j$ is contained in the irreducible component of $Y_{0, V_1V_2}^s$
 given by $X = 0$.
In this case, the image of $\mathscr U_j$ by $\psi_{j, 0}$ is defined by the equation 
 $\widetilde F_{j, 0} = 0$, which we fixed to define $\nu_{ij, n+1}$.
The function $\widetilde F_{j, 0}$ can be expressed as the product of a non-vanishing holomorphic function
 and another function $\widetilde F_{j, 0, Y}$, where
\[
\widetilde F_{j, 0, Y} = bZ+c.
\]
Here, $b$ and $c$ are non-zero complex numbers.
%Note that while $\widetilde F_{j, 0}$ is defined only on $\mathscr W_{j, n+1, V_1V_2}^s$, 
% the function $\widetilde F_{j, 0, Y}$ is defined on $\mathcal Y_{V_1V_2}^s$.
\personal{This point will be used in Lemma \ref{lem:int1}}

Dividing the equation (\ref{eq:star}) by $\widetilde F_{j, 0, Y}\exp(f_{j, n+1})$, we obtain
\begin{equation}\label{eq:Nu2}
\frac{t^{n+1}\nu_{ij, n+1}}{\widetilde F_{j, 0, Y}}
= \frac{\widetilde F_{i, 0}}{\widetilde F_{j, 0, Y}}\exp(f_{i, n+1}-f_{j, n+1}) 
    - G_{ij, n}\frac{\widetilde F_{j, 0}}{\widetilde F_{j, 0, Y}}.
\end{equation}
 on $\mathscr W_{i, n+1, V_1V_2}^s\cap \mathscr W_{j, n+1, V_1V_2}^s$.
Since $\mathscr W_{j, n+1, V_1V_2}^s$ is contained in the component of $Y_{0, V_1V_2}^s$
 given by $X = 0$, we substitute $\frac{t}{Y}$ to the variable $X$.
Note that over $\Bbb C[t]/(t^{n+2})$, we have
\[
\frac{t^{n+1}\nu_{ij, n+1}}{\widetilde F_{j, 0, Y}} = \frac{t^{n+1}\nu_{ij, n+1}}{\widetilde F_{i, 0}},
\]
 after substituting $X = \frac{t}{Y}$.

Now, we have
\[
\frac{\widetilde F_{i, 0}}{\widetilde F_{j, 0, Y}} = 1+\frac{X}{\widetilde F_{i, 0, Y}} = 1+\frac{t}{Y\widetilde F_{j, 0, Y}}.
\]
Furthermore, 
\[
1+\frac{t}{Y\widetilde F_{j, 0, Y}} = \exp(\log(1+\frac{t}{Y\widetilde F_{j, 0, Y}}))
= \exp(\sum_{l=1}^{n+1}\frac{(-1)^{l+1}}{l}(\frac{t}{Y\widetilde F_{j, 0, Y}})^l).
\]
Therefore, we have
\[
\frac{\widetilde F_{i, 0}}{\widetilde F_{j, 0, Y}}\exp(f_{i, n+1}-f_{j, n+1}) 
= \exp(\overline f_{i, n+1} - f_{j, n+1}),
\]
where 
\[
\overline f_{i, n+1} = f_{i, n+1} + \sum_{l=1}^{n+1}\frac{(-1)^{l+1}}{l}(\frac{t}{Y\widetilde F_{j, 0, Y}})^l.
\]

Recall that the function $f_{i, n+1}$ belongs to $\Bbb C[X, Y, Z^{\pm 1}, t, \frac{1}{\widetilde F_{i, 0}}]/(XY-t, t^{n+2})$.
Substituting $X = \frac{t}{Y}$ and expanding $\frac{1}{\widetilde F_{i, 0}}$ as
\[
\frac{1}{\widetilde F_{i, 0}} = \frac{1}{\widetilde F_{j, 0, Y}}\frac{1}{1+\frac{t}{Y\widetilde F_{j, 0, Y}}}
= \frac{1}{\widetilde F_{j, 0, Y}}\sum_{l=0}^{n+1}(-1)^l(\frac{t}{Y\widetilde F_{j, 0, Y}})^l,
\]
we write 
\[
\overline f_{i, n+1} = \sum_{l=1}^{n+1}t^l\widetilde f_{i}(l), 
\]
where $\widetilde f_i(l)$ belongs to $\Bbb C[Y^{\pm 1}, Z^{\pm 1}, \frac{1}{\widetilde F_{j, 0, Y}}]$.
%Note that while the function $\widetilde F_{j, 0}$ is defined only on $\mathscr W_{j, n+1, V_1V_2}^s$, 
% the function $\widetilde F_{j, 0, Y}$ is defined on $Y_{0, V_1V_2}^s$.
%Therefore, the function $\overline f_{i, n+1}$ is defined on the same region 
% $\mathscr W_{i, n+1, V_1V_2}^s$ as $f_{i, n+1}$.
Using these expressions, \cref{eq:Nu2} is written as
\begin{equation}\label{eq:Nu3}
\frac{t^{n+1}\nu_{ij, n+1}}{\widetilde F_{i, 0}}
 = \exp(\overline f_{i, n+1} - f_{j, n+1}) - G_{ij, n}\frac{\widetilde F_{j, 0}}{\widetilde F_{j, 0, Y}}.
\end{equation}
After rewriting \cref{eq:star} in this form, we can once again apply the 
 argument of \cite[Lemma 14]{Nishinou_Obstruction_2018}, and we obtain the following result.
We continue to use the notation introduced above.
\begin{lem}\label{lem:residue}
\begin{enumerate}
\item
	When neither $\mathscr U_i$ nor $\mathscr U_j$ contains a node,
	the equation
	\[
	f_{i}(n+1)-f_{j}(n+1) = \frac{\nu_{ij, n+1}}{\widetilde F_{i, 0}} + \kappa
	\]
	holds on $W_{i, 0, V_1V_2}^s\cap W_{j, 0, V_1V_2}^s$. 
	Here, $f_i(n+1)$ and $f_j(n+1)$ are functions on $\mathscr W_{i, 0, V_1V_2}^s$ and $\mathscr W_{j, 0, V_1V_2}^s$, 
	introduced in Definition \ref{def:f()}, and
	$\kappa$ is a holomorphic function.\\
	
\item
When $\mathscr U_i$ contains a node and $\mathscr U_j$ is contained in 
	the irreducible component of $Y_{0, V_1V_2}^s$
	 given by $Y = 0$, we again obtain the equation
	 \[
	 	f_{i}(n+1)-f_{j}(n+1) = \frac{\nu_{ij, n+1}}{\widetilde F_{i, 0}} + \kappa
	 	\]
	 	on $\mathscr W_{i, X, V_1V_2}^s\cap \mathscr W_{j, 0, V_1V_2}^s$.
	 	Here, $f_i(n+1)$ and $f_j(n+1)$ are functions on $\mathscr W_{i, X, V_1V_2}^s$ and 
	 	$\mathscr W_{j, 0, V_1V_2}^s$, respectively, 
	 	and $\mathscr W_{i, X, V_1V_2}^s$ is the irreducible component of $\mathscr W_{i, 0, V_1V_2}^s$ given by $Y = 0$.\\
	
\item
Finally, when $\mathscr U_i$ contains a node and $\mathscr U_j$ is contained in 
	the irreducible component of $Y_{0, V_1V_2}^s$
	 given by $X = 0$,
	the equation
%\[
%f_{i, X}(n+1)-f_{j}(n+1) = \frac{\nu_{n+1, ij}}{F_{i, 0, X}} + \kappa
%\]
\[
\widetilde f_{i}(n+1)-f_{j}(n+1) = \frac{\nu_{ij, n+1}}{\widetilde F_{i, 0}} + \kappa
\]
 holds on $\mathscr W_{i, Y, V_1V_2}^s\cap \mathscr W_{j, 0, V_1V_2}^s$.
Here, $\widetilde f_i(n+1)$ and $f_j(n+1)$ are functions on $\mathscr W_{i, Y, V_1V_2}^s$ 
 and $\mathscr W_{j, 0, V_1V_2}^s$, respectively, 
 and $\mathscr W_{i, Y, V_1V_2}^s$ is the irreducible component of $\mathscr W_{i, 0, V_1V_2}^s$ given by $X = 0$.
\end{enumerate}
% $f_i(n+1)$ is the coefficient of $t^{n+1}$ of 
% $f_{i, n+1} + \sum_{l=1}^{n+1}\frac{(-1)^{l+1}}{l}(\frac{at}{XF_{i, 0, X}})^i$. 
\end{lem}
\proof
This was proved in \cite[Lemma 14]{Nishinou_Obstruction_2018}
 in the case of deformations of maps from curves into a fixed regular surface.
However, if we have expressions as Eqs.(\ref{eq:Nu1}) and (\ref{eq:Nu3}), 
 the proof is a simple calculation based solely on these equations, independent of any geometric data.
For the reader's convenience, we reproduce the proof.

We deal with the case of \cref{eq:Nu1}.
The case of \cref{eq:Nu3} follows by an entirely analogous argument.
Since the term $G_{ij, n}\frac{\widetilde F_{j, 0}}{\widetilde F_{i, 0}}$  
 is holomorphic, the divergent terms of 
 $\exp(f_{i, n+1}-f_{j, n+1})$ coincide with those of
 $\frac{t^{n+1}\nu_{ij, n+1}}{\widetilde F_{i, 0}}$.
It follows that the function $f_{i, n+1}-f_{j, n+1}$ has no divergent term when reduced over
 $\Bbb C[t]/(t^{n+1})$.
Thus, the divergent terms of $f_{i, n+1}-f_{j, n+1}$ coincide with those of 
 $\frac{t^{n+1}\nu_{ij, n+1}}{\widetilde F_{i, 0}}$.\qed\\

When both vertices $v_1$ and $v_2$ are 2-valent, a similar (and simpler) claim holds.
In particular, in this case we do not need to introduce $\widetilde f_i(n+1)$, and the equation in the third case
 is also $f_{i}(n+1)-f_{j}(n+1) = \frac{\nu_{ij, n+1}}{\widetilde F_{i, 0}} + \kappa$.

\subsubsection{Poincar\'e residue as dual pairing}
%Recall that we need to calculate the pairing $\langle\psi, \xi'\rangle$ in Lemma \ref{lem:pairxi}.
%Since it is difficult to obtain the representative $\xi' = \{\xi_k'\}$ of the class 
% $\widetilde{\nu}_{ij, n+1}$ directly, we look for a suitable substitute.
% 
Let $\Psi_{X_0}$ be the logarithmic volume form on $X_0$ introduced in Definition \ref{def:volumeform}.
The Poincar\'e residue of the 2-form 
\[
\frac{\nu_{ij, n+1}}{\widetilde F_{i, 0}}P_{\Phi}^*\Psi_{X_0}
\]
 along 
\[
\{\widetilde F_{i, 0} = 0\}\cap \mathscr W_{i, 0, V_1V_2}^s\cap \mathscr W_{j, 0, V_1V_2}^s
\] 
% (or $\{F_{i, 0, X} = 0\}\cap W_{i, 0, V_1V_2}^s\cap W_{j, 0, v_1v_2}^s$ if $\mathscr U_i$ contains a node, 
% we omit this case classification afterwards, if no confusion could occur)
 is defined to be the pullback to $\mathscr U_i\cap \mathscr U_j$
 of the 1-form $\zeta_{ij, n+1}$ on $\mathscr W_{i, 0, V_1V_2}^s\cap \mathscr W_{j, 0, V_1V_2}^s$
 determined by the equation
\begin{equation}\label{eq:4}
\frac{\nu_{ij, n+1}}{\widetilde F_{i, 0}}P_{\Phi}^*\Psi_{X_0} 
 = \zeta_{ij, n+1}\wedge \frac{d\widetilde F_{i, 0}}{\widetilde F_{i, 0}}.
\end{equation}
Here, $P_{\Phi}$ is the branched covering introduced in Definition \ref{def:Pphi}.
%\pmb{$\psi_{X_0}$ should be pulled back}
From this definition, it is clear that the pullback of $\zeta_{ij, n+1}$ to $\mathscr U_i\cap \mathscr U_j$
 coincides with the fiberwise pairing between 
 $\nu_{ij, n+1}$ and 
\[
\Psi_{C_0}|_{\mathscr U_i\cap \mathscr U_j} = \psi_{0, V_1V_2}^*P_{\Phi}^*\Psi_{X_0}.
\]
Here, recall that the pullback of $\nu_{ij, n+1}$ to $C_0$ is naturally regarded as a local section 
 of the log normal sheaf $\mathcal N_{\varphi_0}$ on $C_0$ by Lemma \ref{lem:NN_phi}.
Since 
\[
\psi_{0, V_1V_2}^*P_{\Phi}^*\Omega_{X_0}^2 \cong \mathcal N_{\varphi_0}^{\vee}\otimes\omega_{C_0},
\]
 where $\Omega_{X_0}^2$ is the sheaf of logarithmic 2-forms on $X_0$, it follows that
 $\zeta_{ij, n+1}$ naturally defines a local section of the dualizing sheaf $\omega_{C_0}$ of $C_0$.

Lemma \ref{lem:cocycle} implies that the set of sections $\{\zeta_{ij, n+1}\}$ forms a \v{C}ech 1-cocycle on $C_0$
 with values in $\omega_{C_0}$. 
%Note that since the intersection $\mathscr U_i\cap U_j$ does not contain singular points of $C_0$, 
% the sections $\zeta_{ij, n+1}$ are holomorphic.
The cohomology class of $\{\zeta_{ij, n+1}\}$ is given by 
 the pairing 
\[
\Psi_{C_0}\otimes \{\nu_{ij, n+1}\} \in
  H^0(C_0, \mathcal N_{\varphi_0}^{\vee}\otimes \omega_{C_0})\otimes H^1(C_0, \mathcal N_{\varphi_0})\to 
 H^1(C_0, \omega_{C_0}),
\]
 under the identification of Lemma \ref{lem:NN_phi}.
Then, by Propositions \ref{prop:dimension_of_obstruction}
 and \ref{prop:locallifts}, we obtain the following result.
\begin{prop}\label{prop:obstrep}
The vanishing of the cohomology class of $\{\zeta_{ij, n+1}\}$ 
 is equivalent to the vanishing of the obstruction to deforming $\varphi_n$.\qed
\end{prop}
Let $C_{0, v}$ be any irreducible component of $C_0$.
The restriction of $\{\zeta_{ij, n+1}\}$ to $C_{0, v}$ gives an $\omega_{C_{0, v}}$-valued \v{C}ech 1-cocycle on it
 with respect to the open covering induced from the covering $\{\mathscr U_i\}$ of $C_0$.
\begin{rem}
We regard the restriction of $\{\zeta_{ij, n+1}\}$ to $C_{0, v}$ as an $\omega_{C_{0, v}}$-valued \v{C}ech 1-cocycle,
 rather than an $\omega_{C_0}|_{C_{0, v}}$-valued one.
Here, $\omega_{C_{0, v}}$ is the canonical sheaf of $C_{0, v}$, while $\omega_{C_0}|_{C_{0, v}}$
 is the restriction of the dualizing sheaf of $C_0$ to $C_{0, v}$.
The reason is that we have $H^1(C_{0, v}, \omega_{C_{0, v}}) \cong \Bbb C$, which is isomorphic to 
 $H^1(C_0, \omega_{C_0})$, whereas $H^1(C_{0, v}, \omega_{C_0}|_{C_{0, v}}) \cong 0$.
Using this fact, one can see that the restriction of the calculation 
 for $C_0$ and $\mathcal L = \omega_{C_0}$ in Proposition \ref{prop:coupling} to $C_{0, v}$
 gives the same result as the calculation for $C_{0, v}$ and $\mathcal L = \omega_{C_{0, v}}$.
This observation is essential for comparing the cohomology classes
 $[\eta_{ij, n+1, \delta}]$ and $[\zeta_{ij, n+1}]$ discussed in the next subsection.
\end{rem}
By applying Remark \ref{rem:coupling} (3) to the case $\mathcal L = \omega_{C_0}$,
 such a cocycle can be represented by 
 the differences of local meromorphic sections 
\[
\{\zeta_{i, n+1}\}
\]
 on $C_0$.
Moreover, the cohomology class defined by this \v{C}ech 1-cocycle depends only on the 
 residues of $\zeta_{i, n+1}$.
In particular, if such a representative of the class $[\zeta_{ij, n+1}]$ is known, 
 it becomes straightforward to determine whether $[\zeta_{ij, n+1}]$ vanishes or not.

\subsubsection{Integration along fibers and $C^{\infty}$ representative of the obstruction class}\label{subsec:fiberintegral}
In our setting, there is no immediate candidate for the meromorphic sections $\zeta_{i, n+1}$
 described above.
Instead, 
 we have a meromorphic closed 2-form, given by either
\[
f_{i}(n+1)P_{\Phi}^*\Psi_{X_0}\;\; \text{or}\;\; \widetilde f_i(n+1)P_{\Phi}^*\Psi_{X_0}
\]
 defined on each open subset
 $\mathscr W_{i, 0, V_1V_2}^s$, $\mathscr W_{i, X, V_1V_2}^s$ or $\mathscr W_{i, Y, V_1V_2}^s$.
Namely, by \cref{eq:4} and Proposition \ref{prop:obstrep}, 
 the 2-form 
\[
\frac{\nu_{ij, n+1}}{\widetilde F_{i, 0}}P_{\Phi}^*\Psi_{X_0}
\]
 essentially represents the 
 obstruction $\{\zeta_{ij, n+1}\}$.
In this framework, Lemma \ref{lem:residue} suggests that the 
 closed 2-forms $f_{i}(n+1)P_{\Phi}^*\Psi_{X_0}$ and 
 $\widetilde f_i(n+1)P_{\Phi}^*\Psi_{X_0}$ serves as analogues of $\{\zeta_{i, n+1}\}$.
In this subsection, we make this analogy precise, and construct a $C^{\infty}$ \v{C}ech 1-cocycle
 $\{\eta_{ij, n+1}\}$ on $C_0$ from $f_{i}(n+1)P_{\Phi}^*\Psi_{X_0}$ and $\widetilde f_i(n+1)P_{\Phi}^*\Psi_{X_0}$,
 such that the vanishing of its cohomology class
 is equivalent to that of $\{\zeta_{ij, n+1}\}$.
%Note that the image of the map $\psi_{0, V_1V_2}$ is a union of linear curves, as in Subsection \ref{subsec:firstorder}.
%\pmb{$\psi_{0, V_1V_2}$ should be $\psi_{0, s}$?}

Consider the irreducible component $\overline Y_{X, V_1V_2}^s$ of $\overline Y_{0, V_1V_2}^s$ defined by $Y = 0$.
Here, $\overline Y_{0, V_1V_2}^s$ is the central fiber of $\overline{\mathcal Y}_{V_1V_2}^s$,
 which is the closure of $Y_{0, V_1V_2}^s\cong \Spec\Bbb C[X, Y, Z^{\pm 1}]/(XY)$ in $\overline{\mathcal Y}_{V_1V_2}^s$.
Let $\psi_{0, V_1V_2}(C_{0, v})$ be the irreducible component of the image of $\psi_{0, V_1V_2}$
contained in $\overline Y_{X, V_1V_2}^s$.
Fix a Riemannian metric on $\overline Y_{X, V_1V_2}^s$.
Let $\delta$ be a small positive number.
By the exponential map, we can identify the disc bundle of radius $\delta$ in the normal bundle 
 of $\psi_{0, V_1V_2}(C_{0, v})$ in $\overline Y_{X, V_1V_2}^s$ with
 a tubular neighborhood 
\[
N_{\delta}(\psi_{0, V_1V_2}(C_{0, v}))
\]
 of $\psi_{0, V_1V_2}(C_{0, v})$.
The boundary $\partial N_{\delta}(\psi_{0, V_1V_2}(C_{0, v}))$ has a structure of a circle bundle over 
 $C_{0, v}$, and let
\[
\pi_{\delta}\colon \partial N_{\delta}(\psi_{0, V_1V_2}(C_{0, v}))\to C_{0, v}
\]
 be the projection.

Assume $\mathscr W_{i, 0, V_1V_2}^s$ is contained in $\overline Y_{X, V_1V_2}^s$.
In this case, the open subset $\mathscr U_i$ of $C_0$, whose image by $\psi_{0, V_1V_2}$ is contained
 in $\mathscr W_{i, 0, V_1V_2}^s$, does not contain a node.
Then, we can integrate the closed 2-form
 $f_{i}(n+1)P_{\Phi}^*\Psi_{X_0}$ along the fibers of $\pi_{\delta}$.
Let 
\[
\eta_{i, \delta} = \frac{1}{2\pi i}\int_{\delta} f_{i}(n+1)P_{\Phi}^*\Psi_{X_0}
\] 
 denote the integration along the fibers.
This gives a closed 1-form on $\mathscr U_i$.

Next, consider the case when $\mathscr U_i$ contains a node $p$ of $C_0$.
In this case, $\mathscr U_i$ consists of two irreducible components, which we denote
 by 
\begin{equation}\label{eq:scrUXY}
\mathscr U_{i, X}\,\,\,\, \text{and}\,\,\,\, \mathscr U_{i, Y}.
\end{equation}
Each of them is isomorphic to a disc.
Correspondingly, the open subset $\mathscr W_{i, 0, V_1V_2}^s$ of $\overline Y_{0, V_1V_2}^s$
 also has two irreducible components.
We denote these by $\mathscr W_{i, X, V_1V_2}^s$ and $\mathscr W_{i, Y, V_1V_2}^s$, where
\begin{equation}\label{eq:scrWYX}
\mathscr W_{i, X, V_1V_2}^s\subset \overline Y_{X, V_1V_2}^s
\end{equation}
 and 
\begin{equation}\label{eq:scrWYY}
\mathscr W_{i, Y, V_1V_2}^s\subset \overline Y_{Y, V_1V_2}^s.
\end{equation}
Here $\overline Y_{X, V_1V_2}^s$ is defined by $Y=0$, and $\overline Y_{Y, V_1V_2}^s$ is defined by $X=0$.
Moreover, the components $\mathscr U_{i, X}$ and $\mathscr U_{i, Y}$ are mapped to 
 $\mathscr W_{i, X, V_1V_2}^s$ and $\mathscr W_{i, Y, V_1V_2}^s$
 by $\psi_{0, V_1V_2}$, respectively.

Let $\mathring{\mathscr U}_{i, X}$ be the complement of a small closed disc $B_p$ around $p$ in 
 $\mathscr U_{i, X}$.
We can also integrate the closed 2-form $f_i(n+1)P_{\Phi}^*\Psi_{X_0}$ or $\widetilde f_i(n+1)P_{\Phi}^*\Psi_{X_0}$
 along the fibers of $\pi_{\delta}^{-1}(\mathring{\mathscr U}_{i, X})$,
 since these forms do not diverge there. \personal{i.e., the locus $X = 0$ and $\widetilde F_{i,0} = 0$}
As before, we denote the resulting integral by $\eta_{i, \delta}$.
This is a closed 1-form on $\mathring{\mathscr U}_{i, X}$. 

The differences of these $C^{\infty}$ closed 1-forms 
\[
\eta_{ij, n+1, \delta} = \eta_{i, \delta}-\eta_{j, \delta}
\]
 on $\mathscr U_i\cap \mathscr U_j$
 give a \v{C}ech 1-cocycle on each irreducible component of $C_0$ with values in closed 1-forms.
Moreover, on each irreducible component $C_{0, v}\cong \Bbb P^1$ of $C_0$, 
 we have an isomorphism
\[
H^1(C_{0, v}, \mathcal C^1)\cong H^2(C_{0, v}, \Bbb C)\cong H^1(C_{0, v}, \omega_{C_{0, v}}),
\]
 where $\mathcal C^1$ is the sheaf of $C^{\infty}$ closed 1-forms and
 $\omega_{C_{0, v}}$ is the canonical sheaf.
\personal{Since its difficult to consider duality of $C^{\infty}$-forms on a nodal curve, we restrict to each component}

We now review the argument in \cite[Section 3.3]{Nishinou_Obstruction_2018} 
 comparing the class $[\eta_{ij, n+1, \delta}]$ with the obstruction class  
 $[\zeta_{ij, n+1}]$.
See \cite[Section 3.3]{Nishinou_Obstruction_2018} for full details.
First, the sections $\eta_{i, \delta}$ can be regarded as an 
 analogue of the representation of classes in $H^1(C_0, \omega_{C_0})$
 by meromorphic sections on locally closed subsets, as discussed in Subsection \ref{subsec:pairing}.
In particular, the cohomology class of $\{\eta_{ij, n+1, \delta}\}$ depends only on the 
 value of 
\[
\int_{\gamma_p}\eta_{i, \delta},
\]
 where $\gamma_p$ is a contour encircling $B_p$
 in the positive direction with respect to the complex structure, for each node $p$ of $C_0$.
The value of this contour integral is an analogue of the residues of the meromorphic sections  
 discussed in Subsection \ref{subsec:pairing}.
In particular, the following analogue of Proposition \ref{prop:coupling} holds.
\begin{prop}
The cohomology class defined by $\eta_{i, \delta}$ on each component $C_{0, v}$ can be identified with 
 the sum of the values of the contour integrals around the nodes of $C_0$ contained in $C_{0, v}$ 
 under the identification 
\[
H^1(C_{0, v}, \mathcal C^1)\cong H^2(C_{0, v}, \Bbb C)
 \cong H^0(C_{0, v}, \Bbb C)^{\vee}\cong \Bbb C.
\]\qed
\end{prop}

Note that the contour integral equals the integral of the closed 2-form $f_i(n+1)P_{\Phi}^*\Psi_{X_0}$ or
 $\widetilde f_i(n+1)P_{\Phi}^*\Psi_{X_0}$ along the torus $\pi_{\delta}^{-1}(\gamma_p)$ with an appropriate orientation.
By the Stokes theorem, this integral does not depend on $\delta$.
Therefore, for any $\delta$, the cohomology class defined by $\eta_{i, \delta}$ remains unchanged.

On the other hand, on the intersection $\mathscr U_i\cap \mathscr U_j$, 
 the fiberwise integral 
\[
\frac{1}{2\pi i}\int_{\delta}(f_i(n+1)-f_j(n+1))P_{\Phi}^*\Psi_{X_0} 
 = \frac{1}{2\pi i}\int_{\delta}\zeta_{ij, n+1}\wedge\frac{d\widetilde F_{i, 0}}{\widetilde F_{i, 0}}
\]
 or
\[
\frac{1}{2\pi i}\int_{\delta}(\widetilde f_i(n+1)-f_j(n+1))P_{\Phi}^*\Psi_{X_0} 
 = \frac{1}{2\pi i}\int_{\delta}\zeta_{ij, n+1}\wedge\frac{d\widetilde F_{i, 0}}{\widetilde F_{i, 0}}
\]
 converges to $\zeta_{ij, n+1}$ as $\delta$ goes to zero.
Combining this with the observation in the above paragraph, it follows that the cohomology classes 
 defined by $\eta_{ij, n+1, \delta}$ and by $\zeta_{ij, n+1}$ coincide on each $C_{0, v}$.

Note that while $\eta_{ij, n+1, \delta}$ defines a cohomology class on each component of $C_0$, 
 $\zeta_{ij, n+1}$ determines a cohomology class on $C_0$.
However, recalling that the cohomology class determined by $\zeta_{ij, n+1}$ only depends on the
 residues at the nodes, and that the contour integrals of $\eta_{i, \delta}$ correspond to 
 these residues, it follows that $\eta_{i, \delta}$ also canonically determines a cohomology class on $C_0$, 
 which coincides with the cohomology class
 determined by $\zeta_{ij, n+1}$, see \cite[Subsection 3.3.2]{Nishinou_Obstruction_2018}.
Again, analogously to Proposition \ref{prop:coupling}, we obtain the following.
\begin{prop}
The class of $H^1(C_0, \omega_{C_0})$ determined by $\eta_{i, \delta}$
 is identified with the sum of all the contour integrals around the nodes of $C_0$
 under the identification
\[
H^1(C_0, \omega_{C_0})\cong H^0(C_0, \mathcal O_{C_0})^{\vee}\cong \Bbb C.
\]\qed
\end{prop}
Then, by Proposition \ref{prop:obstrep}, we have the following.
\begin{cor}\label{cor:obstsum}
The obstruction to deforming $\varphi_n$ vanishes if and only if 
 the sum of all the contour integrals of $\eta_{i, \delta}$ around the nodes of $C_0$ vanishes.\qed
\end{cor}

Thus, the problem of proving the existence of deformations reduces
 to checking the vanishing of the sum of the contour integrals of $\eta_{i, \delta}$.
Note that at each node of $C_0$, there are two contour integrals associated with it, 
 corresponding to the branches of $C_0$ at the node.
Thus, to verify the vanishing of the obstruction, it suffices to show that the sum of each pair of contour integrals
 associated with each node vanishes.
We will prove this in the rest of this section.

\subsubsection{Vanishing of the sum of integrals}
Recall that the value of the integral along a contour $\gamma_p$ around
 a node is equal to the integral of the closed 2-form $f_i(n+1)P_{\Phi}^*\Psi_{X_0}$ or 
 $\widetilde f_i(n+1)P_{\Phi}^*\Psi_{X_0}$ over the torus $\pi^{-1}_{\delta}(\gamma_p)$.
We first recall the relevant notations. 
The functions $f_i(n+1)$ and $\widetilde f_i(n+1)$ are the coefficients of $t^{n+1}$ in the expansions of the functions 
 $f_{i, n+1}$ and 
\[
\overline f_{i, n+1} = f_{i, n+1}+ \sum_{l=1}^{n+1}\frac{(-1)^{l+1}}{l}(\frac{at}{Y\widetilde F_{j, 0, Y}})^l,
\] 
 respectively.
Here, $f_{i, n+1}$ is determined by the condition that 
\[
\widetilde F_{i, 0}\exp(f_{i, n+1})=0
\]
 is the defining equation of the image of $\psi_{i, n}\colon \mathscr U_{i, n}\to \mathcal Y_{V_1V_2}^s$, 
 and 
\[
\widetilde F_{i, 0} = X+bZ+c
\]
 is a function defining the image of $\psi_{i, 0}$.
Here, $X, Y$ and $Z$ are functions on $\mathcal Y_{V_1V_2}^s\cong \Spec\Bbb C[X, Y, Z^{\pm 1}, t]/(XY-t)$.
Similarly, 
\[
\widetilde F_{j, 0, Y} = bZ+c.
\]
Here, we assume one of the vertices $V_1$ and $V_2$ 
 is 3-valent.
Otherwise, both vertices are 2-valent and we can take $\widetilde F_{i, 0} = bZ+c$,
 and there is no need to consider $\widetilde f_i(n+1)$.
This case is simpler and is omitted.
Moreover, $P_{\Phi}^*\Psi_{X_0}$ 
 is given by 
\[
\frac{dX}{X}\wedge \frac{dZ}{Z} = -\frac{dY}{Y}\wedge \frac{dZ}{Z}
\]
 up to a multiplicative constant.
For notational simplicity, we set this constant to be one.

First, we calculate the contribution from the part 
 $\sum_{l=1}^{n+1}\frac{(-1)^{l+1}}{l}(\frac{at}{Y\widetilde F_{j, 0, Y}})^l$ of $\overline f_{i, n+1}$.
This part is relevant only on the irreducible component of $Y_{0, V_1V_2}^s$ given by $X = 0$, 
 and we assume that the torus $\pi_{\delta}^{-1}(\gamma_p)$ is contained in this component.
\begin{lem}\label{lem:int1}
The integral 
 of the two form 
 $(\frac{at}{Y\widetilde F_{j, 0, Y}})^lP_{\Phi}^*\Psi_{X_0}$
 along the torus $\pi_{\delta}^{-1}(\gamma_p)$
 is zero for any positive integer $l$.
\end{lem}
\proof
We can take 
\[
\widetilde Z = bZ+c
\]
 as one of the coordinates on the component of $Y_{0, V_1V_2}^s$
 given by $X = 0$.
Then, the image of $\psi_{0, V_1V_2}$ coincides with the $Y$-axis in the $(Y, \widetilde Z)$-plane.
Recall that we fix a metric on $Y_{0, V_1V_2}^s$ and take a tubular neighborhood 
$N_{\delta}(\psi_{0, V_1V_2}(\mathscr U_{i, Y}))$ of a part of the image of $\psi_{0, V_1V_2}$,
 where $\mathscr U_{i, Y}$ is defined in Eq.(\ref{eq:scrUXY}).
By taking the metric suitably, we can assume that 
around the point 
\[
(Y, \widetilde Z) = (0, 0),
\]
 $N_{\delta}(\psi_{0, V_1V_2}(\mathscr U_{i, Y}))$ is the product
of the $Y$-axis and a disc on the $\widetilde Z$-axis, and 
the projection 
\[
\pi_{\delta}\colon N_{\delta}(\psi_{0, V_1V_2}(\mathscr U_{i, X}))\to \mathscr U_{i, X}
\]
is compatible with this product structure.

We have 
\[
%\begin{array}{ll}
 (\frac{at}{Y\widetilde F_{i, 0, Y}})^lP_{\Phi}^*\Psi_{X_0}  = 
 -(\frac{at}{Y\widetilde Z})^l \frac{dY}{Y}\wedge \frac{d\widetilde Z}{\widetilde Z-c}%\\
%  & = -\frac{ba^lt^l}{X^{l+1}\widetilde Z^l}\frac{1}{X+c}\frac{dX}{X}\wedge \frac{d\widetilde Z}{1-\frac{\widetilde Z}{X+c}}\\
%  & = -\frac{ba^lt^l}{X^{l+1}\widetilde Z^l(X+c)}\sum_{m=0}^{\infty}(\frac{\widetilde Z}{X+c})^mdX\wedge d\widetilde Z. 
%\end{array}
\]
Since we have $l\geq 1$, this contains only terms with poles of order at least two with respect to $Y$, 
 so the integration with respect to $Y$ vanishes.
This proves the claim.\qed\\

\begin{comment}
\[
\begin{array}{ll}
(\frac{t}{Y(aY+bZ+c)})^l\frac{dY}{Y}\wedge\frac{dZ}{Z} 
 & = (\frac{t}{Y(aY+bZ+c)})^l\frac{dY}{Y}\wedge\frac{d\widetilde Z}{\frac{1}{b}(\widetilde Z-aY-c)} \\
 & = \frac{bt^l}{Y^l\widetilde Z^l}\frac{dY}{Y}\wedge \frac{d\widetilde Z}{(\widetilde Z-aY-c)} \\
 & = \frac{bt^l}{Y^l\widetilde Z^l}\frac{dY}{Y}\wedge (-\frac{1}{aY+c})\frac{d\widetilde Z}{1-\frac{\widetilde Z}{aY+c}}\\
 & = -\frac{bt^l}{Y^l\widetilde Z^l(aY+c)} \frac{dY}{Y}\wedge \sum_{m=0}^{\infty}(\frac{\widetilde Z}{aY+c})^m.
\end{array}
\]
Here, we put $\widetilde Z = aY+bZ+c$.
Integrating this around $\widetilde Z = 0$, we obtain
\[
-\frac{2\pi bt^l}{Y^{l+1}(aY+c)^l}{dY}
 = -\frac{2\pi bt^l}{Y^{l+1}c^l}\frac{1}{(1+\frac{aY}{c})^l}dY
 = -\frac{2\pi bt^l}{c^lY^{l+1}}\sum_{m=0}^{\infty}(-1)^m
     \begin{pmatrix}
      l+m-1\\ m
     \end{pmatrix}
     (\frac{aY}{c})^m.
\]
Integrating around $Y = 0$, we obtain
 $(-1)^{l+1}\frac{4\pi^2 ba^lt^l}{c^{2l}}\begin{pmatrix}
 2l-1\\ l
 \end{pmatrix}
 $.
\end{comment}

Therefore, we need to show that the sum of the integrals of $f_i(n+1)P_{\Phi}^*\Psi_{C_0}$
 over the two 2-tori corresponding to $\gamma_p$ and $\gamma_p'$ cancels.
Here, $\gamma_p$ and $\gamma_p'$ are two contours around a node $p$ on different branches of $C_0$.
This can also be seen by direct calculation, but there is a much simpler geometric argument
 as follows. 
\begin{lem}\label{lem:int2}
The sum of the integrals of $f_i(n+1)P_{\Phi}^*\Psi_{C_0}$
 over the two 2-tori corresponding to $\gamma_p$ and $\gamma_p'$ is zero.
\end{lem}
\proof
Recall that the function $f_{i, n+1}$, in particular its coefficient $f_i(n+1)$ of $t^{n+1}$, belongs to 
 $\Bbb C[X, Y, Z^{\pm 1}, t, \frac{1}{\widetilde F_{i, 0}}]/(XY-t, t^{n+1})$.
By regarding it as an element in $\Bbb C[X, Y, Z^{\pm 1}, t, \frac{1}{\widetilde F_{i, 0}}]/(XY-t)$
 as in Subsection \ref{subsec:lift}, 
 the 2-form $f_i(n+1)P_{\Phi}^*\Psi_{\mathcal X}$ defines a meromorphic 2-form on $\mathcal Y_{V_1V_2}^s$, 
 which extends $f_i(n+1)P_{\Phi}^*\Psi_{X_0}$ on $Y_{0, V_1V_2}^s$.
In the space $\mathcal Y_{V_1V_2}^s$, the two 2-tori $\pi_{\delta}^{-1}(\gamma_p)$ and
 $(\pi_{\delta}')^{-1}(\gamma'_p)$ are homologous, but their orientations are opposite, as
  the contours $\gamma_p$ and $\gamma_p'$ are oriented oppositely.
Here, 
\[
\pi_{\delta}'\colon N_{\delta}(\psi_{0, V_1V_2}(\mathscr U_{i, Y}))\to \mathscr U_{i, Y}
\]
 is the projection.
Note that $f_i(n+1)P_{\Phi}^*\Psi_{\mathcal X}$
 is a closed 2-form.
Moreover, we can take an oriented 3-manifold $K$ in $\mathcal Y_{V_1V_2}^s$
 with 
\[
\partial K = \pi_{\delta}^{-1}(\gamma_p)\cup (\pi_{\delta}')^{-1}(\gamma'_p)
\]
 such that $f_i(n+1)P_{\Phi}^*\Psi_{\mathcal X}$ does not diverge on $K$. 
Therefore, the claim follows from the Stokes theorem.\qed\\

We finally obtain the following.

\begin{thm}\label{thm:main}
Any pre-log curve $\varphi_0\colon C_0\to X_0$ constructed in Theorem \ref{thm:pre-log}
 can be deformed to a generic fiber of $\mathcal X$.
\end{thm} 
\proof 
By combining Lemmas \ref{lem:int1} and \ref{lem:int2} with Corollary \ref{cor:obstsum},
 we obtain a formal family of deformations of $\varphi_0$.
By applying the implicit function theorem \cite[Proposition 1.5]{Kosarew_1991}, 
 we obtain the claim.\qed\\

Conversely, a family of curves $\mathcal X^{\times}\to B^{\times} = B\setminus\{0\}$
 gives rise to a pre-log curve in a suitable $X_0$, as described in Section \ref{sec:tropicalization}.
This pre-log curve, in turn, corresponds a tropical curve on $S$.
However, the resulting tropical curve may have 4- or higher- valent vertices in general.
To ensure that the resulting tropical curve is 3-valent, we need to impose incidence conditions.
Namely, we take a set of general sections $\{x_1, \dots, x_g\}$ of the fibration
 $\mathcal X\to B$ which gives a set of general integer points $\{p_1, \dots, p_g\}$.
See the next section
 for details.
Then, we have the following.
\begin{prop}\label{prop:rev}
Any family of holomorphic curves on $\mathcal X^{\times}\to B^{\times}$
 which contains the sections $\{x_1, \dots, x_g\}$ restricted to $B^{\times}$,
 gives rise to a 3-valent tropical curve on $S$ satisfying the condition of Theorem \ref{thm:pre-log}.
\end{prop}
\proof
As mentioned above, a family of holomorphic curves induces a tropical curve on $S$.
By the incidence condition, the resulting tropical curve must pass through the points $\{p_1, \dots, p_g\}$.
On the other hand, by
 \cite[Proposition 3.20]{Blomme_Enumeration_of_curves_1},
 if a tropical curve of genus $g$ on $S$ contains a vertex with valence at least 4, 
 then its deformation is at most of dimension $g-1$.
Since the set $\{p_1, \dots, p_g\}$ is general, it cannot occur.
Thus, the resulting tropical curve is 3-valent.
Since this tropical curve is realizable, it satisfies the condition of Theorem \ref{thm:pre-log}.\qed

\section{Multiplicity of enumeration: Proof of Theorem \ref{thm:intro_multiplicity}} \label{sec:multiplicity}

In this section, we study the enumeration of parametrized holomorphic curves that give rise to a fixed realizable parametrized tropical curve as described in Section \ref{sec:tropicalization},
 and we provide a proof of \cref{thm:intro_multiplicity}.

Let 
\[
h\colon\Gamma\to S
\]
 be a 3-valent parametrized tropical curve of genus $g$ passing through $g$ integer points
\[
p_1,\dots,p_g\in S
\]
 which is realizable in the sense of Definition \ref{def:realizable}.
Note that such a tropical curve must be 3-valent by Proposition \ref{prop:rev}.
Assume that the preimages $h\inv(p_i)$ lie in the interior of the edges of $\Gamma$, and that these constraints make the tropical curve rigid (see Remark \ref{rem:expdim}).
Choose a point $v_i$ from each $h\inv(p_i)$
  (see Remark \ref{rem:double_edge} below).

\begin{defin}\label{def:'G}
Let $\Gamma'$ denote the graph $\Gamma$ subdivided by $v_i$.
\end{defin}
Let $\mathcal D$ be the polyhedral subdivision of $S$ in \cref{sec:pre-log}.
After refining $\mathcal D$ if necessary, we can assume that each $p_i = h(v_i)$ is a 0-cell of $\mathcal D$.
The subdivision $\mathcal D$ induces a subdivision of $\Gamma$.
\begin{defin}\label{def:hatG}
Let $\hGamma$ be the subdivision of $\Gamma$ induced by $\mathcal D$.
\end{defin}

\begin{rem}\label{rem:double_edge}
Even if the points $p_1, \dots, p_g$ are generic, we cannot assume that the inverse image 
 $h\inv(p_i)$ is a single point.
See Figure \ref{fig:1}.

\begin{figure}[h]
\includegraphics[height=11cm]{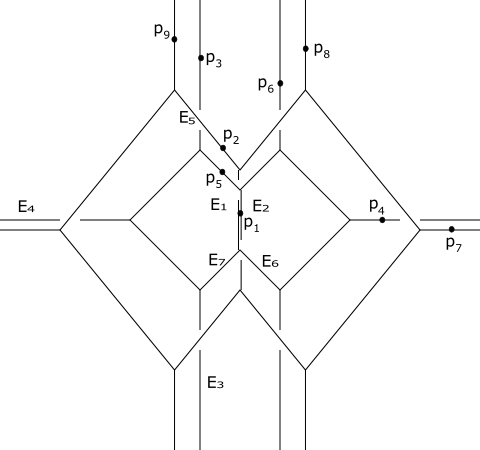}
\caption{A part of a periodic tropical curve in a rectangular fundamental domain.
The dots are the points $p_1, \dots, p_g$, which impose rigidity on the tropical curve.
Note that although the edges $E_1$ and $E_2$ are drawn slightly offset for visual clarity, 
 they actually overlap.
Even if we perturb the points $p_1, \dots, p_g$, this overlap remains unresolved.}\label{fig:1}
\end{figure}
This figure represents a part of a periodic plane tropical curve in a rectangular fundamental domain of the action of 
 $\overline{\Lambda}$.
It gives a tropical curve of genus 9 on $S = \Bbb N_{\Bbb R}/\overline{\Lambda}$,
 which has 24 edges and 16 vertices.
The curve passes through the 9 points $\{p_1, \dots, p_9\}$, which makes it rigid.
Note that the edges $E_1$ and $E_2$ passes through the point $p_1$, so that 
 $h^{-1}(p_1)$ consists of two points.
One can observe that if $p_1$ is perturbed slightly, this property persists, as explained below.
By a similar argument, one can see that this is also true if we perturb other points $p_2, \dots, p_g$.
This suggests that this is a generic phenomenon.

Now, consider shifting $p_1$ slightly to the left, while maintaining the condition that
 the edge $E_1$ continues to pass through $p_1$,
 with all other incidence conditions $\{p_2, \dots, p_9\}$ remaining fixed.
Since the incidence condition $p_2$ must be preserved, we cannot move the edge $E_5$
 except changing its length.
As a result, after shifting $p_1$ and $E_1$, the local configuration near $E_1$ appears as depicted in Figure \ref{fig:2}.

\begin{figure}[h]
\begin{center}
\includegraphics[height=8.9cm]{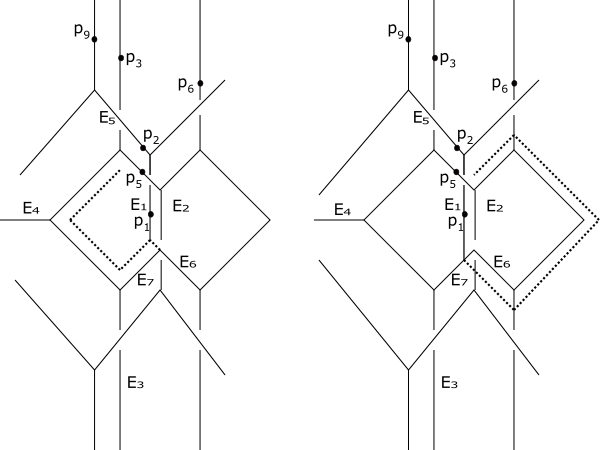}
\end{center}
\caption{A neighborhood of the edge $E_1$ after the shift.
To preserve the incidence condition $p_2$, the edge $E_5$ must remain fixed, except changing its length.
Consequently, in order to move $E_1$, it is necessary to adjust either $E_6$, $E_7$, 
 or both.
}\label{fig:2}
\end{figure}
 
Since $E_5$ must remain fixed, moving $E_1$
 requires shifting either $E_6$, $E_7$, 
 or both.
The picture on the left in Figure \ref{fig:2}
 illustrates the case when the edge $E_7$ moves while $E_6$ remains fixed.
Since the incidence conditions $p_3$ and $p_4$ must also be preserved, we cannot move $E_3$ or $E_4$ except
 changing their lengths 
 (note that $p_4$ is outside the frame, but lies on $E_4$).
Thus, the square-like part containing $E_7$ shrinks, as indicated by the dotted lines in the picture.
However, in this case, the incidence condition $p_5$ cannot be maintained.
Similarly, if both $E_6$ and $E_7$ move, the incidence conditions still cannot be preserved for the same reason.

The picture on the right in Figure \ref{fig:2} illustrates the case when $E_6$ moves while $E_7$ remains fixed.
In this case, to preserve the incidence conditions $p_6$ and $p_7$ (where $p_7$ lies outside the picture), 
 the square-like part containing $E_6$ expands, as indicated by the dotted lines in the picture.
To maintain the incidence condition $p_5$, we need to move $E_2$ to the left, so that 
 the edges $E_1$ and $E_2$ continue to share the point $p_1$,
 ensuring that the original overlapping configuration is preserved.
\end{rem}

We also introduce incidence conditions on $\mathcal X$ as described in the introduction.
One example of such incidence conditions can be constructed following \cite[Section 3]{Nishinou_Toric_2006}.
Namely, let 
\[
x_1, \dots, x_g
\]
 be general points on $\widetilde{\mathcal X}$.
Recall that $h\colon\Gamma\to S$ is the quotient of a $\overline\Lambda$-periodic tropical curve 
\[
\widetilde h\colon \widetilde\Gamma\to N_{\Bbb R}.
\]
Choose inverse images 
\[
\widetilde p_1\dots, \widetilde p_g\in N
\]
 of $p_1,\dots,p_g\in S$.
Each ray in $N_{\Bbb R}\times \Bbb R$ spanned by $(\widetilde p_i, 1)$ gives a one-parameter subgroup 
 acting on $\widetilde{\mathcal X}$.
The intersection of the closure of the orbit of $x_i$ under this group with the central fiber 
 $\widetilde X_0$ of $\widetilde{\mathcal X}$
 is a point in the interior (that is, the complement of the toric divisors) of the component $\widetilde X_{0, \widetilde p_i}$
 (see \cite[Section 3]{Nishinou_Toric_2006}).
The images of these orbit closures in $\mathcal X$ give incidence conditions on $\mathcal X$.
For notational simplicity, we continue to denote these orbit closures by $x_1, \dots, x_g$.
These are sections of the fibration $\mathcal X\to B$.

%Let $\Delta\subset N_{\Bbb R}$ be a parallelogram fundamental domain of the $\overline{\Lambda}$-action
% as in Section \ref{sec:pre-log}.
%Then, the points $p_1, \dots, p_g$ uniquely determine rational points in $\Delta$.
%Each of these points determine a one dimensional subgroup of the torus 

%Let $x_1,\dots,x_g$ be $k$-rational points in $X$, possibly after a finite base field extension, with $\tau(x_i)=p_i$.

We fix an orientation on every edge of $\Gamma$.
We put induced orientations on the edges of the subdivisions.
For each edge $e$ of $\Gamma$, or any of its subdivisions, let $\partial^- e$ and $\partial^+ e$ denote the two endpoints  
 according to the fixed orientation, and let $N_e\subset N$ denote the sublattice generated by the weight vector 
 $w_e$ of $e$.
We introduce the following maps of abelian groups.
\begin{align*}
F\colon\Map(V(\Gamma),N)
&\xrightarrow{\mathmakebox[3ex]{}} \bigoplus_{e \in E(\Gamma)} (N/N_{e})\\
\phi &\xmapsto{\mathmakebox[3ex]{}} \Big(\big(\phi(\partial^+e)-\phi(\partial^-e)\big)_e\Big),\\
G\colon \Map(V(\Gamma'),N)
&\xrightarrow{\mathmakebox[3ex]{}} \bigoplus_{e \in E(\Gamma')} (N/N_{e})\ \oplus\ \bigoplus_{i=1}^g N\\
\phi &\xmapsto{\mathmakebox[3ex]{}}
\Big(\big(\phi(\partial^+e)-\phi(\partial^-e)\big)_e,\big(\phi(v_i)\big)_i\Big),\\
\hF\colon\Map(V(\hGamma),N)
&\xrightarrow{\mathmakebox[3ex]{}} \bigoplus_{e \in E(\hGamma)} (N/N_{e}) \\
\phi &\xmapsto{\mathmakebox[3ex]{}} \Big(\big(\phi(\partial^+e)-\phi(\partial^-e)\big)_e\Big),\\
\hG\colon \Map(V(\hGamma),N)
&\xrightarrow{\mathmakebox[3ex]{}} \bigoplus_{e \in E(\hGamma)} (N/N_{e})\ \oplus\ \bigoplus_{i=1}^g N\\
\phi &\xmapsto{\mathmakebox[3ex]{}}
\Big(\big(\phi(\partial^+e)-\phi(\partial^-e)\big)_e,\big(\phi(v_i)\big)_i\Big).
\end{align*}

For any abelian group $A$, we define
\[F_A \coloneqq F\otimes 1_A\colon\Map(V(\Gamma),N)\otimes A
\xrightarrow{\mathmakebox[3ex]{}} \bigoplus_{e \in E(\Gamma)} (N/N_{e})\otimes A.\]
We define $G_A$, $\hF_A$, and $\hG_A$ similarly.

\personal{For a 3-valent tropical curve $(\Gamma,h)$, the germ of $\Ker b(\Gamma)_\bbR$ at 0 gives a universal infinitesimal deformation.}

\begin{lem}\label{lem:deformation_tropical}
	We have
	\[
	\rk \Ker F = g - 1 + \rk \Coker F,
	\]
	where $\rk$ denotes the rank of abelian group.
\end{lem}
\begin{proof}
	We have
	\begin{align}
	& \rk \Ker F\nonumber\\
	=& \rk\bigg(\bigoplus_{v \in V(\Gamma)} N \bigg) - \rk \bigg( \bigoplus_{e \in E(\Gamma)} (N/N_{e}) \bigg) + \rk \Coker F \nonumber\\
	=& 2\abs{V(\Gamma)} - \abs{E(\Gamma)} + \rk \Coker F \label{eq:dim_trop_1}
	\end{align}
	Since $\Gamma$ is 3-valent, we have
	\[ 3\abs{V(\Gamma)} = 2\abs{E(\Gamma)}.\]
	By computing the Euler characteristic of $\Gamma$, we have
	\[\abs{V(\Gamma)} - \abs{E(\Gamma)} = 1 - g.\]
	Taking the difference of the previous two equations gives
	\[2\abs{V(\Gamma)}-\abs{E(\Gamma)} = g-1.\]
	Substituting it into \cref{eq:dim_trop_1}, we obtain the result.
\end{proof}

\begin{prop}\label{prop:cohomologies}
	Let $\varphi_0\colon C_0\to X_0$ be any pre-log curve associated with the tropical curve $h\colon\hGamma\to S$, and we equip it with a log structure as in \cref{sec:deformation}.
	Let $\mathcal N_{\varphi_0}=\varphi_0^*\Theta_{X_0}/\Theta_{C_0}$ be the log normal sheaf of $\varphi_0$.
	The derived global section $R\Gamma(\mathcal N_{\varphi_0})$ is quasi-isomorphic to the complex
	\[\bigoplus_{v\in V(\Gamma)} N_\bbC \xrightarrow{\ F_\bbC\ } \bigoplus_{e\in E(\Gamma)} (N/N_e)_\bbC\]
	concentrated in degrees 0 and 1.
	In particular, we have
	$H^0(C_0,\mathcal N_{\varphi_0})\simeq\Ker F_\bbC$ and $H^1(C_0,\mathcal N_{\varphi_0})\simeq\Coker F_\bbC$.
\end{prop}
\begin{proof}
	For every 3-valent vertex $v$ of $\hGamma$, we have
	\[\Theta_{C_0}|_{C_{0,v}}=T_{C_{0,v}}\otimes\mathcal O_{C_{0,v}}(-3)=\mathcal O_{C_{0,v}}(2)\otimes\mathcal O_{C_{0,v}}(-3)=\mathcal O_{C_{0,v}}(-1).\]
	\personal{See the computations in \cite[Examples 3.36(6)]{Gross_Tropical_2011}.}
	Since $C_{0,v}$ is rational, $R\Gamma(\mathcal O_{C_{0,v}}(-1))\simeq 0$.
	So, we have
	\[R\Gamma(\mathcal N_{\varphi_0}|_{C_{0,v}})\simeq R\Gamma(\varphi_0^*\Theta_{X_0}|_{C_{0,v}})\simeq N_\bbC.\]
	
	For every 2-valent vertex $v$ of $\hGamma$, we have
	\[\Theta_{C_0}|_{C_{0,v}}=T_{C_{0,v}}\otimes\mathcal O_{C_{0,v}}(-2)=\mathcal O_{C_{0,v}}(2)\otimes\mathcal O_{C_{0,v}}(-2)=\mathcal O_{C_{0,v}}.\]
	Moreover, the map $\Theta_{C_0}|_{C_{0,v}}\to \varphi_0^*\Theta_{X_{0}}|_{C_{0,v}}\simeq N_\bbC\otimes\mathcal O_{C_{0,v}}$ is given by the weight vector of an edge $e$ connected to $v$.
	So, we have
	\[R\Gamma(\mathcal N_{\varphi_0}|_{C_{0,v}})\simeq (N/N_e)_\bbC.\]
	Similarly, for every edge $e$ of $\hGamma$, we have
	\[R\Gamma(\mathcal N_{\varphi_0}|_{p_e})\simeq (N/N_e)_\bbC.\]
	Here, $\mathcal N_{\varphi_0}|_{p_{e}}$ is the fiber of $\mathcal N_{\varphi_0}$ over the node $p_e$ corresponding to the edge $e$.
	
	Therefore, it follows from the following exact triangle
	\[R\Gamma(\mathcal N_{\varphi_0})\xrightarrow{\mathmakebox[3ex]{}} \bigoplus_{v\in V(\hGamma)} R\Gamma(\mathcal N_{\varphi_0}|_{C_{0,v}})\xrightarrow{\mathmakebox[3ex]{}} \bigoplus_{e\in E(\hGamma)}R\Gamma(\mathcal N_{\varphi_0}|_{p_e})\xrightarrow{\mathmakebox[3ex]{+1}}\]
	that $R\Gamma(\mathcal N_{\varphi_0})$ is quasi-isomorphic to the complex
	\[\bigoplus_{v\in V(\Gamma)} N_\bbC \xrightarrow{\ F_\bbC\ } \bigoplus_{e\in E(\Gamma)} (N/N_e)_\bbC\]
	concentrated in degrees 0 and 1.
\end{proof}

\begin{cor} \label{cor:rank}
	We have $\rk \Coker F=1$ and $\rk \Ker F = g$.
\end{cor}
\begin{proof}
	The first equation follows from \cref{prop:dimension_of_obstruction}, Serre duality and \cref{prop:cohomologies}.
	The second equation follows from the first by \cref{lem:deformation_tropical}.
\end{proof}

\begin{rem}
	We can also obtain \cref{cor:rank} purely combinatorially, for example by considering dual spaces, as in the proof of \cref{prop:dimension_of_obstruction}.
\end{rem}

\begin{rem}\label{rem:rigidity}
	Note that the kernel $\Ker F_{\Bbb R}$ (resp.\ $\Ker \hF_{\Bbb R}$) describes
	 the space of infinitesimal deformations of the tropical curve $h\colon\Gamma\to S$ (resp.\ $h\colon\hGamma\to S$).
	Recall that the constraints $h(v_i)=p_i$ make the tropical curve rigid by our assumption.
	Hence, the evaluation map
	\begin{align*}
		\Ker F_\bbR &\longrightarrow \bigoplus_{i=1}^g (N/N_{e_i})_\bbR\\
		\phi&\longmapsto\Big(\big(\phi(\partial^- e_i)\big)_i\Big)
	\end{align*}
	is injective, where $e_i$ denotes the edge of $\Gamma$ containing $v_i$.
	Then, it follows from the dimension counting in \cref{cor:rank} that the map above is, in fact, an isomorphism.
	We deduce that the evaluation map
	\begin{align*}
		\Ker \hF_\bbR &\longrightarrow \bigoplus_{i=1}^g N_\bbR\\
		\phi&\longmapsto\Big(\big(\phi(v_i)\big)_i\Big)
	\end{align*}
	is surjective.
\end{rem}

Recall that we constructed incidence conditions $x_1, \dots, x_g$, which are sections of $\mathcal X\to B$ such that
 the limit $x_i(0)\in X_0$ lies in the smooth locus of $X_0$.
Let 
\[
\PreLog
\]
 denote the set of pre-log curves associated with the tropical curve $h\colon\hGamma\to S$
 with marked points $s_i\in C_{0,v_i}$ for $i=1,\dots,g$.
Let 
\[
\PreLog'\subset\PreLog
\]
 denote the subset consisting of marked pre-log curves $\varphi_0\colon C_0\to X_0$ 
 satisfying the conditions $\varphi_0(s_i)=x_i(0)$ for $i=1,\dots,g$.

\begin{prop} \label{prop:pre-log_torsor}
	The set $\PreLog'$ has a natural structure of a $(\Ker G_{\bbC^*})$-torsor.
	In particular, there exist exactly $\abs{\Ker G_{\bbC^*}}$ isomorphism classes of pre-log curves $\varphi_0\colon C_0\to X_0$ associated with the tropical curve $h\colon\hGamma\to S$ together with marked points $s_i\in C_{0,v_i}$ satisfying $\varphi_0(s_i)=x_i(0)$ for $i=1,\dots,g$.
\end{prop}
\begin{proof}
	To begin, we first prove that the set $\PreLog'$ is non-empty.
	For every vertex $v$ of $\hGamma$, the toric variety $X_{0,h(v)}$ admits an action by $N_{\Bbb C^*}$.
	Now, consider an edge $e\in E(\hGamma)$ with endpoints $v$ and $v'$.
	If the $N_{\bbC^*}$-action on $X_{0,h(v)}$ and the $N_{\bbC^*}$-action on $X_{0,h(v')}$ coincide on the stratum $X_{0,h(e)}=X_{0,h(v)}\cap X_{0,h(v')}$ for every $e$, then, starting 
	from a given pre-log curve associated with the tropical curve $h\colon\hGamma\to S$, 
	we can obtain other pre-log curves via these actions.
	In other words, $\Ker \hF_{\bbC^*}$ acts naturally on the set $\PreLog$.
	
	We identify all the domain curves $C_0$ in the set $\PreLog$.
	For each $i=1,\dots,g$, we choose a closed point $s_i$ in the open stratum $C_{0,v_i}^\circ$ associated with $v_i$.
	Since $v_i$ is 2-valent, the open stratum $C_{0,v_i}^\circ$ is $\bbP^1_\bbC$ minus two points.
	Hence, the choice of $s_i\in C_{0,v_i}^\circ$ is unique up to isomorphisms.
	Now, evaluating at the points $s_i$ gives an evaluation map
	\[\PreLog\longrightarrow\prod_{i=1}^g X^\circ_{0,h(v_i)}.\]
	Note that the group $\Map(V(\hGamma),N)\otimes\bbC^*$ naturally acts on $\prod_{i=1}^g X^\circ_{0,h(v_i)}$.
	In particular, $\Ker\hF_{\bbC^*}$ also acts on $\prod_{i=1}^g X^\circ_{0,h(v_i)}$, 
	and it commutes with the evaluation map.	
	Here, both $\Ker\hF_{\bbC^*}$ and $\prod_{i=1}^g X^\circ_{0,h(v_i)}$ are isomorphic to algebraic tori.
	The derivative of the $\Ker\hF_{\bbC^*}$ action on $\prod_{i=1}^g X^\circ_{0,h(v_i)}$ is equal to the map
	\begin{align*}
	\Ker\hF_\bbC &\longrightarrow \bigoplus_{i=1}^g N_\bbC\\
	\phi&\longmapsto\Big(\big(\phi(v_i)\big)_i\Big),
	\end{align*}
	which is surjective by \cref{rem:rigidity}.
	This implies that the action of $\Ker\hF_{\bbC^*}$ on $\prod_{i=1}^g X^\circ_{0,h(v_i)}$ is transitive.
	Therefore, starting from any pre-log curve associated with $h\colon\hGamma\to S$, we can use the action of
	 $\Ker\hF_{\bbC^*}$ to obtain a pre-log curve satisfying the given constraints.
	Thus, the set $\PreLog'$ is non-empty.
	
	By \cite[Proposition 5.5]{Nishinou_Toric_2006}, the action of $\Ker\hF_{\bbC^*}$ on $\PreLog$ is transitive.
	Hence, the action of the subgroup $\Ker\hG_{\bbC^*}\subset\Ker\hF_{\Bbb C^*}$ on the subset $\PreLog'\subset\PreLog$ is also transitive.
	However, the action is not free due to the presence of 2-valent vertices in $\hGamma$ other than $v_i$.
	The stabilizer of any point in $\PreLog'$ is isomorphic to
	\[\bigoplus_{v\in V(\hGamma)\setminus V(\Gamma')} N_v\otimes\bbC^*,\]
	where $N_v\subset N$ denotes the sublattice generated by the weight vector of any edge connected to the 2-valent vertex $v$.
	By \cite[Proposition 2.18]{Tyomkin_Tropical_2012}, there is a short exact sequence
	\[ 0 \longrightarrow \bigoplus_{v \in V(\hGamma)\setminus V(\Gamma')} N_v\otimes \bbC^* 
	\longrightarrow \Ker \hG_{\bbC^*} \longrightarrow \Ker G_{\bbC^*} \longrightarrow 0.\]
	It follows that $\Ker G_{\bbC^*}$ acts freely and transitively on the set $\PreLog'$ of pre-log curves 
	satisfying the given constraints.
\end{proof}

\begin{lem}\label{lem:Psi}
	In the context of \cref{prop:cohomologies}, assume we have $s_i\in C_{0,v_i}$ for $i=1,\dots,g$ such that $\varphi_0(s_i)=x_i(0)$.
	Then, the map
	\[\Psi\colon H^0(C_0,\mathcal N_{\varphi_0})\longrightarrow\bigoplus_{i=1}^g \Theta_{X_0}(\varphi_0(s_i))/\varphi_{0*}(\Theta_{C_0}(s_i))\]
	given by evaluation of the sections of $\mathcal N_{\varphi_0}$ at the marked points $s_i$ is an isomorphism.
\end{lem}
\begin{proof}
	The range of $\Psi$ is isomorphic to $\bigoplus_{i=1}^g (N/N_{e_i})_\bbC$, where $e_i$ denotes the edge of $\Gamma$ containing $v_i$.
	By \cref{prop:cohomologies}, the domain of $\Psi$ is isomorphic to $\Ker F_\bbC$.
	So, we can identify the map $\Psi$ with
	\begin{align*}
	\Ker F_\bbC &\longrightarrow \bigoplus_{i=1}^g (N/N_{e_i})_\bbC\\
	\phi&\longmapsto\Big(\big(\phi(\partial^- e_i)\big)_i\Big).
	\end{align*}
	Therefore, $\Psi$ is an isomorphism by \cref{rem:rigidity}.	
\end{proof}

%For every $n\in\bbZ_{\ge 0}$, let $k_n$ denote the ring $\bbC[t]/(t^{n+1})$.

\begin{prop}\label{prop:smoothing_with_constraints}
	Let $\varphi_0\colon C_0\to X_0$ be a pre-log curve associated with the tropical curve $h\colon\hGamma\to S$, equipped with a log structure.
	Assume we have $s_i\in C_{0,v_i}$ for $i=1,\dots,g$ such that $\varphi_0(s_i)=x_i(0)$.
	Then, for every $n\in\bbZ_{\ge 0}$, there exists a unique pointed curve 
	\[
	\big(C_n,(s_1^n,\dots,s_g^n)\big)
	\]
	 over 
	$\bbC[t]/(t^{n+1})$ extending $\big(C_0,(s_1,\dots,s_g)\big)$, 
	and a $\bbC[t]/(t^{n+1})$-morphism 
	\[
	\varphi_n\colon C_n\to X_n
	\] 
	extending $\varphi_0\colon C_0\to X_0$, such that 
	\[
	\varphi_n\circ s_i^n=x_i
	\] %\circ\iota_n$ 
	for $i=1,\dots, g$. %, where $\iota_n$ denotes the embedding $\Spec k_n\to\Spec\bbC\llp t\rrp$.
\end{prop}
\begin{proof}
	The proposition holds for $n=0$ by assumption.
	Now, assuming the claim holds for $n$, let us prove it for $n+1$.
	By Theorem \ref{thm:main}, there exists a $\bbC[t]/(t^{n+2})$-curve $C_{n+1}$ extending $C_n$ and a $\bbC[t]/(t^{n+2})$-morphism 
	\[
	\varphi_{n+1}\colon C_{n+1}\to X_{n+1}
	\]
	 extending $\varphi_n\colon C_n\to X_n$.
	By deformation theory, once one extension exists, the set of all extensions is a torsor over $H^0(C_0,\mathcal N_{\varphi_0})$.
	Note that the deformations of a marked point $s_i$ from order $n$ to $n+1$ is a torsor over $\Theta_{C_0}(s_i)\simeq\bbC$.
	Hence, the set of all extensions $\varphi_{n+1}\colon C_{n+1}\to X_{n+1}$ together with the marked points is a torsor over 
	\[
	H^0(C_0,\mathcal N_{\varphi_0})\oplus\bigoplus_{i=1}^g\Theta_{C_0}(s_i).
	\]
	Consider the evaluation map
	\[
	\begin{array}{l}
	\Psi\oplus {\varphi_0}_*\colon H^0(C_0,\mathcal N_{\varphi_0})\oplus\bigoplus_{i=1}^g\Theta_{C_0}(s_i)\\
	\hspace{1in}\longrightarrow \bigoplus_{i=1}^g\big( \Theta_{X_0}(\varphi_0(s_i))/\varphi_{0*}(\Theta_{C_0}(s_i))\oplus \varphi_{0*}(\Theta_{C_0}(s_i))\big).
	\end{array}
	\]
	It is an isomorphism by \cref{lem:Psi}.
	Hence, there exists a unique extension $\varphi_{n+1}\colon C_{n+1}\to X_{n+1}$ together with marked points $s_i^{n+1}$ such that 
	\[
	\varphi_{n+1}\circ s_i^{n+1}=x_i
	\]
	 %\iota_{n+1}$
     for $i=1,\dots,g$, completing the proof.
\end{proof}

To conclude, \cref{prop:pre-log_torsor} shows that  there are exactly $\abs{\Ker G_{\bbC^*}}$ isomorphism classes of pre-log curves $\varphi_0\colon C_0\to X_0$ associated with the tropical curve $h\colon\hGamma\to S$ together with marked points $s_i\in C_{0,v_i}$ satisfying $\varphi_0(s_i)=x_i(0)$ for $i=1,\dots,g$.
For each such pre-log curve $\varphi_0\colon C_0\to X_0$, \cite[Proposition 7.1]{Nishinou_Toric_2006} 
 (see also \cite[Proposition 4.23]{Gross_Tropical_2011}) asserts that the number of isomorphism classes of
 log structures on $C_0$ is equal to 
\[
\prod_{e\in E(\Gamma')}W_e, 
\]
 the product of all the scalar edge weights of $\Gamma'$.
Once the log structure is fixed, \cref{prop:smoothing_with_constraints}, together with
 the implicit function theorem \cite[Proposition 1.5]{Kosarew_1991},
 guarantees a unique deformation to the generic fiber satisfying the constraints $x_i$.
Therefore, 
 the total number of parametrized holomorphic curves in $\mathcal X^{\times}$ associated with the realizable tropical curve $h\colon\Gamma\to S$ satisfying the constraints $x_i$ is equal to the product
\[
\abs{\Ker G_{\bbC^*}}\cdot\prod_{e\in E(\Gamma')} W_e.
\]
This completes the proof of \cref{thm:intro_multiplicity}.

%\linespread{1.0}
\bibliographystyle{myamsalpha}
\bibliography{dahema}

\end{document}